\documentclass[12pt,reqno]{amsart}
\usepackage[margin=1in]{geometry}
\usepackage{amsmath,amssymb,amsthm,graphicx,amsxtra, setspace}
\usepackage[utf8]{inputenc}
\usepackage{mathrsfs}
\usepackage{hyperref}
\usepackage{upgreek}
\usepackage{xcolor}
\usepackage{mathtools}
\allowdisplaybreaks

\colorlet{darkblue}{blue!50!black}
\hypersetup{
	colorlinks,%
	citecolor=red,%
	filecolor=red,%
	linkcolor=darkblue,%
	urlcolor=blue,%
	pdfnewwindow=true,%
	pdfstartview={FitH}
}
\colorlet{darkblue}{red!100!black}

\let\oldtocsection=\tocsection

\let\oldtocsubsection=\tocsubsection

\let\oldtocsubsubsection=\tocsubsubsection

\renewcommand{\tocsection}[2]{\hspace{0em}\oldtocsection{#1}{#2}}
\renewcommand{\tocsubsection}[2]{\hspace{1em}\oldtocsubsection{#1}{#2}}
\renewcommand{\tocsubsubsection}[2]{\hspace{2em}\oldtocsubsubsection{#1}{#2}}

\newtheorem{theorem}{Theorem}[section]
\newtheorem{lemma}[theorem]{Lemma}
\newtheorem{proposition}[theorem]{Proposition}

\newtheorem{corollary}[theorem]{Corollary}
\newtheorem{definition}[theorem]{Definition}

\newtheorem{remark}[theorem]{Remark}
\newtheorem{hypothesis}[theorem]{Hypothesis}

\let\originalleft\left
\let\originalright\right
\renewcommand{\left}{\mathopen{}\mathclose\bgroup\originalleft}
\renewcommand{\right}{\aftergroup\egroup\originalright}

\usepackage{cite}

\renewcommand{\d}{\/\mathrm{d}\/}

\def\e{\varepsilon}

\def\S{\mathcal{S}}

\def\L{\mathrm{L}}

\def\A{\mathrm{A}}

\def\2{\varsigma}
\def\1{\mathcal{O}}

\def\C{\mathrm{C}}
\def\c{\mathfrak{c}}
\def\ttt{\tau\wedge\tau_n}
\def\ee{\epsilon}
\def\loc{\mathrm{loc}}
\def\Ds{\mathbb{D}}

\def\B{\mathrm{B}}
\def\D{\mathrm{D}}

\def\E{\mathbb{E}}

\def\p{\mathfrak{p}}

\def\z{\zeta}
\def\v{\mathfrak{v}}

\def\W{\mathrm{W}}

\def\N{\mathbb{N}}

\def\R{\mathbb{R}}
\def\P{\mathbb{P}}

\def\H{\mathrm{H}}

\def\3{\varrho}

\newcommand{\Addresses}{{
		\footnote{
			\noindent \textsuperscript{1,2}Department of Mathematics, Indian Institute of Technology Roorkee-IIT Roorkee,
			Haridwar Highway, Roorkee, Uttarakhand 247667, INDIA.\par\nopagebreak
			\noindent  \textit{e-mail:} \texttt{Manil T. Mohan: maniltmohan@ma.iitr.ac.in, maniltmohan@gmail.com.}
			
			\textit{e-mail:} \texttt{Ankit Kumar: akumar14@mt.iitr.ac.in.}
			
			\noindent \textsuperscript{*}Corresponding author.
			
			\textit{Key words:} Stochastic generalized Burgers-Huxley equation,  Absolute continuity, Weak differentiability, Malliavin calculus.
			
			Mathematics Subject Classification (2020): Primary 60H15; Secondary 60H07; 60H20.

}}}

\begin{document}	
	
	\title[Existence of density for the solution of SGBH equation]{Absolute continuity of the solution to  stochastic generalized Burgers-Huxley equation
		\Addresses}
	
	\author[A. Kumar and M. T. Mohan]
	{Ankit Kumar\textsuperscript{1} and Manil T. Mohan\textsuperscript{2*}}

	\maketitle
	
	\begin{abstract}
		The present work deals with the global solvability as well as absolute continuity of the law of the solution to stochastic generalized Burgers-Huxley (SGBH) equation driven by multiplicative space-time white noise in a  bounded interval of $\mathbb{R}$. We first prove the existence of a unique local mild solution to SGBH equation with the help of a truncation argument and contraction mapping principle. Then global solvability results are obtained by using  uniform bounds of the local mild solution and stopping time arguments. Later, we establish a comparison theorem for the solution of SGBH equation having  higher order nonlinearities  and it plays a crucial role in this work. Then, we discuss the weak differentiability of the solution to SGBH equation in the Malliavin calculus sense.  Finally, we obtain the absolute continuity of the law of the solution with respect to the Lebesgue measure on $\mathbb{R}$,  and the existence of density with the aid of comparison theorem and weak differentiability of the solution. 
	\end{abstract}


	\section{Introduction}\label{sec1}\setcounter{equation}{0}
	The Burgers-Huxley equation is a special class of nonlinear advection-diffusion-reaction problems, which has variety of applications in material sciences, mechanical engineering, neurophysiology, etc. (cf.  dynamics of ferroelectric materials \cite{OYY}, action potential propagation in nerve fibers \cite{XWZZ},   particle transport \cite{JS}, wall motion in liquid crystals \cite{XYW}, etc.). 	The generalized Burgers-Huxley equation (GBH) equation describes a prototype model for describing the interaction between reaction mechanisms, convection effects and diffusion transports (cf. \cite{VJE,MTMAK}, etc.).  We consider the  following stochastic  generalized Burgers-Huxley (SGBH) equation perturbed by a random forcing:
	\begin{align}\label{1.01}
		\frac{\partial u(t,x)}{\partial t}&= \nu\frac{\partial^2 u(t,x)}{\partial x^2}-\alpha u^\delta(t,x)\frac{\partial u(t,x)}{\partial x}+\beta u(t,x)(1-u^\delta(t,x))(u^{\delta}(t,x)-\gamma) \nonumber\\&\quad+g(t,x,u(t,x))\frac{\partial^2\W(t,x) }{\partial x\partial t},\ \text{ for }\ (t,x)\in(0,T)\times(0,1),  \tag{1.1a}
	\end{align}	where $\nu>0$ is the viscosity coefficient, $\alpha>0$ is the advection coefficient, $\beta>0$, $\delta\geq 1$ and $\gamma \in (0,1)$ are parameters, with the Dirichlet boundary condition: 
\begin{align}\label{1.1b}\tag{1.1b}
	u(t,0)=u(t,1)=0, \; t\geq 0,
\end{align}
along with the initial condition 
\begin{align}\label{1.1c}\tag{1.1c}
u(0,x)=u_0(x). 
\end{align}
The noise coefficient $g$ is bounded and  satisfies the Lipschitz continuity in the last variable (Hypothesis \ref{H1}).  In the above equation, $\W(t,x),$ $t\geq 0, x\in\R$  is a zero mean Gaussian  process, whose covariance function is given by $$\mathbb{E}[ \W(t,x) \W(t,y)]= (t\wedge s)(x\wedge y), \ t,s\geq 0, \ x,y\in \R.$$   
In the deterministic case, for $\beta=0$ and $\delta=1$, the equation \eqref{1.01} is the classical viscous Burgers' equation  and for $\delta=1$, it is called the Burgers-Huxley equation. For the global solvability results of deterministic GBH equation and numerical studies of the model, one may refer to \cite{VJE,AKMT,MTMAK}, etc. and the references therein. 

The global solvability results for stochastic Burgers-Huxley equation perturbed by multiplicative correlated Gaussian noise is obtained in \cite{MTMSBH}.   The SGBH equation perturbed by space-time white noise (additive) is considered in the work \cite{MTMSGBH}, where the author established the existence of a  unique  mild solution using fixed point arguments. The authors in \cite{AKMTMSGBH} proved the existence of a unique mild as well as strong solution to  SGBH equation perturbed by additive Gaussian noise which is white in time and correlated in space. They have also proved the strong Feller and irreducibility properties of the  Markov semigroup associated with the solution,  which implies the uniqueness of invariant measure, and investigated the Donsker-Varadhan large deviation principle (LDP) for occupation measure. The papers discussed above do not cover the noise considered in \eqref{1.01}, and we need a different analysis to tackle the problem.  Therefore, we prove  the global existence and uniqueness of mild solution of the system \eqref{1.01}-\eqref{1.1c} utilizing the techniques available in \cite{IGDN}, where the authors considered  stochastic Burgers' equation perturbed by multiplicative space-time white noise on the real line. For the global solvability and analysis of stochastic Burgers' equation, the interested readers are referred to see \cite{DaZ,DDR,GDDG,IGDN,LZN}, etc. 
 In particular, the authors in \cite{LZN} discussed  the  existence  of a unique mild solution for stochastic Burgers' equation perturbed by a bounded, Lipschitz  nonlinearity  in a bounded interval. 
 
 The stochastic calculus of variations (or Malliavin calculus, \cite{ND}) has been used extensively to study the   existence of density for solutions of stochastic partial differential equations (SPDEs), see for  example \cite{ADMR,EPTZ,LZN,AJLDNRP,BFMZ,PLM,CMLQS,MSSAS}, etc.    The authors in \cite{EPTZ} established the absolute continuity of the solution of a  parabolic SPDE,  where the drift and diffusion  functions are assumed to be measurable and locally bounded with locally bounded derivatives. The absolute continuity of the law of the solution to stochastic Burgers' equation perturbed by space-time white noise is established in \cite{LZN}. The author in \cite{PLM} proved the existence of densities for a class of parabolic SPDEs of Burgers' type introduced in \cite{IG}.   For the  state-independent diffusion,  using the techniques of the Malliavin calculus,  the existence of a smooth density  for stochastic Burgers' equation driven by space-time white noise has been obtained in \cite{AJLDNRP}.  For a class of SPDEs like the stochastic wave and heat equations with multiplicative noise, Lipschitz coefficients, the authors in \cite{MSSAS} proved the existence of density of the law of the solution in some Besov space.   Absolute continuity of the law  of solution to a parabolic dissipative SPDE  of reaction-diffusion type perturbed by multiplicative Wiener noise in an open bounded domain in $\R^d$ with smooth boundary is established in \cite{CMLQS}. The existence of densities for stochastic differential equations (SDEs) perturbed by a stable-like L\'evy process under some non-degeneracy condition with H\"older continuous coefficients in some Besov spaces has been discussed in \cite{ADNF}.  The absolute continuity of the law of the  solution  of two stochastic fluid Lagrangian models  for viscous flow in two dimensions using Malliavin calculus is established in \cite{SSSMX}. The existence of densities for the law of finite-dimensional functional solutions of the three-dimensional stochastic Navier-Stokes equations (SNSE) has been studied in \cite{ADMR} (see \cite{MR} also). The existence and smoothness of density of solution of a nonlinear stochastic heat equation perturbed by a additive Wiener noise with the nonlinear drift term (having polynomial growth) in a bounded domain of $\R^d$ with smooth boundary has been obtained in \cite{CMENLQS}. In a recent paper \cite{COCT}, the authors proved the existence and the Besov regularity of the density of the solution for a class of parabolic SPDEs, which covers  stochastic Burgers' equation on an unbounded domain and the approach is based on the fractional integration by parts.  Recently, the absolute continuity of the law for the two-dimensional SNSE perturbed by Gaussian noise which is white in time and colored in space has been proved in \cite{BFMZ}.

Compared to stochastic Burgers' equation,   the model \eqref{1.01}-\eqref{1.1c} is  not explored much in the literature, and in this work, we shed some light on the stochastic analysis of this model. Note also that the model \eqref{1.01}-\eqref{1.1c}  contains convective, diffusive and reaction terms, which makes the analysis more challenging and interesting compared to other related models.    The three major objectives of this work are listed below: 
\begin{itemize}
	\item The first and foremost aim of this article is to establish the existence of a unique mild solution to the SGBH equation \eqref{1.01}. We follow the works \cite{LZN,MTMSGBH,IGDN}, etc. to achieve this goal. By considering a truncated system  and applying fixed point arguments, we first show the existence of a unique local mild solution to \eqref{1.01} with $u_0\in \L^p(0,1)$, for $p\geq 2\delta+1$. Later, we extend this solution to a global mild solution by using  uniform bounds of the local mild solution  and stopping time arguments. The uniform bounds of the local mild solution are obtained by showing the existence of a weak solution to a transformed system (see \eqref{316} below).  For $u_0\in \C([0,1])$, we show that the solution is regular by using the properties of the heat semigroup. 
	\item In the second part of the work, we prove a comparison theorem for the solutions of the SGBH equation \eqref{1.01}, that is, if $u_0(x)\leq v_0(x)$  for a.e. $x\in[0,1]$, then $u(t,x)\leq v(t,x)$ for all $t\in[0,T]$ and for a.e. $x\in[0,1]$, $\mathbb{P}$-a.s., where $u(t,x)$ and $v(t,x)$ are the unique mild  solutions of \eqref{1.01} with the initial data  $u_0,v_0$, respectively. We mainly borrow the ideas from the works \cite{IG,MP} to establish the comparison theorem. We should point out here that the drift nonlinearity appearing in the works \cite{IG,MP} is of quadratic order only. On the other hand, the nonlinearity of the drift of our model is of polynomial order and is of degree $\geq 2$. 
	\item Next, we  investigate the smoothness of the solution of the SGBH equation \eqref{1.01} in the sense of stochastic calculus of variations. For $u_0\in \C([0,1])$, our goal is to prove that the solution $u(t,x)$ of \eqref{1.01} is in the space $\Ds_{\loc}^{1,p},$ for $p> \max\{2\delta+1,6\}$ (in the sense of Malliavin calculus), for all $(t,x)\in [0,T]\times[0,1]$ (see section \ref{sec6} for  the definition of function spaces). In order to do this, we  assume that the coefficient $g$ of the noise is continuous and satisfies a non-degeneracy condition such that $g(0,y,u_0(y))\neq0$ for some $y\in(0,1)$. Under this assumption, we are able to prove that for any $(t,x)\in[0,T]\times[0,1]$, the law of the random variable $u(t,x)$ (solution of SGBH equation) is absolutely continuous with respect to the Lebesgue measure on $\mathbb{R}$ and hence the existence of density also by an application of the Radon-Nikodym theorem. Our result is based on the general criterion given in the work \cite{NBFH} and we follow the works \cite{LZN,EPTZ} to obtain this result. 
\end{itemize}

	The rest of the article is organized as follows. In the next section, we provide the necessary function spaces needed to obtain the global solvability results of the system \eqref{1.01}-\eqref{1.1c} as well as we prove several estimates on the operators appearing in the mild form of the solution of \eqref{1.01}-\eqref{1.1c}, which will be used in the proofs of comparison theorem and  existence of density results (Lemmas \ref{lemma3.1}, \ref{lemma3.2}, \ref{lemma3.3}). The existence and uniqueness of local as well as global mild solution of the SGBH equation \eqref{1.01} is discussed in section \ref{sec4} (Propositions \ref{theorem4.2}, \ref{lem3.4} and Theorem \ref{theorem4.3}). In section \ref{sec5}, we  prove a comparison theorem, which plays an important role in the proof of existence of density for our model (Theorem \ref{theorem5.2}).	
		We establish some technical results  to obtain the existence of density (Lemmas  \ref{lemma6.6}, \ref{lemma6.7} and Proposition \ref{prop6.5})  in section \ref{sec6}. We wind up the article  by  proving the absolute continuity  of the law of the solution to SGBH equation with respect to the Lebesgue measure on $\mathbb{R}$ and hence the existence of density in section \ref{sec7} (Theorems \ref{theorem7.1}, \ref{thm7.3}).
	
		\section{Mathematical Formulation}\label{sec2}\setcounter{equation}{0}
	This section provides the necessary function spaces needed to obtain the main  results of this paper. 
	\subsection{Function spaces} Let us fix $\mathcal{O}=(0,1)$. Let $\C_{0}^{\infty}(\1)$ denote the space of all infinite times differentiable functions having compact support in $\1$. The Lebesgue spaces are denoted by $\L^{p}(\1)$ for $p\in [1,\infty]$, and the norm in $\L^p(\1)$ is denoted by $\|\cdot\|_{\L^p}$ and for $p=2$, the inner product in $\L^2(\1)$ is denoted by $(\cdot,\cdot)$. We denote the Sobolev spaces by $\H^{k}(\1),$ for $k\in\mathbb{N}$. Let $\H_0^1(\1)$ denote the closure of $\C_{0}^{\infty}(\1)$ in $\H^1$-norm. As we are working in a bounded domain, by using Poincar\'e's inequality (see \eqref{poin} below), we infer that the norm $(\|\cdot\|_{\L^2}^2+\|\partial_{x}\cdot\|_{\L^2}^2)^{\frac{1}{2}}$ is equivalent to the seminorm $\|\partial_{x}\cdot\|_{\L^2}$ and hence $\|\partial_{x}\cdot\|_{\L^2}$ defines a norm on $\H_0^1(\1)$. We also have the continuous embedding $\H_0^1(\1) \subset \L^2(\1) \subset \H^{-1}(\1)$, where $\H^{-1}(\1)$ is the dual space of $\H_0^1(\1)$. For the bounded domain $\1,$ the embedding  $\H_0^1(\1) \subset \L^2(\1)$ is compact. The duality pairing between $\H_0^1(\1)$ and its dual $\H^{-1}(\1)$ ({based on the inner product in the Hilbert space $\L^2(\1)$}), and $\L^p(\mathcal{O})$ and its dual $\mathrm{L}^{\frac{p}{p-1}}(\1)$ are   denoted by $\langle \cdot, \cdot \rangle$. 	
	In one dimension,  the embedding of $\H^{\sigma}(\1) \subset \L^p(\1)$ is compact for any $\sigma> \frac{1}{2}-\frac{1}{p}$, for $p\geq 2$. 

	For every $p\geq 1$ and $p>1+\e$, we denote by {$\W^{\frac{1+\e}{p},p}([0,1];H)$} (fractional Sobolev space) as  the set of continuous functions $f:[0,1]\to  H$ such that \begin{align*}
		\|f\|_{\frac{1+\e}{p},p,H}^p = \int_{0}^{1}\int_{0}^{1}\frac{\|f(x)-f(y)\|_{H}^p}{|x-y|^{2+\e}}\d x\d y<\infty,
	\end{align*}where $H$ denotes a real separable Hilbert space with the norm $\|\cdot\|_H$ and we will omit the subindex $H$ whenever $H=\R$ in the sequel. Also, from  \cite{GRR} for any $p\geq 1$ and $\e>0$, we infer that 
\begin{align}\label{{2.01}}
	\|f\|_{\L^{\infty} (\mathcal{O};H)}\leq \|f(0)\|_H+C_{\e,p}\|f\|_{\frac{1+\e}{p},p,H}.
\end{align}

	\subsection{Operators and their properties}\label{sec3}\setcounter{equation}{0}
In this subsection, we provide some operators and technical results which play a crucial role in the proofs of solution of the equation \eqref{4.01} established  in section \ref{sec4} and the comparison theorem obtained  in section \ref{sec5}. Before defining the operators, we introduce the the fundamental solution $G(t,x,y)$ of the  heat equation in the interval $[0,1]$ with the Dirichlet boundary conditions, defined by
\begin{align*}
	G(t,x,y)=\frac{1}{\sqrt{4\pi t}}\sum_{n=-\infty}^{\infty}\bigg[e^{-\frac{(y-x-2n)^2}{4t}}-e^{-\frac{(y+x-2n)^2}{4t}}\bigg],
\end{align*}for all $0<t\leq T$ and $x,y\in[0,1]$.
We use the following estimates frequently in the sequel (cf. \cite{LZN, IG, IGDN}):
\begin{align}\label{A1}{\tag{A1}}
	|G(t,x,y)| \leq Kt^{-\frac{1}{2}}e^{-\frac{|x-y|^2}{\ell_1t}},
\end{align}
\begin{align}\label{A2}{\tag{A2}}
	\bigg|\frac{\partial G}{\partial y}(t,x,y)\bigg| &\leq Kt^{-1}e^{-\frac{|x-y|^2}{\ell_2t}},
\end{align}

\begin{align}\label{A3}{\tag{A3}}
	\bigg|\frac{\partial G}{\partial t}(t,x,y)\bigg| &\leq Kt^{-\frac{3}{2}}e^{-\frac{|x-y|^2}{\ell_3t}},
\end{align}
\begin{align}\label{A4}{\tag{A4}}
	\bigg|\frac{\partial^2 G}{\partial y\partial t}(t,x,y)\bigg| &\leq Kt^{-2}e^{-\frac{|x-y|^2}{\ell_4t}},
\end{align}
\begin{align}\label{A5}{\tag{A5}}
	|G(t,x,z)-G(t,y,z)| &\leq K|x-y|^{\vartheta}t^{-\frac{\vartheta}{2}-\frac{1}{2}}\max\bigg\{e^{-\frac{|x-z|^2}{\ell_5t}},e^{-\frac{|y-z|^2}{\ell_5t}}\bigg\},
\end{align}
\begin{align}\label{A6}{\tag{A6}}
	\bigg|\frac{\partial G}{\partial z}(t,x,z)-\frac{\partial G}{\partial z}(t,y,z)\bigg| \leq K|x-y|^{\vartheta}t^{-1-\frac{\vartheta}{2}}\max\bigg\{e^{-\frac{|x-y|^2}{\ell_6t}},e^{-\frac{|x-z|^2}{\ell_6t}}\bigg\},
\end{align}for all $0<t\leq T, \ x,y,z\in[0,1]$, $K,\;\ell_i$ for $i=1,\ldots,6$, are some positive constants and $\vartheta\in[0,1]$.  The following estimate (see \cite{IG}) has been used frequently in the sequel: \begin{align*}\label{A7}\tag{A7}
	\big\|e^{-\frac{|\cdot|^2}{\ell(t-s)}}\big\|_{\L^p} \leq C (t-s)^{\frac{1}{2p}},
\end{align*} for any positive constant $\ell$ and $p\geq 1$. Let us define the operators $J_1$ and $J_2$ by 
\begin{align*}
J_1(v)(t,x)&:=\int_{0}^{t}\int_{0}^{1}G(t-s,x,y)\big((1+\gamma)v^{\delta+1}-\gamma v-v^{2\delta+1}\big)(s,y)\d y\d s\\&=:((1+\gamma)J_{11}-\gamma J_{12}-J_{13})v(t,x), \\
J_2(v)(t,x)&:=\int_{0}^{t}\int_{0}^{1}\frac{\partial G}{\partial y}(t-s,x,y)v^{\delta+1}(s,y)\d y \d s,
\end{align*} for all $t\in[0,T]$, $x,y\in[0,1]$, where $v$ is a function in $\L^{\infty}(0,T;\L^p(\1))$ for some $p\geq 1$. 

\begin{lemma}\label{lemma3.1}
	Under $\eqref{A1}-\eqref{A7}$, we have 
	\begin{enumerate}
		\item[(a)] $J_{11}$ is a bounded operator from $\L^{\eta_1}(0,T;\L^p(\1))$ into $\C([0,T];\L^p(\1))$ for $p\geq \delta+1$ and $\eta_1>\frac{2p(\delta+1)}{2p-\delta}$.
		\item[(b)] $J_{12}$ is a bounded operator from $\L^{\eta_2}(0,T;\L^p(\1))$ into $\C([0,T];\L^p(\1))$ for $p\geq 1$ and  $\eta_2> 1$.
		\item[(c)] $J_{13}$ is a bounded operator from $\L^{\eta_3}(0,T;\L^p(\1))$ into $\C([0,T];\L^p(\1))$ for $p\geq 2\delta+1$  and $\eta_3>\frac{p(2\delta+1)}{p-\delta}$.
	\end{enumerate}Moreover, for the values of $\eta_i$, $i=1,2,3,$ given above, the following estimates hold: \begin{enumerate}

		\item For every $0\leq t\leq T$, there is a constants $C$ such that 
		\begin{align*}
			\|J_{11}v(t,\cdot)\|_{\L^p}	&\leq C\int_{0}^{t}(t-s)^{-\frac{\delta}{2p}}\|v(s)\|_{\L^p}^{\delta+1}\d s\leq Ct^{-\frac{\delta}{2p}+1-\frac{\delta+1}{\eta_1}}\bigg(\int_{0}^{t}\|v(s)\|_{\L^p}^{\eta_1}\d s\bigg)^{\frac{\delta+1}{\eta_1}}  , \\
			\|J_{12}v(t,\cdot)\|_{\L^p}	&\leq C\int_{0}^{t}\|v(s)\|_{\L^p}\d s\leq Ct^{1-\frac{1}{\eta_2}}\bigg(\int_{0}^{t}\|v(s)\|_{\L^p}^{\eta_2}\d s\bigg)^{\eta_2} , \\
			\|J_{13}v(t,\cdot)\|_{\L^p}	&\leq C\int_{0}^{t}(t-s)^{-\frac{\delta}{p}}\|v(s)\|_{\L^p}^{2\delta+1}\d s \leq Ct^{-\frac{\delta}{p}+1-\frac{2\delta+1}{\eta_3}}\bigg(\int_{0}^{t}\|v(s)\|_{\L^p}^{\eta_3}\d s\bigg)^{\frac{2\delta+1}{\eta_3}}.
		\end{align*} 
		\item For $0{\;\leq s}\leq t \leq T$, $0<\vartheta_1<1-\frac{2p(\delta+1)+\delta\eta_1}{2p\eta_1}$, $0<\vartheta_2<1-\frac{1}{\eta_2}$ and $0<\vartheta_3<1-\frac{p(2\delta+1)+\delta\eta_3}{p\eta_3}$, there is a constant $C$ such that 
		\begin{align*}
			\|J_{11}v(t,\cdot)-J_{11}v(s,\cdot)\|_{\L^p}& \leq 
			C(t-s)^{\vartheta_1}\bigg(\int_{0}^{t}\|v(s)\|_{\L^p}^{\eta_1}\d s\bigg)^{\frac{\delta+1}{\eta_1}}, \\
			\|J_{12}v(t,\cdot)-J_{12}v(s,\cdot)\|_{\L^p}& \leq 
			C(t-s)^{\vartheta_2}\bigg(\int_{0}^{t}\|v(s)\|_{\L^p}^{\eta_2}\d s\bigg)^{\frac{1}{\eta_2}},
			\\ 	\|J_{13}v(t,\cdot)-J_{13}v(s,\cdot)\|_{\L^p}& \leq 
			C(t-s)^{\vartheta_3}\bigg(\int_{0}^{t}\|v(s)\|_{\L^p}^{\eta_3}\d s\bigg)^{\frac{2\delta+1}{\eta_3}}.
		\end{align*}
		\item For every $0\leq t \leq T$, $0<\varrho_4<2\big(1-\frac{2p(\delta+1)+\delta\eta_1}{2p\eta_1}\big)$, $0<\varrho_5<2\big(1-\frac{1}{\eta_2}\big)$ and $0<\varrho_6<2\big(1-\frac{p(2\delta+1)+\delta\eta_3}{p\eta_3}\big)$, there is a constant $C$ such that \begin{align*}
			\|J_{11}v(t,\cdot)-J_{11}v(t,\cdot+z)\|_{\L^p}&\leq C|z|^{\varrho_4}t^{-\frac{\delta}{2p}-\frac{\varrho_4}{2}+1-\frac{\delta+1}{\eta_1}}\bigg(\int_{0}^{t}\|v(s)\|_{\L^p}^{\eta_1}\d s\bigg)^{\frac{\delta+1}{\eta_1}} ,
			\\
			\|J_{12}v(t,\cdot)-J_{12}v(t,\cdot+z)\|_{\L^p}&\leq C|z|^{\varrho_5}t^{-\frac{\varrho_5}{2}+1-\frac{1}{\eta_2}}\bigg(\int_{0}^{t}\|v(s)\|_{\L^p}^{\eta_2}\d s\bigg)^{\frac{1}{\eta_2}} ,
			\\
			\|J_{13}v(t,\cdot)-J_{13}v(t,\cdot+z)\|_{\L^p} & \leq C|z|^{\varrho_6}t^{-\frac{\delta}{p}-\frac{\varrho_6}{2}+1-\frac{2\delta+1}{\eta_3}}\bigg(\int_{0}^{t}\|v(s)\|_{\L^p}^{\eta_3}\d s\bigg)^{\frac{2\delta+1}{\eta_3}},
		\end{align*}for all $z\in\R$.  We also set $J_{1i}v(t,y):=0, \text{ for } i=\{1,2,3\}$, whenever $y\in\R\backslash[0,1]$.
	\end{enumerate}
\end{lemma}
\begin{proof}
	Here we discuss the estimates for $J_{13}v(t,\cdot)$ only.  Using similar ideas, one can establish the estimates  for $J_{11}v(t,\cdot)$ and $J_{12}v(t,\cdot)$.
	Using Minkowski's inequality, the estimates \eqref{A1}, \eqref{A7}, Young's inequality for convolution and H\"older's inequality, we obtain 
	\begin{align*}
		\|J_{13}v(t,\cdot)\|_{\L^p}&=\bigg(\int_{0}^{1}\bigg|\int_{0}^{t}\int_{0}^{1}G(t-s,x,y)v^{2\delta+1}(s,y)\d y\d s\bigg|^p\d x\bigg)^{\frac{1}{p}} \\& 
		\leq C\int_{0}^{t}(t-s)^{-\frac{1}{2}}\big\|e^{-\frac{|\cdot|^2}{\ell_1(t-s)}}*|v^{2\delta+1}(s,\cdot)|\big\|_{\L^p}\d s \\& \leq C\int_{0}^{t}(t-s)^{{-\frac{\delta}{p}}}\|v(s)\|_{\L^p}^{2\delta+1}\d s 
		\\&\leq C\bigg(\int_{0}^{t}(t-s)^{-\frac{\delta\eta_3}{p(\eta_3-2\delta-1)}}\d s\bigg)^{1-\frac{2\delta+1}{\eta_3}}\bigg(\int_{0}^{t}\|v(s)\|_{\L^p}^{\eta_3}\d s\bigg)^{\frac{2\delta+1}{\eta_3}} \\& \leq Ct^{-\frac{\delta}{p}+1-\frac{2\delta+1}{\eta_3}}\bigg(\int_{0}^{t}\|v(s)\|_{\L^p}^{\eta_3}\d s\bigg)^{\frac{2\delta+1}{\eta_3}}, 
	\end{align*}for all $0\leq t\leq T$, where $p\geq 2\delta+1$ and $\eta_3>\frac{p(2\delta+1)}{p-\delta}$. Now, we consider 
	\begin{align*}
		\|J_{13}v(t,\cdot)-J_{13}v(s,\cdot)\|_{\L^p} \leq M+N, 
	\end{align*}where \begin{align*}
		M&=\bigg\|\int_{s}^{t}\big(G(t-r,0,\cdot)*v^{2\delta+1}(r,\cdot)\big)\d r\bigg\|_{\L^p}, 
	\end{align*}and \begin{align*}
		N=\bigg\|\int_{0}^{s}\big[\big(G(t-r,0,\cdot)-G(s-r,0,\cdot)\big)*v^{2\delta+1}(r,\cdot)\big]\d r\bigg\|_{\L^p}.
	\end{align*} 
Using the first part of the Lemma, one can estimate $M$ as
	\begin{align*}
		M\leq C\int_{s}^{t}(t-r)^{{-\frac{\delta}{p}}}\|v(r)\|_{\L^p}^{2\delta+1}\d r. 
	\end{align*}For $N$, we use the estimates \eqref{A3}, \eqref{A7} and Young's and H\"older's inequalities, for $\vartheta_3+\varrho_3=1$, to get 
	\begin{align*}
		N&=\left\|\int_0^s\int_s^t\frac{\partial G(\theta-r,0,\cdot)}{\partial t}*v^{2\delta+1}(r,\cdot)\d\theta\d r\right\|_{\L^p}\\&\leq C\int_{0}^{s}\int_{s}^{t}(\theta-r)^{-\frac{3}{2}}\big\|e^{-\frac{|\cdot|^2}{\ell_3(\theta-r)}}*|v^{2\delta+1}(r,\cdot)|\big\|_{\L^p}\d \theta\d r \\& \leq C\int_{0}^{s}\int_{s}^{t}(\theta-r)^{-\frac{\delta}{p}-1}\|v(r)\|_{\L^p}^{2\delta+1}\d \theta \d r\\& \leq C(t-s)^{\vartheta_3}\int_{0}^{s}\bigg(\int_{s}^{t}(\theta-r)^{-\frac{1}{\varrho_3}\big(\frac{\delta}{p}+1\big)}\d \theta\bigg)^{\varrho_3}\|v(r)\|_{\L^p}^{2\delta+1}\d r\\& \leq 
		C(t-s)^{\vartheta_3}\int_{0}^{s}(s-r)^{-\frac{\delta}{p}-1+\varrho_3}\|v(r)\|_{\L^p}^{2\delta+1}\d r	\\& \leq C(t-s)^{\vartheta_3}\bigg(\int_{0}^{s}(s-r)^{\frac{\eta_3}{\eta_3-2\delta-1}\big(-\frac{\delta}{p}-1+\varrho_3\big)}\d r\bigg)^{1-\frac{2\delta+1}{\eta_3}}\bigg(\int_{0}^{s}\|v(r)\|_{\L^p}^{\eta_3}\d r\bigg)^{\frac{2\delta+1}{\eta_3}} \\&\leq 	C(t-s)^{\vartheta_3}\bigg(\int_{0}^{s}\|v(r)\|_{\L^p}^{\eta_3}\d r\bigg)^{\frac{2\delta+1}{\eta_3}} 
		\\&\leq 	C(t-s)^{\vartheta_3}\bigg(\int_{0}^{t}\|v(s)\|_{\L^p}^{\eta_3}\d s\bigg)^{\frac{2\delta+1}{\eta_3}} ,	
	\end{align*}for all $0\leq s\leq t\leq T$, where $0<\vartheta_3<1-\frac{p(2\delta+1)+\delta\eta_3}{p\eta_3}$. Next, we consider \begin{align*}
		\|J_{13}v(t,\cdot)-J_{13}v(t,\cdot+z)\|_{\L^p} \leq P+Q, 
	\end{align*}
	where \begin{align*}
		P &= \int_{0}^{t}\bigg\{\int_{0}^{1}\chi_{\{x+z\in[0,1]\}}\bigg|\int_{0}^{1}\chi_{\{|x-y|\leq |z|\}}|G(t-s,x+z,y)-G(t-s,x,y)|\\& \qquad \times v^{2\delta+1}(s,y)\d y\bigg|^p\d x\bigg\}^{\frac{1}{p}}\d s,
	\end{align*}and 
	\begin{align*}
		Q&= \int_{0}^{t}\bigg\{\int_{0}^{1}\chi_{\{x+z\in[0,1]\}}\bigg|\int_{0}^{1}\chi_{\{|x-y|> |z|\}}|G(t-s,x+z,y)-G(t-s,x,y)|\\& \qquad \times v^{2\delta+1}(s,y)\d y\bigg|^p\d x\bigg\}^{\frac{1}{p}}\d s,
	\end{align*}where $\chi_{A}$ denotes the characteristic function of the set $A$. Using Minkowski's inequality, the estimates \eqref{A5}, \eqref{A7}, Young's and H\"older's inequalities, we estimate $P$ as
	\begin{align*}
		P &\leq C|z|^{\varrho_6}\int_{0}^{t}\bigg\{\int_{0}^{1}\bigg|\int_{0}^{1} (t-s)^{-\frac{1}{2}-\frac{\varrho_6}{2}}\chi_{\{|x-y|\leq |z|\}}\max\left\{e^{-\frac{|x+z-y|^2}{\ell_5(t-s)}},e^{-\frac{|x-y|^2}{\ell_5(t-s)}}\right\}\\&\qquad \times v^{2\delta+1}(s,y)\d y\bigg|^p\d x \bigg\}^{\frac{1}{p}}\d s \\& \leq C|z|^{\varrho_6}\int_{0}^{t}(t-s)^{-\frac{\delta}{p}-\frac{\varrho_6}{2}}\|v(s)\|_{\L^p}^{2\delta+1}\d s \\& \leq C|z|^{\varrho_6}t^{-\frac{\delta}{p}+\frac{\varrho_6}{2}+1-\frac{2\delta+1}{\eta_3}}\bigg(\int_{0}^{t}\|v(s)\|_{\L^p}^{\eta_3}\d s\bigg)^{\frac{2\delta+1}{\eta_3}},
	\end{align*}for all $t\in[0,T],\; z\in\R$, where $0<\varrho_6<2\big(1-\frac{p(2\delta+1)+\delta\eta_3}{p\eta_3}\big)$.
	Applying  Minkowski's inequality,  the estimates \eqref{A5}, \eqref{A7}  and Young's inequality  on $Q$, we find 
	\begin{align*}
		Q &\leq C|z|^{\varrho_6}\int_{0}^{t}\bigg\{\int_{0}^{1}\bigg|\int_{0}^{1} (t-s)^{-\frac{1}{2}-\frac{\varrho_6}{2}}\chi_{\{|x-y|> |z|\}}\max\left\{e^{-\frac{|x+z-y|^2}{\ell_5(t-s)}},e^{-\frac{|x-y|^2}{\ell_5(t-s)}}\right\}\\&\qquad\times v^{2\delta+1}(s,y)\d y\bigg|^p\d x \bigg\}^{\frac{1}{p}}\d s  \\& \leq 
		C|z|^{\varrho_6}\int_{0}^{t}(t-s)^{-\frac{(\varrho_6+1)}{2}}\bigg\{\int_{0}^{1}\bigg|\int_{0}^{1}\chi_{\{|x-y|>|z|\}}e^{-\frac{|\theta-y|^2}{\ell_5(t-s)}}v^{2\delta+1}(s,y)\d y\bigg|^p\d x\bigg\}^{\frac{1}{p}}\d s, 
	\end{align*} where $\theta$ lies between $x$ and $x+z$. Whenever $|x-y|\geq |z|$ and $\theta$ lies  between $x$ to $x+z$,  the value of $|\theta-y|$ is either $\geq \frac{1}{\sqrt{2}}|x-z-y|$ or $\frac{1}{\sqrt{2}}|x+z-y|$. Applying Young's  and H\"older's inequalities in the above estimate, we obtain 
	\begin{align*}
		Q &\leq C|z|^{\varrho_6}\int_{0}^{t}(t-s)^{-\frac{\delta}{p}-\frac{\varrho_6}{2}}\|v(s)\|_{\L^p}^{2\delta+1}\d s \\& \leq C|z|^{\varrho_6}t^{-\frac{\delta}{p}+\frac{\varrho_6}{2}+1-\frac{2\delta+1}{\eta_3}}\bigg(\int_{0}^{t}\|v(s)\|_{\L^p}^{\eta_3}\d s\bigg)^{\frac{2\delta+1}{\eta_3}},
	\end{align*} for all $t\in[0,T],\; z\in\R$, where $0<\varrho_6<2\big(1-\frac{p(2\delta+1)+\delta\eta_3}{p\eta_3}\big)$.
	
	Proofs of the other estimates   follow on the similar lines as we did for $J_{13}v(t,\cdot)$.  To complete the proof of this Lemma, we choose $\eta=\max\{\eta_1,\eta_2,\eta_3\}$.
\end{proof}
\begin{lemma}\label{lemma3.2}
	Under  $\eqref{A2}-\eqref{A7}$, for $p\geq \delta+1$ and $\eta>\frac{2p(\delta+1)}{p-\delta}$, the operator $J_2$ is bounded from $\L^\eta(0,T;\L^p(\1))$ into $\C([0,T];\L^p(\1))$. Moreover the following estimates hold for $p\geq  \delta+1$:
	\begin{enumerate}
		\item For every $0\leq t\leq T$, there is a constant $C$ such that 
		\begin{align*}
			\|J_2v(t,\cdot)\|_{\L^{p}}&\leq C\int_{0}^{t}(t-s)^{-\frac{\delta}{2p}-\frac{1}{2}}\|v(s)\|_{\L^p}^{\delta+1}\d r\\&\leq Ct^{-\frac{\delta}{2p}+\frac{1}{2}-\frac{\delta+1}{\eta}}\bigg(\int_{0}^{t}\|v(s)\|_{\L^p}^\eta\d s\bigg)^{\frac{\delta+1}{\eta}}.
		\end{align*}
		\item For $0{\;\leq s}\leq t\leq T$, $0<\vartheta<\frac{1}{2}\big(1-\frac{2p(\delta+1)+\delta\eta}{p\eta}\big)$, there is a constant $C$ such that
		\begin{align*}
			\|J_2v(t,\cdot)-J_2v(s,\cdot)\|_{\L^{p}} \leq 	C(t-s)^{\vartheta}\bigg(\int_{0}^{t}\|v(s)\|_{\L^p}^{\eta}	\d s\bigg)^{\frac{\delta+1}{\eta}}.
		\end{align*}
		\item For every $0\leq t\leq T,\; 0<\varrho_7<1-\frac{2p(\delta+1)+\delta\eta}{p\eta}$, there is a constant $C$ such that 
		\begin{align*}
			\|J_2v(t,\cdot)-J_2v(t,\cdot+z)\|_{\L^{p}}\leq C|z|^{\varrho_7}t^{-\frac{\delta}{2p}-\frac{\varrho_7}{2}+\frac{1}{2}-\frac{\delta+1}{\eta}}\bigg(\int_{0}^{t}\|v(s)\|_{\L^p}^{\eta}\d s\bigg)^{\frac{\delta+1}{\eta}},
		\end{align*}for all $z\in\R$. We also set $J_2v(t,y):=0$, whenever $y\in\R\backslash[0,1]$.
	\end{enumerate}
\end{lemma}
\begin{proof}
	Using Minkowski's inequality, the estimates \eqref{A2}, \eqref{A7}, and Young's and H\"older's  inequalities, we obtain  
	\begin{align*}
		\|J_2v(t,\cdot)\|_{\L^p}&= \bigg(\int_{0}^{1}\bigg|\int_{0}^{t}\int_{0}^{1}\frac{\partial G}{\partial y}(t-s,x,y)v^{\delta+1}\d y \d s\bigg|^p\d x\bigg)^{\frac{1}p} \\& \leq 
		C\int_{0}^{t}(t-s)^{-1}\big\|e^{-\frac{|\cdot|^2}{\ell_2(t-s)}}*|v^{\delta+1}(s,\cdot)|\big\|_{\L^p}\d s \\&\leq C\int_{0}^{t}(t-s)^{-\frac{\delta}{2p}-\frac{1}{2}}\|v(s)\|_{\L^p}^{\delta+1}\d s \\&\leq C\bigg(\int_{0}^{t}(t-s)^{-\frac{\eta}{\eta-\delta-1}\big(\frac{\delta}{2p}+\frac{1}{2}\big)}\d s\bigg)^{1-\frac{\delta+1}{\eta}}\bigg(\int_{0}^{t}\|v(s)\|_{\L^p}^\eta\d s\bigg)^{\frac{\delta+1}{\eta}} \\&\leq Ct^{-\frac{\delta}{2p}+\frac{1}{2}-\frac{\delta+1}{\eta}}\bigg(\int_{0}^{t}\|v(s)\|_{\L^p}^\eta\d s\bigg)^{\frac{\delta+1}{\eta}},
	\end{align*}for $0\leq t\leq T$, where $p\geq \delta+1$ and $\eta>\frac{2p(\delta+1)}{p-\delta}$. Let us now consider 
	\begin{align*}
		\|J_2v(t,\cdot)-J_2v(s,\cdot)\|_{\L^p}\leq M+N, 
	\end{align*}where \begin{align*}
		M=\bigg\|\int_{s}^{t}\int_{0}^{1}\frac{\partial G}{\partial y}(t-r,\cdot,y)v^{\delta+1}(r,\cdot)\d y\d r\bigg\|_{\L^p},
	\end{align*}and \begin{align*}
		N= \bigg\|\int_{0}^{s}\int_{0}^{1}\bigg(\frac{\partial G}{\partial y}(t-r,\cdot,y)-\frac{\partial G}{\partial y}(s-r,\cdot,y)\bigg)v^{\delta+1}(r,\cdot)\d y\d r\bigg\|_{\L^p}.
	\end{align*}Using the part (1) of the Lemma, we estimate the term $M$ as
	\begin{align*}
		M\leq C\int_{s}^{t}(t-r)^{-\frac{\delta}{2p}-\frac{1}{2}}\|v(r)\|_{\L^p}^{\delta+1}\d r 
		 , \text{ for } p\geq \delta+1.
	\end{align*} For the term $N,$ using the estimates \eqref{A4}, \eqref{A7}, Young's and H\"older's inequalities, for $\vartheta+\varrho=1,$ we get
	\begin{align*}
		N&\leq \int_{0}^{s}\bigg\|\int_{0}^{1}\int_s^t\frac{\partial ^2G}{\partial t\partial y}(\theta-r,\cdot,y)v^{\delta+1}(r,\cdot)\d\theta\d y\bigg\|_{\L^p}\d r\\&\leq C\int_{0}^{s}\bigg\|\int_{0}^{1}\bigg(\int_{s}^{t}(\theta-r)^{-2}e^{-\frac{|\cdot|^2}{\ell_4(\theta-r)}}\d \theta\bigg)v^{\delta+1}(r,\cdot)\d y\bigg\|_{\L^p}\d r \\& \leq C\int_{0}^{s}\int_{s}^{t}(\theta-r)^{-2}\big\|e^{-\frac{|\cdot|^2}{\ell_4(\theta-r)}}*|v^{\delta+1}(r,\cdot)|\big\|_{\L^p}\d \theta \d r \\&  \leq C\int_{0}^{s}\int_{s}^{t}(\theta-r)^{-\frac{\delta}{2p}-\frac{3}{2}}\|v(r)\|_{\L^p}^{\delta+1}\d \theta \d r \\& \leq C(t-s)^{\vartheta}\int_{0}^{s}\bigg(\int_{s}^{t}(\theta-r)^{-\frac{1}{\varrho}\big(\frac{\delta}{2p}+\frac{3}{2}\big)}\d \theta\bigg)^{\varrho}\|v(r)\|_{\L^p}^{\delta+1}\d r \\& \leq 	C(t-s)^{\vartheta}\int_{0}^{s}(s-r)^{-\frac{\delta}{2p}-\frac{3}{2}+\varrho}\|v(r)\|_{\L^p}^{\delta+1}	\d r \\& \leq C(t-s)^{\vartheta}\bigg(\int_{0}^{s}(s-r)^{\frac{\eta}{\eta-\delta-1}\big(-\frac{\delta}{2p}-\frac{3}{2}+\varrho\big)}\d r\bigg)^{1-\frac{\delta+1}{\eta}}\bigg(\int_{0}^{s}\|v(r)\|_{\L^p}^{\eta}\d r\bigg)^{\frac{\delta+1}{\eta}} \\&  \leq C(t-s)^{\vartheta}\bigg(\int_{0}^{s}\|v(r)\|_{\L^p}^{\eta}\d r\bigg)^{\frac{\delta+1}{\eta}}
		 \\&  \leq C(t-s)^{\vartheta}\bigg(\int_{0}^{t}\|v(s)\|_{\L^p}^{\eta}\d s\bigg)^{\frac{\delta+1}{\eta}},
	\end{align*}for all $0\leq s\leq t\leq T$, where $0<\vartheta<\frac{1}{2}\big(1-\frac{2p(\delta+1)+\delta\eta}{p\eta}\big)$. Next, we consider 
	\begin{align*}
		\|J_2v(t,\cdot)-J_2v(t,\cdot+z)\|_{\L^p}\leq P+Q,
	\end{align*}where \begin{align*}
		P &= \int_{0}^{t}\bigg\{\int_{0}^{1}\chi_{\{x+z\in[0,1]\}}\bigg|\int_{0}^{1}\chi_{\{|x-y|\leq |z|\}}\bigg(\frac{\partial G}{\partial y}(t-s,x+z,y)-\frac{\partial G}{\partial y}(t-s,x,y)\bigg)\\& \qquad \times v^{\delta+1}(s,y)\d y\bigg|^p\d x\bigg\}^{\frac{1}p}\d s,
	\end{align*}and 
	\begin{align*}
		Q&= \int_{0}^{t}\bigg\{\int_{0}^{1}\chi_{\{x+z\in[0,1]\}}\bigg|\int_{0}^{1}\chi_{\{|x-y|> |z|\}}\bigg(\frac{\partial G}{\partial y}(t-s,x+z,y)-\frac{\partial G}{\partial y}(t-s,x,y)\bigg)\\& \qquad \times v^{\delta+1}(s,y)\d y\bigg|^p\d x\bigg\}^{\frac{1}p}\d s.
	\end{align*}Using Minkowski's inequality, the estimates \eqref{A6}, \eqref{A7}, Young's and H\"older's inequalities, one can estimate $P$ as
	\begin{align*}
		P&\leq C|z|^{\varrho_7}\int_{0}^{t}(t-s)^{-\frac{\delta}{2p}-\frac{1}{2}-\frac{\varrho_7}{2}}\|v(s)\|_{\L^p}^{\delta+1}\d s \\&\leq C|z|^{\varrho_7}t^{-\frac{\delta}{2p}-\frac{\varrho_7}{2}+\frac{1}{2}-\frac{\delta+1}{\eta}}\bigg(\int_{0}^{t}\|v(s)\|_{\L^p}^{\eta}\d s\bigg)^{\frac{\delta+1}{\eta}},
	\end{align*}for all $t\in[0,T], z\in\R$, where $0<\varrho_7<1-\frac{2p(\delta+1)+\delta\eta}{p\eta}$. 
	Using Minkowski's inequality, the estimates \eqref{A6}, \eqref{A7} and Young's inequality, we estimate $Q$ as
	\begin{align*}
		Q &\leq C|z|^{\varrho_7}\int_{0}^{t}(t-s)^{-\frac{\delta}{2p}-\frac{1}{2}-\frac{\varrho_7}{2}}\|v(s)\|_{\L^p}^{\delta+1}\d s\\&\leq C|z|^{\varrho_7}t^{-\frac{\delta}{2p}-\frac{\varrho_7}{2}+\frac{1}{2}-\frac{\delta+1}{\eta}}\bigg(\int_{0}^{t}\|v(s)\|_{\L^p}^{\eta}\d s\bigg)^{\frac{\delta+1}{\eta}},
	\end{align*}for all $t\in[0,T],\; z\in\R$, where $0<\varrho_7<1-\frac{2p(\delta+1)+\delta\eta}{p\eta}$.
	Hence the proof is completed.
\end{proof}
\begin{lemma}\label{lemma3.3}
	The following estimates hold for all $t\in[0,T]$: 
	\begin{itemize}
		\item[(i)] for $p\geq \delta+1$, we have 
		\begin{align}\label{3.01}
			\|(J_{11}v)(t)\|_{\frac{1+\e}{p},p} \leq Ct^{-\frac{\delta}{2p}+\frac{1}{2}-\frac{\delta+1}{\eta}}\bigg(\int_{0}^{t}\|v(s)\|_{\L^p}^\eta\d s\bigg)^{\frac{\delta+1}{\eta}},
		\end{align}where $\eta>\frac{2p(\delta+1)}{p-\delta}$,
		\item[(ii)] for $p\geq 1$, we have 
		\begin{align}\label{3.02}
			\|(J_{12}v)(t)\|_{\frac{1+\e}{p},p} \leq Ct^{\frac{1}{2}-\frac{1}{\eta_4}}\bigg(\int_{0}^{t}\|v(s)\|_{\L^p}^{\eta_4}\d s\bigg)^{\frac{1}{\eta_4}},
		\end{align}where $\eta_4>2$,
		\item[(iii)] for $p\geq 2\delta+1$, we have 
		\begin{align}\label{3.04}
			\|(J_{13}v)(t)\|_{\frac{1+\e}{p},p} \leq Ct^{\frac{1}{2}-\frac{\delta}{p}-\frac{2\delta+1}{\eta_5}}\bigg(\int_{0}^{t}\|v(s)\|_{\L^p}^{\eta_5}\d s\bigg)^\frac{2\delta+1}{\eta_5},
		\end{align}where $\eta_5>\frac{2p(2\delta+1)}{p-2\delta}$,
		\item[(iv)] for $p> 2\delta+1 $, we have 
		\begin{align}\label{lemma3.3.4}
			\|(J_2v)(t)\|_{\frac{1+\e}{p},2p}  \leq Ct^{\frac{1}{4}-\frac{2\delta-1}{4p}-\frac{\delta+1}{\eta_6}}\bigg(\int_{0}^{t}\|v(s)\|_{\L^p}^{\eta_6}\d s\bigg)^{\frac{\delta+1}{\eta_6}},
		\end{align}where $\eta_6>\frac{4p(\delta+1)}{p-2\delta-1}$.
	\end{itemize}
\end{lemma}
\begin{proof} Let $p\geq  \delta+1$. 
	Using the estimate \eqref{A5} with $\vartheta=1$, we get 
	\begin{align*}
		\|(J_{11}v)(t)\|_{\frac{1+\e}{p},p}^p &\leq K^p \int_{0}^{1}\int_{0}^{1}|x-y|^{-2-\e}\bigg|\int_{0}^{t}\int_{0}^{1}|x-y|(t-s)^{-1}\\&\qquad\times\max \bigg\{e^{-\frac{|x-y|^2}{\ell_5(t-s)}},e^{-\frac{|y-z|^2}{\ell_5(t-s)}}\bigg\}|v^{\delta+1}(s,y)|\d z\d s\bigg|^p\d x \d y \\&
		\leq C\int_{0}^{1}\bigg|\int_{0}^{t}\int_{0}^{1}(t-s)^{-1}e^{-\frac{|\cdot|^2}{\ell_5(t-s)}}|v^{\delta+1}(s,\cdot)|\d z \d s\bigg|^p\d x,
	\end{align*} for $p>1+\e$. Using  Minkowski's, Young's and H\"older's inequalities,  and the estimate \eqref{A7}, we find
	\begin{align*}
		\|(J_{11}v)(t)\|_{\frac{1+\e}{p},p} & \leq C\int_{0}^{t}(t-s)^{-1} \big\|e^{-\frac{|\cdot|^2}{\ell_5(t-s)}}*|v^{\delta+1}(s,\cdot)|\big\|_{\L^p}\d s
		\\& \leq C\int_{0}^{t}(t-s)^{-1}\big\|e^{-\frac{|\cdot|^2}{\ell_5(t-s)}}\big\|_{\L^{\frac{p}{p-\delta}}}\|v(s)\|_{\L^p}^{\delta+1}\d s \\& \leq C\int_{0}^{t}(t-s)^{-1+\frac{p-\delta}{2p}}\|v(s)\|_{\L^p}^{\delta+1}\d s \\&\leq 
		Ct^{-\frac{\delta}{2p}+\frac{1}{2}-\frac{\delta+1}{\eta}}\bigg(\int_{0}^{t}\|v(s)\|_{\L^p}^\eta\d s\bigg)^{\frac{\delta+1}{\eta}},
	\end{align*} for $p\geq \delta+1$, where $\eta>\frac{2p(\delta+1)}{p-\delta}$ and \eqref{3.01} follows.

	Similar calculations as we performed for \eqref{3.01} helps us to obtain \eqref{3.02} and \eqref{3.04} for $p\geq 1$ and $p\geq 2\delta+1,$ respectively.
	
	Let us now establish the final estimate \eqref{lemma3.3.4}. Let $p>2\delta+1$. Using the estimates \eqref{A6} and \eqref{A7} with $\vartheta=\frac{1}{2}$, we obtain 
	\begin{align*}
		\|(J_2v)(t)\|_{\frac{1+\e}{p},2p}^{2p} &\leq K^{2p}\int_{0}^{1}\int_{0}^{1}|x-y|^{-2-\e}\bigg|\int_{0}^{t}\int_{0}^{1}|x-y|^{\frac{1}{2}}(t-s)^{-\frac{5}{4}}\\&\qquad\times\max \bigg\{e^{-\frac{|x-y|^2}{\ell_6(t-s)}},e^{-\frac{|y-z|^2}{\ell_6(t-s)}}\bigg\}|v^{\delta+1}(s,y)|\d z\d s\bigg|^{2p}\d x \d y \\& 
		\leq C\int_{0}^{1}\bigg|\int_{0}^{t}\int_{0}^{1}(t-s)^{-\frac{5}{4}}e^{-\frac{|\cdot|^2}{\ell_6(t-s)}}|v^{\delta+1}(s,\cdot)|\d z\d s\bigg|^{2p}\d x,
	\end{align*}for $p>1+\e$. Applying Minkowski's, Young's  and H\"older's inequalities, and  the estimate \eqref{A7}, we find 
	\begin{align*}
		\|(J_2v)(t)\|_{\frac{1+\e}{p},2p} & \leq C\int_{0}^{t}(t-s)^{-\frac{5}{4}}\big\|e^{-\frac{|\cdot|^2}{\ell_6(t-s)}}*|v^{\delta+1}(s,\cdot)\big\|_{\L^{2p}}\d s
		\\&\leq C \int_{0}^{t}(t-s)^{-\frac{5}{4}}\big\|e^{-\frac{|\cdot|^2}{\ell_6(t-s)}}\big\|_{\L^{\frac{2p}{2p-2\delta-1}}}\|v(s)\|_{\L^p}^{\delta+1}\d s \\& \leq C\int_{0}^{t}(t-s)^{-\frac{5}{4}+\frac{2p-2\delta-1}{4p}}\|v(s)\|_{\L^p}^{\delta+1}\d s 
		\\&\leq Ct^{\frac{1}{4}-\frac{2\delta-1}{4p}-\frac{\delta+1}{\eta_6}}\bigg(\int_{0}^{t}\|v(s)\|_{\L^p}^{\eta_6}\d s\bigg)^{\frac{\delta+1}{\eta_6}},
	\end{align*}for  $p>2\delta+1$  where $\eta_6>\frac{4p(\delta+1)}{p-2\delta-1}$ and the proof is completed. 
\end{proof}

As a consequence of above Lemma \ref{lemma3.3}, for all $p\geq   2\delta+1$ and $\eta>\frac{2p(2\delta+1)}{p-2\delta}$, the operator $J_1:\L^\eta(0,T; \L^p(\1)) \to \C([0,T];\W^{\frac{1+\e}{p},p}(\1))$ is bounded and for all $p>2\delta+1$ and $\eta>\frac{4p(\delta+1)}{p-2\delta-1}$,  the operator $J_2:\L^\eta(0,T; \L^p(\1)) \to \C([0,T];\W^{\frac{1+\e}{p},p}(\1))$ is bounded.

Let $\phi=\{\phi(t,x):(t,x)\in[0,T]\times[0,1]\}$ be an $\R$-valued adapted process such that 
\begin{align}
	\E\bigg[\int_{0}^{T}\int_{0}^{1}|\phi(t,x)|^p\d x\d t\bigg]<\infty,
\end{align} for some $p\geq 1$. Define 
\begin{align*}
	(J_3\phi)(t,x) =\int_{0}^{t}\int_{0}^{1}G(t-s,x,y)\phi(s,y)\W(\d s,\d y),
\end{align*}for all $(t,x)\in[0,T]\times[0,1]$.
Let us now  provide  a result (see Lemma 2.2, \cite{LZN}) based on Burkholder-Davis-Gundy  inequality (cf. \cite{IY}), which helps us to estimate the operator $J_3$.
\begin{lemma}[Lemma 2.2, \cite{LZN}]\label{lemma3.5}
	For any $p\geq2$ and $\vartheta \in(0,1)$, we have 
	\begin{align}\label{lemma3.5.1}
		\E\left[\|(J_3\phi)(t)\|_{\L^p}^p\right] \leq C\int_{0}^{T}(t-s)^{-\frac{1}{2}-\vartheta}\E \big[\|\phi(s)\|_{\L^p}^p\big] \d s,
	\end{align}for every $0\leq t\leq T$. If $p\geq 2$ and $q >4$,  then we have 
	\begin{align}\label{lemma3.5.2}
		\E\left[\sup_{t\in[0,T]}\|(J_3\phi)(t)\|_{\L^p}^q\right] \leq C\int_{0}^{T} \E\big[\|\phi(s)\|_{\L^p}^q\big]\d s.
	\end{align}Finally if $q\geq p>6$ and $0<\e'<\e<\frac{q(p-2)}{2p}-2<\frac{p}{2}-3,$ it holds 
	\begin{align}\label{lemma3.5.3}
		\E\left[\sup_{t\in[0,T]}\|(J_3\phi)(t)\|_{\frac{1+\e}{p},p}^q\right] \leq CT^{q\left(\frac{p-2}{4p}\right)-1-\frac{\e'}{2}}\E\left[\int_{0}^{T}\|\phi(s)\|_{\L^p}^q\d s\right].
	\end{align}
\end{lemma} Let us recall a result from \cite{IG}, which helps us to obtain the uniform tightness for the operators $J_i$: 
\begin{lemma}[Lemma 3.3, \cite{IG}]\label{lemma3.6}
	Let $\varsigma_n(t,y)$ be a sequence of random fields on $[0,T]\times[0,1]$ such that $
	\sup\limits_{0\leq t\leq T}\|\varsigma_n(t,\cdot)\|_{\L^q}\leq \xi_n$, for $q\in[1,p],$ where $\xi_n$ is a finite random variable for every $n$. Assume that the sequence $\{\xi_n\}$ is bounded in probability, that is, 
	\begin{align*}
		\lim_{C\to\infty}\sup_{n}\P(\xi_n\geq C)=0.
	\end{align*} {Under $\eqref{A1}-\eqref{A4}$,}  the sequences $J_i(\varsigma_n)$, for $i=\{1,2\}$, are uniformly tight in $\C([0,T];\L^p(\1)),$ for $p\geq 1$.
\end{lemma}

\begin{remark}
	As discussed in \cite{LZN}, Lemmas \ref{lemma3.1}-\ref{lemma3.6} hold true if we replace $\L^p(0,1)$ by $\L^p(0,1;H)$ also. 
\end{remark}

\section{Solvability Results}\label{sec4}\setcounter{equation}{0}
In this section, we discuss the existence and uniqueness of  mild solution to the system \eqref{1.01}. In order to find a solution of the SGBH equation \eqref{1.01}, we introduce an integral form \eqref{4.02} of the equation \eqref{1.01} in the sense of Walsh (cf. \cite{JBW}). Later, with the help of a truncation defined in \eqref{4.03}, we obtain a truncated integral equation given by \eqref{4.05}. We first prove the existence of a local mild solution to the integral equation \eqref{4.02} (or global mild solution to the truncated integral equation \eqref{4.05}) up to a stopping time by using fixed point arguments (contraction mapping principle). Then we prove  uniform bounds for the weak solution of a related system \eqref{316} for any arbitrary time in Proposition \ref{lem3.4}. The global existence is established by proving that the stopping time up to which the existence of mild solution is known is the same as that of the arbitrary time almost surely.
\begin{definition}[Mild solution]\label{def4.1}
	An $\L^p(\1)$-valued and $\mathscr{F}_t$-adapted stochastic process $u:[0,\infty)\times[0,1]\times\Omega \to \R $ with $\P$-a.s. continuous trajectories on $t\in[0,T]$, is called a \emph{mild solution} to \eqref{1.01}-\eqref{1.1c}, if for any $T>0$, $u(t):=u(t,\cdot,\cdot)$ satisfies the following integral equation:\begin{align}\label{4.01}\nonumber
		u(t)&=\int_0^1G(t,x,y)u_0(y)\d y+\frac{\alpha}{\delta+1}\int_{0}^{t}\int_{0}^{1}\frac{\partial G}{\partial y}(t-s,x,y)\p(u(s,y))\d y \d s \\&\nonumber\quad+\beta\int_{0}^{t}\int_{0}^{1}G(t-s,x,y)\c(u(s,y))\d y\d s \\&\quad+\int_{0}^{t}\int_{0}^{1}G(t-s,x,y)g(s,y,u(s,y))\W(\d s,\d y),
	\end{align} 
$\P$-a.s., for all $t\in[0,T]$, where $\p(u)=u^{\delta+1}$,  $\c(u)=u(1-u^{\delta})(u^{\delta}-\gamma)$ and  $G(\cdot,\cdot,\cdot)$ is the fundamental solution of  the heat equation in the interval $[0,1]$ with Dirichlet boundary conditions.
\end{definition} 
{In the sequel, we denote the term $\int_0^1G(t,x,y)u_0(y)\d y$ by $G(t,x;u_0)$}. A solution to \eqref{1.01} can be defined in the following way also:     An $\L^2(\1)$-valued continuous and $\mathscr{F}_t$-adapted stochastic process $u=\{u(t):t\in[0,T]\}$ is a solution to \eqref{1.01}, if for any test function $\varphi\in \H^2(\1)\cap\H^1_0(\1)$, we have 
\begin{align}
	\int_0^1u(t,x)\varphi(x)\d x&=\int_0^1u_0(x)\varphi(x)\d x+\int_0^t\int_0^1u(s,x)\frac{\partial^2\varphi(x)}{\partial x^2}\d x\d s\nonumber\\&\quad+\frac{\alpha}{\delta+1}\int_0^t\int_0^1\p(u(s,x))\frac{\partial\varphi(x)}{\partial x}\d x\d s+\beta\int_0^t\int_0^1\c(u(s,x))\varphi(x)\d x\d s\nonumber\\&\quad +\int_0^t\int_0^1g(s,x,u(s,x))\varphi(x)\W(\d s,\d x), \ \mathbb{P}\text{-a.s.},
\end{align}
for all $t\in[0,T]$. 
From Theorem 2.2, \cite{IGDN} (see Proposition 3.5, \cite{IG} also),  by using Lemmas \ref{lemma3.1}-\ref{lemma3.6}, we infer that  an $\L^2(\1)$-valued continuous and $\mathscr{F}_t$-adapted stochastic process $u=\{u(t):t\in[0,T]\}$ is a solution to \eqref{1.01} if and only if $u$ satisfies the  integral equation \eqref{4.01} for all $t\in[0,T]$  and a.e. $x\in[0,1]$. For the \emph{existence and uniqueness of mild solution}, we introduce the following hypothesis on the noise coefficient of \eqref{1.01}:
\begin{hypothesis}\label{H1} The function $g:[0,T]\times[0,1]\times \R\to \R$ is a measurable function, satisfying the following conditions:
	\begin{align*}
		|g(t,x,r)|\leq K \ \text{ and }\  |g(t,x,r)-g(t,x,s)|\leq L|r-s|,
	\end{align*} for all $t\in[0,T], $ $x\in[0,1], $ $r,s\in\R$, and for some constants $K,L>0$.
\end{hypothesis}

\subsection{Existence and uniqueness of a local mild solution}
In order to find a solution of SGBH equation, we introduce the following integral form:
\begin{align}\label{4.02}\nonumber
	u(t,x)&=G(t,x;u_0)+\beta\int_{0}^{t}\int_{0}^{1}G(t-s,x,y)\c(u(s,y))\d y\d s\\&\nonumber\quad+\frac{\alpha}{\delta+1}\int_{0}^{t}\int_{0}^{1}\frac{\partial G}{\partial y}(t-s,x,y)\p(u(s,y))\d y \d s\\&\quad+\int_{0}^{t}\int_{0}^{1}G(t-s,x,y)g(s,y,u(s,y))\W(\d s,\d y),
\end{align}	where $G(t,x,y)$ is the fundamental solution of heat equation in $[0,1]$  with the Dirichlet boundary conditions. In order to obtain a solution of integral equation \eqref{4.02}, we use a truncation technique.  For any fixed natural number $n\geq 1$, consider an open ball $B(0,n)$ in the space $\L^p(\1)$ with the center at the origin and radius $n$. Define a mapping $\pi_n:\L^p(\1)\to B(0,n)$  by
\begin{align}\label{4.03}
	\pi_ny=\begin{cases}
		y, &\text { if } \|y\|_{\L^p}\leq n,\\
		\frac{n}{\|y\|_{\L^p}}y, &\text{ if } \|y\|_{\L^p}> n.
	\end{cases}
\end{align}
Let us first show that 
\begin{align}\label{34}
	\|\pi_n u-\pi_nv\|_{\L^p} \leq C\|u-v\|_{\L^p}, \ \text{ for all }\  u,v\in \L^p(\mathcal{O}). 
\end{align}
Without loss of generality, we may assume that $\|u\|_{\L^p} \leq \|v\|_{\L^p}.$ 
\begin{itemize}
	\item[Case 1:] If $\|u\|_{\L^p}\leq n$ and $\|v\|_{\L^p}\leq n$, then $\pi_nu=u$ and $\pi_n v=v$, so that we get \eqref{34} with $C=1$. 
	\item[Case 2:] If $\|u\|_{\L^p}\leq n$ and $\|v\|_{\L^p}> n$,  then $\pi_nu=u$ and $\pi_n v=\frac{n}{\|v\|_{\L^p}}v$. We estimate the term  $\|\pi_n u-\pi_nv\|_{\L^p}$ as
	\begin{align*}
		\|\pi_n u-\pi_nv\|_{\L^p} &= \bigg\|u-\frac{n}{\|v\|_{\L^p}}v\bigg\|_{\L^p} = \frac{1}{\|v\|_{\L^p}}\|(u-v)\|v\|_{\L^p}+v(\|v\|_{\L^p}-n)\|_{\L^p} \\& \leq \|u-v\|_{\L^p}+ \|v\|_{\L^p}-n \leq \|u-v\|_{\L^p}+( \|v\|_{\L^p}-\|u\|_{\L^p}) \\& \leq 2\|u-v\|_{\L^p},
	\end{align*}
and \eqref{34} follows with $C=2$. 
	\item[Case 3:] If $\|u\|_{\L^p}>n$ and $\|v\|_{\L^p}>n$, then $\pi_nu=\frac{n}{\|u\|_{\L^p}}u$ and $\pi_n v=\frac{n}{\|v\|_{\L^p}}v$. Using these relations, we estimate the term $\|\pi_n u-\pi_nv\|_{\L^p} $ as
	\begin{align*}
		\|\pi_n u-\pi_nv\|_{\L^p}  &= \bigg\| \frac{n}{\|u\|_{\L^p}}u-\frac{n}{\|v\|_{\L^p}}v\bigg\|_{\L^p}\\&= \frac{n}{\|u\|_{\L^p}\|v\|_{\L^p}}\|(u-v)\|v\|_{\L^p}+v(\|v\|_{\L^p}-\|u\|_{\L^p})\|_{\L^p} \\& 
		\leq \|u-v\|_{\L^p}+(\|v\|_{\L^p}-\|u\|_{\L^p}) \leq 2\|u-v\|_{\L^p}, 
	\end{align*}
so that  we obtain \eqref{34} with $C=2$. 
\end{itemize}
Also we can write $\pi_ny=y\phi_n(\|y\|_{\L^p}^p)$, where for any  $r>0,$ we set 
\begin{align}\label{4.04}
	\phi_n(r)=\chi_{[0,n^p]}(r)+nr^{-\frac{1}{p}}\chi_{(n^p,\infty)}(r).
\end{align}
Observe that $|\phi_n(r)|\leq 1$ and $|\phi'_n(r)|\leq \frac{1}{pr}\chi_{(n^p,\infty)}(r),$ for all $r>0$. Let us introduce the truncated integral equation 
\begin{align}\label{4.05}\nonumber
	u(t,x)&=G(t,x;u_0)+\beta\int_{0}^{t}\int_{0}^{1}G(t-s,x,y)\c(\pi_nu(s,y))\d y\d s\\&\quad\nonumber+\frac{\alpha}{\delta+1}\int_{0}^{t}\int_{0}^{1}\frac{\partial G}{\partial y}(t-s,x,y)\p(\pi_nu(s,y))\d y \d s\\&\quad +\int_{0}^{t}\int_{0}^{1}G(t-s,x,y)g(s,y,\pi_nu(s,y))\W(\d s,\d y).
\end{align} Let us set 
\begin{align}\label{4.06}\nonumber
	\mathscr{A}_1u(t,x)&:= \int_{0}^{t}\int_{0}^{1}G(t-s,x,y)\c(\pi_nu(s,y))\d y\d s, \\ \nonumber
	\mathscr{A}_2u(t,x)&:=\int_{0}^{t}\int_{0}^{1}\frac{\partial G}{\partial y}(t-s,x,y)\p(\pi_nu(s,y))\d y \d s,\\\nonumber
	\mathscr{A}_3u(t,x)&:=\int_{0}^{t}\int_{0}^{1}G(t-s,x,y)g(s,y,\pi_nu(s,y))\W(\d s,\d y), \\ 
	\mathscr{A}u(t,x)&:=G(t,x;u_0)+	\beta\mathscr{A}_1u(t,x)+\frac{\alpha}{\delta+1}\mathscr{A}_2u(t,x)+\mathscr{A}_3u(t,x).
\end{align}
The operator $\mathscr{A}$ is defined on the Banach space $\mathcal{H}$ formed {by the adapted processes $u:\Omega\times[0,T]\to \L^p(\1)$ such that}  
 \begin{align}\label{4.07}
	\|u\|_{\mathcal{H}}^p:=\int_{0}^{T}e^{-\lambda t}\E\big[\|u(t)\|_{\L^p}^p\big]\d t <\infty,
\end{align}
where $\lambda$ will be fixed later.

\begin{proposition}\label{theorem4.2}
	Let us assume that $u_0\in \L^p(\1),$ for $p\geq 2\delta+1$. Then there exists a unique $\L^p(\1)$-valued $\mathscr{F}_t$-adapted continuous process $u^n(\cdot,\cdot)$ satisfying \eqref{4.05}  such that \begin{align}\label{4.8}\E\left(\sup_{t\in[0,T]}\|u^n(t)\|_{\L^p}^p\right)\leq C(n,T).\end{align}
\end{proposition}
\begin{proof}
	To prove this theorem, we first establish that $\mathscr{A}$ is a contraction map and then by contraction mapping principle,  we ensure the existence and uniqueness of solution to the truncated integral equation  \eqref{4.05}.

		\vskip 0.2 cm 
	\noindent	\textbf{Claim 1:}  For $p\geq 2\delta+1$,  $\E\left(\sup\limits_{t\in[0,T]}\|\mathscr{A}u(t)\|_{\L^p}^p\right)\leq C(n,T)$.
	Using Minkowski's inequality, the estimates \eqref{A1}, \eqref{A7}  and Young's inequality, we get
	\begin{align*}
		&\|\mathscr{A}_1u(t)\|_{\L^p} \\&= \bigg\|\int_{0}^{t}\int_{0}^{1}G(t-s,x,y)\left[(1+\gamma)(\pi_nu)^{\delta+1}-\gamma (\pi_nu) -(\pi_nu)^{2\delta+1}\right]\d y\d s\bigg\|_{\L^p}
		\\&\leq C\int_{0}^{t}(t-s)^{-\frac{1}{2}}\left\{(1+\gamma)\|e^{-\frac{|\cdot|^2}{\ell_1(t-s)}}*|(\pi_nu)^{\delta+1}|\|_{\L^p}+\gamma\|e^{-\frac{|\cdot|^2}{\ell_1(t-s)}}*|(\pi_nu)|\|_{\L^p}\right.\\&\qquad\left.+\|e^{-\frac{|\cdot|^2}{\ell_1(t-s)}}*|(\pi_nu)^{2\delta+1}|\|_{\L^p}\right\}\d s
		\\&\leq  C\int_{0}^{t}\left\{(1+\gamma)(t-s)^{-\frac{1}{2}+\frac{p-\delta}{2p}}\|\pi_nu\|_{\L^p}^{\delta+1}+\gamma\|(\pi_nu)\|_{\L^p}+(t-s)^{-\frac{1}{2}+\frac{p-2\delta}{2p}}\|\pi_nu\|_{\L^p}^{2\delta+1}\right\}\d s
		\\& \leq C\int_{0}^{t}\left\{(1+\gamma)n^{\delta+1}(t-s)^{-\frac{1}{2}+\frac{p-\delta}{2p}}+n\gamma+n^{2\delta+1}(t-s)^{-\frac{1}{2}+\frac{p-2\delta}{2p}}\right\}\d s,
	\end{align*}for $p\geq 2\delta+1$. Again, using Minkowski's inequality, the estimates \eqref{A2}, \eqref{A7} and Young's inequality, we estimate the term $\|\mathscr{A}_2u(t)\|_{\L^p}$ as
	\begin{align*}
		\|\mathscr{A}_2u(t)\|_{\L^p} &= \bigg\|\int_{0}^{t}\int_{0}^{1}\frac{\partial G}{\partial y}(t-s,x,y)(\pi_nu)^{\delta+1}\d y\d s\bigg\|_{\L^p}\\& \leq C \int_{0}^{t}(t-s)^{-1}\|e^{-\frac{|\cdot|^2}{\ell_2(t-s)}}*| (\pi_nu)^{\delta+1}|\|_{\L^p}\d s \\& \leq 
		C\int_{0}^{t} (t-s)^{-1+\frac{p-\delta}{2p}}\|\pi_nu\|_{\L^p}^{\delta+1}\d s
		\\& \leq 
		Cn^{\delta+1}\int_{0}^{t} (t-s)^{-1+\frac{p-\delta}{2p}}\d s,
	\end{align*}for $p\geq \delta+1$. For the final term, we use the estimate \eqref{lemma3.5.2}  form Lemma \ref{lemma3.5}  to find 
	\begin{align*}
		\E\left[\sup_{0\leq t\leq T}\|\mathscr{A}_3u(t)\|_{\L^p}^q\right]&=\E\bigg[\sup_{0\leq t\leq T}\bigg\|\int_{0}^{t}\int_{0}^{1}G(t-s,x,y)g(s,y,\pi_nu(s,y))\W(\d s,\d y)\bigg\|_{\L^p}^q\bigg]\\& \leq C\int_{0}^{T}\E\big[\|g(s,\cdot,\pi_nu(s))\|_{\L^p}^q\big]\d s \leq CK^qT<+\infty,
	\end{align*}for any $p\geq 2$ and $q> 4$. Combining all the above estimates, one can complete the proof.

\vskip 0.2 cm
\noindent \textbf{Claim 2:} \emph{$\mathscr{A}$ is a contraction.} 
Let $u,v\in\mathcal{H}$.  Using Taylor's formula, the estimates \eqref{A1}, \eqref{A7}, Minkowski's, Young's and H\"older's inequalities, we get for any $p\geq 2\delta+1$
	\begin{align}\label{LMS}\nonumber
		&\|\mathscr{A}_1u(t)-\mathscr{A}_1v(t)\|_{\L^p} \\&\nonumber= \bigg\|\int_{0}^{t}\int_{0}^{1}G(t-s,x,y)\big[(1+\gamma)((\pi_nu)^{\delta+1}-(\pi_nv)^{\delta+1})-\gamma(\pi_nu-\pi_nv) \\&\nonumber\qquad-((\pi_nu)^{2\delta+1}-(\pi_nv)^{2\delta+1})\big]\d y\d s\bigg\|_{\L^p} \\&\nonumber \leq 
		C\int_{0}^{t}(t-s)^{-\frac{1}{2}} \left\{(1+\gamma)\|e^{-\frac{|\cdot|^2}{\ell_1(t-s)}}*|(\pi_nu)^{\delta+1}-(\pi_nv)^{\delta+1}|\|_{\L^p}+\gamma\|e^{-\frac{|\cdot|^2}{\ell_1(t-s)}}*|\pi_nu-\pi_nv|\|_{\L^p}\right.\\&\nonumber\qquad\left. +\|e^{-\frac{|\cdot|^2}{\ell_1(t-s)}}*|(\pi_nu)^{2\delta+1}-(\pi_nv)^{2\delta+1}|\|_{\L^p}
		\right\}\d s \\&\nonumber \leq C\int_{0}^{t} (t-s)^{-\frac{1}{2}}\bigg\{(1+\gamma)\|e^{-\frac{|\cdot|^2}{\ell_1(t-s)}}\|_{\L^{\frac{p}{p-\delta}}} \|(\pi_nu)^{\delta+1}-(\pi_nv)^{\delta+1}\|_{\L^{\frac{p}{\delta+1}}}+\gamma\|e^{-\frac{|\cdot|^2}{\ell_1(t-s)}}\|_{\L^1}\\&\nonumber\qquad\times\|\pi_nu-\pi_nv\|_{\L^p}+\|e^{-\frac{|\cdot|^2}{\ell_1(t-s)}}\|_{\L^{\frac{p}{p-2\delta}}} \|(\pi_nu)^{2\delta+1}-(\pi_nv)^{2\delta+1}\|_{\L^{\frac{p}{2\delta+1}}} \bigg\}\d s \\&\nonumber \leq  C\int_{0}^{t} (t-s)^{-\frac{1}{2}}\left\{(t-s)^{\frac{p-\delta}{2p}}(1+\gamma)(\delta+1)\|(\pi_nu-\pi_nv)(\theta (\pi_nu)+(1-\theta)(\pi_nv))^\delta\|_{\L^{\frac{p}{\delta+1}}}\right.\\&\nonumber \qquad\left.+(t-s)^{\frac{p-2\delta}{2p}}(2\delta+1)\|(\pi_nu-\pi_nv)(\theta_1(\pi_nu)+(1-\theta_1)(\pi_nv))^{2\delta}\|_{\L^{\frac{p}{2\delta+1}}}\right. \\&\nonumber \qquad\left.+\gamma(t-s)^{\frac{1}{2}}\|\pi_nu-\pi_nv\|_{\L^p}\right\}\d s
		\\& \nonumber\leq C \int_{0}^{t}\left\{(t-s)^{-\frac{\delta}{2p}}(1+\gamma)(\delta+1)\big(\|\pi_nu\|_{\L^p}+\|\pi_nv\|_{\L^p}\big)^\delta+\gamma+(t-s)^{-\frac{\delta}{p}}(2\delta+1)\right.\\&\nonumber\qquad\left.\times\big(\|\pi_nu\|_{\L^p}+\|\pi_nv\|_{\L^p}\big)^{2\delta}\right\}\|\pi_nu-\pi_nv\|_{\L^p}\d s
			\\& \leq C \int_{0}^{t}\left\{(t-s)^{-\frac{\delta}{2p}}(1+\gamma)(\delta+1)(2n)^\delta+\gamma+(t-s)^{-\frac{\delta}{p}}(2\delta+1)(2n)^{2\delta}\right\}\|\pi_nu-\pi_nv\|_{\L^p}\d s	\nonumber\\& \leq C \int_{0}^{t}\left\{(t-s)^{-\frac{\delta}{2p}}(1+\gamma)(\delta+1)(2n)^\delta+\gamma+(t-s)^{-\frac{\delta}{p}}(2\delta+1)(2n)^{2\delta}\right\}\|u-v\|_{\L^p}\d s,
	\end{align} where we have used \eqref{34}. 
Using Taylor's formula, the estimates \eqref{A2}, \eqref{A7}, Minkowski's, Young's and H\"older's inequalities, we estimate the term $\|\mathscr{A}_2u(t)-\mathscr{A}_2v(t)\|_{\L^p}$ as
	\begin{align}\label{4.09}\nonumber
		\|\mathscr{A}_2u(t)-\mathscr{A}_2v(t)\|_{\L^p} &=\bigg\|\int_{0}^{t}\int_{0}^{1}\frac{\partial G}{\partial y}(t-s,x,y)((\pi_nu)^{\delta+1}-(\pi_nv)^{\delta+1})\d y\d s\bigg\|_{\L^p}\\&  \nonumber
		\leq C \int_{0}^{t} (t-s)^{-1} \|e^{-\frac{|\cdot|^2}{\ell_2(t-s)}}*|(\pi_nu)^{\delta+1}-(\pi_nv)^{\delta+1}|\|_{\L^p} \d s \\& \nonumber\leq C\int_{0}^{t}(t-s)^{-1}\|e^{-\frac{|\cdot|^2}{\ell_2(t-s)}}\|_{\L^{\frac{p}{p-\delta}}}\|(\pi_nu)^{\delta+1}-(\pi_nv)^{\delta+1}\|_{\L^{\frac{p}{\delta+1}}}\d s 
		\\& \nonumber \leq C\int_{0}^{t}(t-s)^{-1+\frac{p-\delta}{2p}}(\delta+1)\|(\pi_nu-\pi_nv)(\theta(\pi_nu)+(1-\theta)(\pi_nv))^\delta\|_{\L^{\frac{p}{\delta+1}}}\d s 
		\\& \nonumber \leq C\int_{0}^{t}(t-s)^{-\frac{1}{2}-\frac{\delta}{2p}}(\delta+1)\big(\|\pi_nu\|_{\L^p}+\|\pi_nv\|_{\L^p}\big)^\delta \|\pi_nu-\pi_nv\|_{\L^p}\d s \\& \leq C\int_{0}^{t}(t-s)^{-\frac{1}{2}-\frac{\delta}{2p}}(\delta+1)(2n)^\delta\|\pi_nu-\pi_nv\|_{\L^p}\d s,
		\end{align}for $p\geq \delta+1$. Again using \eqref{34}, for $p\geq \delta+1$, we find
\begin{align*}
	\|\mathscr{A}_2u(t)-\mathscr{A}_2v(t)\|_{\L^p} \leq C\int_{0}^{t}(t-s)^{-\frac{1}{2}-\frac{\delta}{2p}}(\delta+1)(2n)^\delta\|u-v\|_{\L^p}\d s .
\end{align*}
Using Burkholder's inequality, \eqref{lemma3.5.1} for $\vartheta\in(0,1/2)$ form Lemma \ref{lemma3.5}, we obtain  that for any $p\geq 2$,  
\begin{align*}
	\E\big[\|\mathscr{A}_3u(t)-\mathscr{A}_3v(t)\|_{\L^p}^p\big] \leq C\int_{0}^{t}(t-s)^{-\frac{1}{2}-\vartheta}\E\big[\|\pi_nu-\pi_nv\|_{\L^p}^p\big]\d s.
\end{align*}
Once again using \eqref{34}, we deduce 
\begin{align*}
	\E\big[\|\mathscr{A}_3u(t)-\mathscr{A}_3v(t)\|_{\L^p}^p\big]&\leq C\int_{0}^{t}(t-s)^{-\frac{1}{2}-\vartheta}\E\big[\|u-v\|_{\L^p}^p\big]\d s.
\end{align*}
Using the above estimates and applying H\'older's inequality, we conclude
	\begin{align*}
		\|\mathscr{A}u-\mathscr{A}v\|_{\mathcal{H}}^p&=\int_{0}^{T}e^{-\lambda t}\E\big[\|\mathscr{A}u(t)-\mathscr{A}v(t)\|_{\L^p}^p\big]\d t\\& 
		\leq C\int_{0}^{T}e^{-\lambda t}\bigg\{
		\int_{0}^{t}\bigg((t-s)^{-\frac{\delta}{2p}}+1+(t-s)^{-\frac{\delta}{p}}+(t-s)^{-\frac{1}{2}-\frac{\delta}{2p}}+(t-s)^{-\frac{1}{2}-\vartheta}\\&\qquad\times\E\big[\|u(s)-v(s)\|_{\L^p}^p\big]\d s\bigg\}\d t \\&\leq C \bigg(\int_{0}^{\infty}e^{-\lambda y}y^{-\frac{\delta}{2p}} \d y+ \int_{0}^{\infty}e^{-\lambda y}\d y+\int_{0}^{\infty}e^{-\lambda y}y^{-\frac{\delta}{p}} \d y +\int_{0}^{\infty}e^{-\lambda y}y^{-\frac{1}{2}-\frac{\delta}{2p}} \d y\\&\qquad+\int_{0}^{\infty}e^{-\lambda y}y^{-\frac{1}{2}-\vartheta} \d y\bigg)\|u-v\|_{\mathcal{H}}^p \\& \leq 
		C\bigg(\frac{\Gamma\big(1-\frac{\delta}{2p}\big)}{\lambda^{1-\frac{\delta}{2p}}}+\frac{1}{\lambda}+\frac{\Gamma\big(1-\frac{\delta}{p}\big)}{\lambda^{1-\frac{\delta}{p}}}+\frac{\Gamma\big(\frac{1}{2}-\frac{\delta}{p}\big)}{\lambda^{\frac{1}{2}-\frac{\delta}{p}}}+\frac{\Gamma\big(\frac{1}{2}-\vartheta\big)}{\lambda^{\frac{1}{2}-\vartheta}}\bigg)\|u-v\|_{\mathcal{H}}^p,
	\end{align*}for $p\geq 2\delta+1$ and $\vartheta\in(0,1/2)$,  where the constant $C=C(n,\alpha,\beta,\gamma,\delta,p,T)$ and $\Gamma(\cdot)$ represents the gamma function. Combining the above inequality and  Claim 1, one can fix the constant $\lambda>0$ such that \begin{align*}
	C\bigg(\frac{\Gamma\big(1-\frac{\delta}{2p}\big)}{\lambda^{1-\frac{\delta}{2p}}}+\frac{1}{\lambda}+\frac{\Gamma\big(1-\frac{\delta}{p}\big)}{\lambda^{1-\frac{\delta}{p}}}+\frac{\Gamma\big(\frac{1}{2}-\frac{\delta}{p}\big)}{\lambda^{\frac{1}{2}-\frac{\delta}{p}}}+\frac{\Gamma\big(\frac{1}{2}-\vartheta\big)}{\lambda^{\frac{1}{2}-\vartheta}}\bigg)<1,
	\end{align*}for $p\geq 2\delta+1$ and $\vartheta\in(0,1/2)$. For this chosen value of $\lambda,$ the operator $\mathscr{A}$ is a contraction map on $\mathcal{H}$. Therefore there exists a unique fixed point for the contraction map $\mathscr{A}$ and this gives the existence of a unique mild solution (denoted by $u^n(\cdot,\cdot)$) of the truncated integral equation \eqref{4.05}. The estimate \eqref{4.8} follows from Claim 1.  {From the representation  \eqref{4.05} and Lemmas \ref{lemma3.1} and  \ref{lemma3.2}, and Kolmogorov's continuity theorem (see \cite{JBW}), one can show that $u^n(\cdot,\cdot)$ has an $\L^p(\1)$-valued $\mathscr{F}_t$-adapted continuous modification (cf. page 290, \cite{IG} for more details). }
\end{proof}

\begin{corollary}\label{cor3.4}
	{Let $\tau$ and $\sigma$ be stopping times bounded by some constant $T$. Let $u_1$ and $u_2$ be solutions of \eqref{4.05} in the stochastic intervals $[0,\tau]$ and $[0,\sigma]$, respectively. Then $u_1(t)=u_2(t)$, $\mathbb{P}$-a.s., in $\mathrm{L}^p(\mathcal{O})$, for $p\geq 2\delta+1$,  for all $t\in[0,\tau\wedge\sigma]$.}
\end{corollary}

The basic idea of construction of the solution $u(\cdot,\cdot)$ of the integral equation \eqref{4.01} from the solution $u^n(\cdot,\cdot)$ of the truncated integral equation \eqref{4.05} is borrowed from the works \cite{GDDG,IGDN}. Let us now prove the global existence and uniqueness of mild solution of the equation \eqref{1.01}. Proposition \ref{theorem4.2}  implies the uniqueness of a solution and a local existence of a solution of \eqref{1.01}. The global existence will be proved using the uniform estimate obtained in the following proposition.
{
\begin{proposition}\label{lem3.4}
	Let $\varphi=\left\{\varphi(t,x),\ t\in[0,T],\ x\in[0,1]\right\}$ be a continuous function belonging to $\C([0,T];\L^p(\1))$, for $p \geq 2$. Let $v\in \C([0,T];\L^p(\1)) $, for $p\geq  2\delta+1$ be a solution of the integral equation 
	\begin{align}\label{4.13}\nonumber
		v(t,x)&=G(t,x;u_0)+\beta\int_{0}^{t}\int_{0}^{1}G(t-s,x,y)\c(v(s,y)+\varphi(s,y))\d y\d s\\&\quad+\frac{\alpha}{\delta+1}\int_{0}^{t}\int_{0}^{1}\frac{\partial G}{\partial y}(t-s,x,y)\p(v(s,y)+\varphi(s,y))\d y \d s, 
	\end{align}
where  $u_0\in \L^p(\1).$ Then we have 
 \begin{align}\label{3.13}\nonumber
	&\|v(t)\|_{\L^p}^p +\frac{\nu p(p-1)}{2}\int_0^t\||v(s)|^{\frac{p-2}{2}}\partial_x v(s)\|_{\L^2}^2\d s+\beta p\gamma\int_0^t\|v(s)\|_{\L^p}^p\d s+\frac{\beta p}{8}\int_0^t\|v(s)\|_{\L^{p+2\delta}}^{p+2\delta}\d s \\& \leq  \|u_0\|_{\L^p}^p+pK_1T+pK_2T\sup_{t\in[0,T]}\|\varphi(t)\|_{\L^{p(\delta+1)}}^{p(\delta+1)}+pK_3T\sup_{t\in[0,T]}\|\varphi(t)\|_{\L^{p+2\delta}}^{p+2\delta},
\end{align}
where 
\begin{align}\label{3150}
K_1&=\frac{2\delta  }{p+2\delta}\bigg(\frac{8p}{p+2\delta}\bigg)^{\frac{p}{2\delta}}\bigg\{   \frac{2^{\delta}(p-1)^2\alpha^2}{4\nu} +2^{2\delta}\beta(1+\gamma)^2+2^\delta\beta(1+\gamma)+2^{2\delta-1}\beta(2\delta+1) \bigg\}^{\frac{p+2\delta}{2\delta}},\\
K_2&=
 \frac{2^\delta\beta(1+\gamma)}{p}\bigg(\frac{p-1}{p}\bigg)^{p-1} +\frac{2^{\delta}(p-1)^2\alpha^2}{\nu p}\bigg(\frac{2(p-2)}{p}\bigg)^{\frac{2}{p-2}}\label{314},\\ \nonumber K_3&=\frac{1}{p+2\delta}\bigg\{	\bigg(\frac{4(p+2\delta-1)}{\beta(p+2\delta)}\bigg)^{p+2\delta-1}\big(2^{2\delta-1}\beta(2\delta+1)\big)^{p+2\delta} +2\bigg(\frac{4(p+2\delta-2)}{\beta(p+2\delta)}\bigg)^{\frac{p+2\delta-2}{2}}\\&\qquad\times\bigg(\frac{2^{\delta}(p-1)^2\alpha^2}{2\nu}\bigg)^{\frac{p+2\delta}{2}}\bigg\},\label{315}
\end{align}for all $t\in[0,T]$.
\end{proposition}
}
\begin{proof}
	The proof of  this proposition is divided in the following steps:
	\vskip 0.2cm 
	\noindent\textbf{Step 1:} 
	Let us take  the initial data $u_0\in\L^p(\1)$, for $p\geq 2\delta+1$ and assume that $v(\cdot)$ is a solution of \eqref{4.13} corresponding to  $u_0$. If $v_1,v_2\in\C([0,T];\L^p(\1)) $ are two	solutions of \eqref{4.13} corresponding to the initial conditions $u_0^1$ and $u_0^2$, respectively, then we have 
	\begin{align}\label{3.013}\nonumber
		&\|v_1(t)-v_2(t)\|_{\L^p}\\&\nonumber\leq \|u_0^1-u_0^2\|_{\L^p}\\&\nonumber\quad +C\beta\int_0^t\bigg\{(t-s)^{-\frac{\delta}{2p}}(1+\gamma)(\delta+1)\big(\|v_1(s)+\varphi(s)\|_{\L^p}+\|v_2(s)+\varphi(s)\|_{\L^p}\big)^\delta\\&\nonumber\qquad+\gamma+(t-s)^{-\frac{\delta}{p}}(2\delta+1)\big(\|v_1(s)+\varphi(s)\|_{\L^p}+\|v_2(s)+\varphi(s)\|_{\L^p}\big)^{2\delta}\bigg\}\|v_1(s)-v_2(s)\|_{\L^p}\d s\\&\quad + C\alpha\int_0^t(t-s)^{-\frac{1}{2}-\frac{\delta}{2p}}\big(\|v_1(s)+\varphi(s)\|_{\L^p}+\|v_2(s)+\varphi(s)\|_{\L^p}\big)^\delta\|v_1(s)-v_2(s)\|_{\L^p}\d s,
		\end{align}for $p\geq2\delta+1$, where we have used a similar calculation as in Claim 2 of Proposition \ref{theorem4.2}. 
	Since $v_1,v_2\in\C([0,T];\L^p(\1)) $, an application of Gronwall's inequality yields 
	\begin{align}
		\|v_1(t)-v_2(t)\|_{\L^p}\leq C\|u_0^1-u_0^2\|_{\L^p},\ \text{ for all }\ t\in[0,T],
	\end{align}
which implies the uniqueness of the solution to the integral equation \eqref{4.13}. 
\vskip 0.2cm 
\noindent\textbf{Step 2:} 
Let us now  show that  there is a unique  solution to \eqref{4.13} satisfying   the estimate \eqref{3.13}.  Using  Lemmas \ref{lemma3.1}-\ref{lemma3.5}, one can show in a similar way as in Theorems 2.2 and 2.3, \cite{IGDN} (see \cite{IG} also) that an $\L^p(\1)$-valued continuous function $v$  is a weak solution to
	\begin{equation}\label{316}
		\left\{
		\begin{aligned}
			\frac{\partial v(t)}{\partial t}  &= \nu \frac{\partial^2 v(t) }{\partial x^2}-\frac{\alpha}{\delta+1}\frac{\partial}{\partial x}(\p(v(t)+\varphi(t)))+\beta \c(v(t)+\varphi(t)),\; t\in(0,T),\\
			v(0)&=u_0, 
		\end{aligned}
		\right.
	\end{equation}
 if $v$ satisfies the  integral equation \eqref{4.13} for all $t\in[0,T]$ and for almost all $x\in[0,1]$.
 
In order to prove the existence of weak solution to \eqref{316},  we  first define the operator $\mathfrak{A}:\H_0^1(\1)\to \H^{-1}(\1)$ as 
 \begin{align}
 	\langle\mathfrak{A}(\psi),\zeta\rangle &=\nu\int_0^1\frac{\partial \psi(x)}{\partial x}\frac{\partial \zeta(x)}{\partial x}\d x-\frac{\alpha}{\delta+1}\int_0^1\mathfrak{p}(\psi(x)+\varphi(x))\frac{\partial \zeta(x)}{\partial x}\d x\nonumber\\&\quad-\beta\int_0^1\mathfrak{c}(\psi(x)+\varphi(x))\zeta(x)\d x,
 \end{align}
for all $\psi,\zeta\in \H_0^1(\1)$. Note that as we are working in one-dimension, one can replace the partial derivative symbol with full derivative symbol also. 
\vskip 0.1 cm
\noindent
\emph{The operator $\mathfrak{A}$ is well-defined:} 
It can be easily seen that the operator $\mathfrak{A}$ is well defined since for all $\psi,\zeta\in \H_0^1(\1)$, by using Sobolev's inequality, we have 
\begin{align}\label{3p18}\nonumber
	&|\langle\mathfrak{A}(\psi),\zeta\rangle |\\&\leq \bigg[\nu\|\partial_x\psi\|_{\L^2}+\frac{\alpha}{\delta+1}\|\psi+\varphi\|_{\L^{2(\delta+1)}}^{\delta+1}+\beta\left(\gamma\|\psi+\varphi\|_{\L^2}+(1+\gamma)\|\psi+\varphi\|_{\L^{2(\delta+1)}}^{\delta+1}\right)\bigg]\|\partial_x\zeta\|_{\L^2}\nonumber\\&\quad+\beta\|\psi+\varphi\|_{\L^{2(\delta+1)}}^{2\delta+1}\|\zeta\|_{\L^{2(\delta+1)}}\\&\leq C\bigg[\|\psi\|_{\H_0^1}+\|\psi\|_{\H_0^1}^{\delta+1}+\|\psi\|_{\H_0^1}^{2\delta+1}+\|\varphi\|_{\L^{2(\delta+1)}}^{\delta+1}+\|\varphi\|_{\L^{2(\delta+1)}}^{2\delta+1}\bigg]\|\zeta\|_{\H_0^1},\nonumber
	\end{align}
where we have used the Sobolev embedding $\H_0^1(\1)\subset \L^p(\1),$ for any $p\in[1,\infty],$ also. Therefore, we deduce that 
\begin{align*}
	\|\mathfrak{A}(\psi)\|_{\H^{-1}}\leq C\bigg[\|\psi\|_{\H_0^1}+\|\psi\|_{\H_0^1}^{\delta+1}+\|\psi\|_{\H_0^1}^{2\delta+1}+\|\varphi\|_{\L^{2(\delta+1)}}^{\delta+1}+\|\varphi\|_{\L^{(2\delta+1)}}^{2\delta+1}\bigg],
\end{align*}
for all $\psi\in\H_0^1(\1)$. Hence,  we  cast the system  \eqref{316}  in the following evolution equation: 
\begin{align}\label{318}
	v(t)=u_0-\int_0^t\mathfrak{A}(v(s))\d s,
	\end{align}
for all $t\in[0,T]$ in the triplet $\H_0^1(\1)\subset\L^2(\1)\subset\H^{-1}(\1)$. Let us now prove some properties of the operator $\mathfrak{A}$, so that we obtain the unique weak solution of \eqref{316}. 

\vskip 0.1 cm
\noindent\emph{The operator $\mathfrak{A}$ is coercive:} 
By using integration by parts, Taylor's formula, H\"older's and Young's inequalities, we have for all $\psi\in\H_0^1(\1)$
\begin{align}\label{3p20}
&\langle	\mathfrak{A}(\psi),\psi\rangle\nonumber\\&=\nu\|\partial_x\psi\|_{\L^2}^2-\frac{\alpha}{\delta+1}\int_0^1\mathfrak{p}(\psi(x)+\varphi(x))\frac{\partial\psi(x)}{\partial x}\d x+\beta\gamma\|\psi+\varphi\|_{\L^2}^2+\beta\|\psi+\varphi\|_{\L^{2(\delta+1)}}^{2(\delta+1)}\nonumber\\&\quad-\beta(1+\gamma)\int_0^1(\psi(x)+\varphi(x))^{\delta+1}(\psi(x)+\varphi(x))\d x+\beta\int_0^1\mathfrak{c}(\psi(x)+\varphi(x))\varphi(x)\d x\nonumber\\&\geq \frac{\nu}{2}\|\psi\|_{\H_0^1}^2+\frac{\beta}{2}\|\psi+\varphi\|_{\L^{2(\delta+1)}}^{2(\delta+1)}-2\beta(1+\gamma)^2\|\psi+\varphi\|_{\L^2}^2-\beta\left[2(1+\gamma)^2+\frac{\gamma}{4}\right]\|\varphi\|_{\L^2}^2\nonumber\\&\quad-\left\{\frac{\alpha^2}{\nu}2^{2(\delta-1)}\left[1+\frac{1}{\delta+1}\left(\frac{8\delta}{\beta(\delta+1)}\right)^{\delta}\left(\frac{\alpha^2}{\nu}2^{2(\delta-1)}\right)^{\delta}\right]+\frac{\beta}{2\delta+1}\left(\frac{4(2\delta+1)}{\delta+1}\right)^{2\delta+1}\right\}\|\varphi\|_{\L^{2(\delta+1)}}^{2(\delta+1)}, 
\end{align}
which  implies  that the operator $\mathfrak{A}$ is coercive. 
\vskip 0.1 cm
\noindent
\emph{The operator $\mathfrak{A}$ is locally monotone:} 
The following local monotonicity condition holds for the operator $\mathfrak{A}$: 
\begin{align}\label{320}\nonumber
&	\langle	\mathfrak{A}(\psi_1)-\mathfrak{A}(\psi_2),\psi_1-\psi_2\rangle\\&\nonumber\quad+\bigg\{2^{2\delta-1}\beta(1+\gamma)^2(\delta+1)^2+\frac{27\alpha^4}{32\nu^3}\big(\|\psi_1+\varphi\|_{\L^{2\delta}}+\|\psi_2+\varphi\|_{\L^{2\delta}}\big)^{4\delta}\bigg\}\|\psi_1-\psi_2\|_{\L^2}^2\nonumber\\&\geq 
 \frac{\nu}{2}\|\partial_x(\psi_1-\psi_2)\|_{\L^2}^2+\beta\gamma\|\psi_1-\psi_2\|_{\L^2}^2+\frac{\beta}{2^{2\delta+1}}\|\psi_1-\psi_2\|_{\L^{2(\delta+1)}}^{2(\delta+1)},
	\end{align}
for all $\psi_1,\psi_2\in \H_0^1(\1)$. The condition \eqref{320} can be justified as follows:
\begin{align}\label{LMC1}\nonumber
	&	\langle	\mathfrak{A}(\psi_1)-\mathfrak{A}(\psi_2),\psi_1-\psi_2\rangle \\&\nonumber= \nu \|\partial_x(\psi_1-\psi_2)\|_{\L^2}^2-\frac{\alpha}{\delta+1}\big(\p(\psi_1+\varphi)-\p(\psi_2+\varphi),\partial_x(\psi_1-\psi_2)\big)\\&\quad -\beta\big(\c(\psi_1+\varphi)-\c(\psi_2+\varphi),\psi_1-\psi_2\big),
\end{align}where $\psi_1,\psi_2\in \H_0^1(\1)$. Using Taylor's formula, H\"older's, Agmon's and Young's inequalities in the term $\frac{\alpha}{\delta+1}\big(\p(\psi_1+\varphi)-\p(\psi_2+\varphi),\partial_x(\psi_1-\psi_2)\big)$, we find
\begin{align}\label{LMC2}\nonumber
	&\frac{\alpha}{\delta+1}\big(\p(\psi_1+\varphi)-\p(\psi_2+\varphi),\partial_x(\psi_1-\psi_2)\big) \\&\nonumber = \alpha \big((\psi_1-\psi_2)(\theta (\psi_1+\varphi)+(1-\theta) (\psi_2+\varphi))^\delta,\partial_x(\psi_1-\psi_2)\big) \\& \nonumber\leq 
	 \alpha \big(\|\psi_1+\varphi\|_{\L^{2\delta}}+\|\psi_2+\varphi\|_{\L^{2\delta}}\big)^\delta\|\psi_1-\psi_2\|_{\L^\infty}\|\partial_x(\psi_1-\psi_2)\|_{\L^2}\\& \nonumber\leq 
	  \alpha \big(\|\psi_1+\varphi\|_{\L^{2\delta}}+\|\psi_2+\varphi\|_{\L^{2\delta}}\big)^\delta\|\psi_1-\psi_2\|_{\L^2}^{\frac{1}{2}}\|\partial_x(\psi_1-\psi_2)\|_{\L^2}^{\frac{3}{2}} \\& \leq 
	   \frac{\nu}{2}\|\partial_x(\psi_1-\psi_2)\|_{\L^2}^2+\frac{27\alpha^4}{32\nu^3}\big(\|\psi_1+\varphi\|_{\L^{2\delta}}+\|\psi_2+\varphi\|_{\L^{2\delta}}\big)^{4\delta}\|\psi_1-\psi_2\|_{\L^2}^2.
\end{align}From Theorem 3.1, \cite{MTMAK} (see Eq. (40)), we have 
\begin{align}\label{LMC3}\nonumber
&	\beta\big(\c(\psi_1+\varphi)-\c(\psi_2+\varphi),\psi_1-\psi_2\big)\\& \nonumber\leq -\beta\gamma \|\psi_1-\psi_2\|_{\L^2}^2-\frac{\beta}{4}\|(\psi_1+\varphi)^\delta(\psi_1-\psi_2)\|_{\L^2}^2-\frac{\beta}{4}\|(\psi_2+\varphi)^\delta(\psi_1-\psi_2)\|_{\L^2}^2\\&\quad +2^{2\delta-1}\beta(1+\gamma)^2(\delta+1)^2\|\psi_1-\psi_2\|_{\L^2}^2.
\end{align}
Note that 
\begin{align}\label{LMC4}
	\|\psi_1-\psi_2\|_{\L^{2(\delta+1)}}^{2(\delta+1)}&=\int_0^1|(\psi_1(x)+\varphi(x))-(\psi_2(x)+\varphi(x))|^{2\delta}|\psi_1(x)-\psi_2(x)|^2\d x\nonumber\\&\leq 2^{2\delta-1}\int_0^1\left(|\psi_1(x)+\varphi(x)|^{2\delta}+|\psi_2(x)+\varphi(x)|^{2\delta}\right)|\psi_1(x)-\psi_2(x)|^2\d x\nonumber\\&=2^{2\delta-1}\left(\|(\psi_1+\varphi)^{\delta}(\psi_1-\psi_2)\|_{\L^2}^2+\|(\psi_2+\varphi)^{\delta}(\psi_1-\psi_2)\|_{\L^2}^2\right).
\end{align}
Substituting \eqref{LMC2}-\eqref{LMC4} in \eqref{LMC1}, we obtain the required condition \eqref{320}, which implies that the operator $\mathfrak{A}$ satisfies a local monotonicity condition (on a ball in $\L^{2\delta}(\1)$). That is, if we take a ball $B_N:=\{\zeta\in\L^{2\delta}(\1):\|\zeta\|_{\L^{2\delta}}\leq N\}$, then from \eqref{320}, we infer
\begin{align}
		\langle	\mathfrak{A}(\psi_1)-\mathfrak{A}(\psi_2),\psi_1-\psi_2\rangle+C(1+N^{4\delta})\|\psi_1-\psi_2\|_{\L^2}^2\geq 0, 
\end{align}
for all $\psi_1,\psi_2\in B_N\subset\H_0^1(\1)$. 
\vskip 0.1 cm
\noindent
\emph{The operator $\mathfrak{A}$ is hemicontinuous:} 
Hemicontinuity of the operator $\mathfrak{A}(\cdot)$  follows easily from demicontinuity. 
Let us prove the demicontinuity of the operator $\mathfrak{A}(\cdot)$. We choose a sequence $\{\psi_n\}\in\H_0^1(\1)$ such that $\psi_n\to \psi $ in $\H_0^1(\1)$.  For any $\zeta\in \H_0^1(\1)$, we consider 
\begin{align}\label{DC1}\nonumber
\left|	\langle \mathfrak{A}(\psi_n)-\mathfrak{A}(\psi),\zeta\rangle \right|&= \big|\nu \big(\partial_x(\psi_n-\psi),\partial_x\zeta\big)-\frac{\alpha}{\delta+1}\big(\p(\psi_n+\varphi)-\p(\psi+\varphi), \partial_x\zeta\big)\\&\quad-\beta \big(\c(\psi_n+\varphi)-\c(\psi+\varphi),\zeta\big)\big|. 
\end{align}Let us first take $|\big(\partial_x(\psi_n-\psi),\partial_xv\big)|$, and estimate it as 
\begin{align}\label{DC2}
	\big|\big(\partial_x(\psi_n-\psi),\partial_x\zeta\big)\big| \leq \|\psi_n-\psi\|_{\H_0^1}\|\zeta\|_{\H_0^1} \to 0, \text{ as }n\to\infty, 
\end{align}since $\psi_n\to\psi$ in $\H_0^1(\1)$. 

Using Taylor's formula, H\"older's, Sobolev's and Young's inequalities,  we estimate the term $\big|\big(\p(\psi_n+\varphi)-\p(\psi+\varphi), \partial_x\zeta\big)\big|$ as
\begin{align}\label{DC3}\nonumber
	&\left|\big(\p(\psi_n+\varphi)-\p(\psi+\varphi), \partial_x\zeta\big)\right|\\&\nonumber \leq \big( \|\psi_n+\varphi\|_{\L^{2\delta}}+\|\psi+\varphi\|_{\L^{2\delta}}\big)^{\delta}\|\psi_n-\psi\|_{\L^\infty}\|\partial_x\zeta\|_{\L^2}\\&\nonumber \leq C \big( \|\psi_n+\varphi\|_{\L^{2\delta}}+\|\psi+\varphi\|_{\L^{2\delta}}\big)^{\delta}\|\psi_n-\psi\|_{\H_0^1}\|\partial_x\zeta\|_{\L^2}\\& \to 0, \text{ as } n\to \infty,
\end{align}since $\psi_n\to\psi$ in $\H_0^1(\1)$.  Similarly, we estimate the final term in the right hand side of \eqref{DC1} as
\begin{align}\label{DC4}\nonumber
	&	\left| \big(\c(\psi_n+\varphi)-\c(\psi+\varphi),\zeta\big) \right| \\&\nonumber\leq C \left\{(1+\gamma)(\delta+1)\big(\|\psi_n+\varphi\|_{\L^{2\delta}}+\|\psi+\varphi\|_{\L^{2\delta}}\big)^\delta+\gamma\right.\\&\nonumber\qquad\left.+ (2\delta+1)\big(\|\psi_n+\varphi\|_{\L^{4\delta}}+\|\psi+\varphi\|_{\L^{4\delta}}\big)^{2\delta} \right\}\|\psi_n-\psi\|_{\H_0^1}\|\zeta\|_{\L^2}\\& \to 0, \text{ as } n\to \infty,
\end{align}since $\psi_n\to\psi$ in $\H_0^1(\1)$. From the convergences \eqref{DC1}-\eqref{DC4}, we deduce
\begin{align}\label{DC5}
		|\langle \mathfrak{A}(\psi_n)-\mathfrak{A}(\psi),\zeta\rangle|\to0, \text{ as } n\to \infty ,
		\end{align}for all $\zeta\in \H_0^1(\1)$. Therefore $\mathfrak{A}(\psi_n)\rightharpoonup \mathfrak{A}(\psi), $ in $\H^{-1}(\1)$ and  the operator $\mathfrak{A}:\H_0^1(\1)\to\H^{-1}(\1)$ is demicontinuous, which implies that the operator $\mathfrak{A}(\cdot)$ is hemicontinuous also.
\vskip 0.1 cm
\noindent
\emph{The operator $\mathfrak{A}$ is locally Lipschitz :} Using the same procedure as in the \emph{hemicontinuity} part, for all $\psi_1,\psi_2,\zeta\in\H_0^1(\1),$ we find
\begin{align}\label{LL1}\nonumber
	&	|\langle	\mathfrak{A}(\psi_1)-\mathfrak{A}(\psi_2),\zeta\rangle |\\&\nonumber\leq  \nu \|\psi_1-\psi_2\|_{\H_0^1}\|\zeta\|_{\H_0^1}+C\alpha  \big(\|\psi_1\|_{\H_0^1}+\|\psi_2\|_{\H_0^1}+2\|\varphi\|_{\L^{2\delta}}\big)^\delta\|\psi_1-\psi_2\|_{\H_0^1}\|\zeta\|_{\H_0^1}\\&\nonumber\quad+C\beta \left\{ (1+\gamma)(1+\delta)\big(\|\psi_1\|_{\H_0^1}+\|\psi_2\|_{\H_0^1}+2\|\varphi\|_{\L^{2\delta}}\big)^{\delta} +\gamma\right.\\&\qquad\left.+(2\delta+1) \big(\|\psi_1\|_{\H_0^1}+\|\psi_2\|_{\H_0^1}+2\|\varphi\|_{\L^{4\delta}}\big)^{2\delta} \right\}\|\psi_1-\psi_2\|_{\H_0^1}\|\zeta\|_{\H_0^1}, 
\end{align}where we have replaced $\psi_n$ and $\psi$ by $\psi_1$ and $\psi_2$, respectively, in \eqref{DC1}. From \eqref{LL1}, we find 
\begin{align}\label{LL2}\nonumber
&	\|\mathfrak{A}(\psi_1)-\mathfrak{A}(\psi_2)\|_{\H^{-1}}  \\&\nonumber\leq \left\{\nu+C\alpha (2r+2\|\varphi\|_{\L^{2\delta}})^\delta+C\beta\big((1+\gamma)(1+\delta)(2r+2\|\varphi\|_{\L^{2\delta}})^{\delta}+\gamma\right.\\&\qquad\left.+(2\delta+1)(2r+2\|\varphi\|_{\L^{4\delta}})^{2\delta}\big)\right\}\|\psi_1-\psi_2\|_{\H_0^1},
\end{align}for $\|\psi_1\|_{\H_0^1},\|\psi_2\|_{\H_0^1}\leq r$. Thus the operator $\mathfrak{A}:\H_0^1(\1)\to\H^{-1}(\1)$ is locally Lipschitz.

\vskip 0.2cm 
\noindent\textbf{Step 3:} 
Let us now prove the existence of a weak solution to the equation \eqref{316} by using a Faedo-Galerkin approximation, local monotonicity condition and hemicontinuity.

Let
$\{w_1,\ldots,w_n,\ldots\}$ be a complete orthonormal system in
$\L^2(\1)$ belonging to $\H_0^1(\1)$ and let $\L^2_n(\1)$ be the
$n$-dimensional subspace of $\L^2(\1)$ spanned by $\{w_1,\ldots,w_n \}$. Let $\mathrm{P}_n$ denote
the projection of $\H^{-1}(\1)$ to $\L^2_n(\1)$, that is, $\mathrm{P}_nx=\sum\limits_{k=1}^n\langle x,w_k\rangle w_k$. Since every element $x\in\L^2(\1)$ induces a functional $x^*\in\L^2(\1)$  by the formula $\langle x^*,y\rangle =(x,y)$, $y\in\H_0^1(\1)$, then $\mathrm{P}_n\big|_{\L^2(\1)}$, the orthogonal projection of $\L^2(\1)$ onto $\L^2_n(\1)$  is given by $\mathrm{P}_nx=\sum\limits_{k=1}^n\left(x,w_k\right)w_k$. Hence in particular, $\mathrm{P}_n$ is the orthogonal projection from $\L^2(\1)$ onto $\text{span}\{w_1,\ldots,w_n\}$.  In fact, one can take $\{w_k\}_{k=1}^{\infty}$ as  the eigenvalues and  corresponding eigenfunctions of the operator $\A:= -\frac{\partial^2  }{\partial x^2},$  which is given by 
\begin{align*}
	\lambda_k = k^2\pi^2 \ \text{ and } \ w_k(\xi) = \sqrt{2} \sin(k\pi x), \   k=1,2,\ldots	.
\end{align*}
As we are working on the bounded domain $\1$, the inverse of $\A$, that is, $\A^{-1}$ exists and is a compact operator on $\L^2(\1)$. Moreover, we can define the fractional powers of $\A$ and 
\begin{align}\label{poin}
\|\partial_xu\|_{\L^2}^2=	\|\A^{\frac{1}{2}}u\|_{\L^2}^2=\sum_{k=1}^{\infty}|(u,w_k)|^2 \geq \lambda_1\sum_{k=1}^{\infty}|(u,w_k)|^2 = \lambda_1\|u\|_{\L^2}^2 = \pi^2\|u\|_{\L^2}^2,
\end{align} 
which is the Poincar\'e inequality. 

We define $\mathfrak{A}_n(v^n)=\mathrm{P}_n\mathfrak{A}(v^n)$. With the above setting, let us now consider the following system of ODEs:
\begin{equation}\label{4p11}
	\left\{
	\begin{aligned}
	\langle\partial_tv^n(t),\varphi\rangle&=-\langle\mathfrak{A}_n(v^n(t)),\varphi\rangle,\\
		(v^n(0),\varphi)&=(v_0^n,\varphi),
	\end{aligned}
	\right.
\end{equation}
with $v_0^n=\mathrm{P}_nv_0,$ for all $\varphi\in\L^2_n(\1)$. 
Since $\mathfrak{A}_n(\cdot)$ is locally Lipschitz (see \eqref{LL2}), the system  \eqref{4p11} has a unique local solution $v^n\in\C([0,T^*];\L^2_n(\1))$,  for some $T^*\in(0,T]$ (existence and uniqueness is due to Picard's  theorem). 

Let us now establish the energy estimates satisfied by the system \eqref{4p11} and extend this local solution to a global one. Taking the inner product with $v^n(\cdot)$ to the first equation in \eqref{4p11} and using \eqref{3p20}, we find for all $t\in[0,T]$ 
\begin{align}\label{337}
	\|v^n(t)\|_{\L^2}^2&=\|v_0^n\|_{\L^2}^2-2\int_0^t\langle \mathfrak{A}_n(v^n(s)),v^n(s)\rangle\d s\\&\leq \|v_0\|_{\L^2}^2-\nu\int_0^t\|\partial_xv^n(s)\|_{\L^2}^2\d s-\beta\int_0^t\|v^n(s)+\varphi(s)\|_{\L^{2(\delta+1)}}^{2(\delta+1)}\d s\nonumber\\&\quad+4\beta(1+\gamma)^2\int_0^t\|v^n(s)+\varphi(s)\|_{\L^2}^2\d s+2\beta\left[2(1+\gamma)^2+\frac{\gamma}{4}\right]\int_0^t\|\varphi(s)\|_{\L^2}^2\d s\nonumber\\&\quad+\tilde{K}\int_0^t\|\varphi(s)\|_{\L^{2(\delta+1)}}^{2(\delta+1)}\d s,\nonumber
\end{align}
where $\tilde{K}=2\left\{\frac{\alpha^2}{\nu}2^{2(\delta-1)}\left[1+\frac{1}{\delta+1}\left(\frac{8\delta}{\beta(\delta+1)}\right)^{\delta}\left(\frac{\alpha^2}{\nu}2^{2(\delta-1)}\right)^{\delta}\right]+\frac{\beta}{2\delta+1}\left(\frac{4(2\delta+1)}{\delta+1}\right)^{2\delta+1}\right\}$. Therefore, an application of Gronwall's inequality yields 
\begin{align}\label{323}
	\sup_{t\in[0,T]}	\|v^n(t)\|_{\L^2}^2+\nu\int_0^T\|\partial_x v^n(t)\|_{\L^2}^2\d t+\beta\int_0^T\|v^n(t)\|_{\L^{2(\delta+1)}}^{2(\delta+1)}\d t\leq C. 
\end{align}
Since $v_0\in\L^p(\1)$, for $p\geq 2\delta+1$, one can obtain further energy estimates satisfied by $v^n(\cdot)$. Let us take the inner product with $|v^n(\cdot)|^{p-2}v^n(\cdot)$ to the first  equation in \eqref{4p11} to find
	\begin{align*}
	&	\frac{1}{p}\frac{\d }{\d t} \|v^n(t)\|_{\L^p}^p +\nu (p-1)\||v^n(t)|^{\frac{p-2}{2}}\partial_x v^n(t)\|_{\L^2}^2 \\&=-\frac{\alpha}{\delta+1}\big(\partial_x(v^n(t)+\varphi(t))^{\delta+1},|v^n(t)|^{p-2}v^n(t)\big)\\&\quad+\beta(1+\gamma)\big((v^n(t)+\varphi(t))^{\delta+1},|v^n(t)|^{p-2}v^n(t)\big)\\&\quad -\beta\gamma \big(v^n(t)+\varphi(t),|v^n(t)|^{p-2}v^n(t)\big)-\beta\big((v^n(t)+\varphi(t))^{2\delta+1},|v^n(t)|^{p-2}v^n(t)\big)\\&=-\frac{\alpha}{\delta+1}\big(\partial_x(v^n(t)+\varphi(t))^{\delta+1},|v^n(t)|^{p-2}v^n(t)\big)\\&\quad+\beta(1+\gamma)\big((v^n(t)+\varphi(t))^{\delta+1},|v^n(t)|^{p-2}v^n(t)\big)\\&\quad -\beta\gamma\|v^n(t)\|_{\L^p}^p-\beta\gamma \big(\varphi(t),|v^n(t)|^{p-2}v^n(t)\big)-\beta\|v^n(t)\|_{\L^{p+2\delta}}^{p+2\delta}\\&\quad-\beta(2\delta+1)\big(\varphi(t)(\theta_1v^n(t)+(1-\theta_1)\varphi(t))^{2\delta},|v^n(t)|^{p-2}v^n(t)\big),
\end{align*}for some $\theta_1\in(0,1)$, where we have used Taylor's formula. One can rewrite the previous equality as 
\begin{align}\label{UEE1}\nonumber
	&	\frac{1}{p}\frac{\d }{\d t} \|v^n(t)\|_{\L^p}^p +\nu (p-1)\||v^n(t)|^{\frac{p-2}{2}}\partial_x v^n(t)\|_{\L^2}^2+\beta\gamma\|v^n(t)\|_{\L^p}^p+\beta\|v^n(t)\|_{\L^{p+2\delta}}^{p+2\delta} 
	\\&\nonumber=-\frac{\alpha}{\delta+1}\big(\partial_x(v^n(t)+\varphi(t))^{\delta+1},|v^n(t)|^{p-2}v^n(t)\big)\\&\nonumber\quad+\beta(1+\gamma)\big((v^n(t)+\varphi(t))^{\delta+1},|v^n(t)|^{p-2}v^n(t)\big)\\&\nonumber\quad -\beta\gamma \big(\varphi(t),|v^n(t)|^{p-2}v^n(t)\big)-\beta(2\delta+1)\big(\varphi(t)(\theta_1v^n(t)+(1-\theta_1)\varphi(t))^{2\delta},|v^n(t)|^{p-2}v^n(t)\big)\\&=: I_1+I_2+I_3+I_4.
\end{align}
An integration by parts and Taylor's formula produce 
\begin{align}\label{UEE2}\nonumber
	&\frac{\alpha}{\delta+1}\big(\partial_x(v^n+\varphi)^{\delta+1},|v^n|^{p-2}v^n\big) \\&\nonumber= -\frac{\alpha(p-1)}{\delta+1}\big((v^n+\varphi)^{\delta+1},|v^n|^{p-2}\partial_xv^n\big)\\&\nonumber = -\frac{\alpha(p-1)}{\delta+1}\big(v^n,|v^n|^{p-2}\partial_xv^n\big)\\&\nonumber\quad-\alpha(p-1)\big(\varphi(\theta_2v^n+(1-\theta_2)\varphi)^{\delta},|v^n|^{p-2}\partial_xv^n\big)\\&=-\alpha(p-1)\big(\varphi(\theta_2v^n+(1-\theta_2)\varphi)^{\delta},|v^n|^{p-2}\partial_xv^n\big),
\end{align}for some $\theta_2\in(0,1)$, where we have used  the fact that $\big(v,|v|^{p-2}\partial_xv\big)=0$. Let us consider the term $I_1$ and estimate it using \eqref{UEE2}, H\"older's and Young's inequalities as 
\begin{align}\label{UEE3}\nonumber
	|I_1| &\leq \alpha(p-1)\||v^n|^{\frac{p-2}{2}}\varphi(\theta_2v^n+(1-\theta_2)\varphi)^\delta\|_{\L^2}\||v^n|^{\frac{p-2}{2}}\partial_x v^n\|_{\L^2} \\& \nonumber\leq 
	\frac{\nu}{2}\||v^n|^{\frac{p-2}{2}}\partial_xv^n\|_{\L^2}^2+\frac{2^{\delta}(p-1)^2\alpha^2}{2\nu}\left\{\||v^n|^{\frac{p+2\delta-2}{2}}\varphi\|_{\L^2}^2+\||v^n|^{\frac{p-2}{2}}|\varphi|^{\delta+1}\|_{\L^2}^2\right\} \\&\nonumber \leq 
	\frac{\nu}{2}\||v^n|^{\frac{p-2}{2}}\partial_xv^n\|_{\L^2}^2+\frac{2^{\delta}(p-1)^2\alpha^2}{2\nu}\|v^n\|_{\L^{p+2\delta}}^{p+2\delta-2}\|\varphi\|_{\L^{p+2\delta}}^2\\&\nonumber\quad+\frac{2^{\delta}(p-1)^2\alpha^2}{2\nu}\|v^n\|_{\L^p}^{p-2}\|\varphi\|_{\L^{p(\delta+1)}}^{2(\delta+1)}\\&\nonumber \leq 
	\frac{\nu}{2}\||v^n|^{\frac{p-2}{2}}\partial_xv^n\|_{\L^2}^2+\frac{\beta}{4}\|v^n\|_{\L^{p+2\delta}}^{p+2\delta}+\frac{2}{p+2\delta}\bigg(\frac{4(p+2\delta-2)}{\beta(p+2\delta)}\bigg)^{\frac{p+2\delta-2}{2}}\\&\nonumber\qquad\times\bigg(\frac{2^{\delta}(p-1)^2\alpha^2}{2\nu}\bigg)^{\frac{p+2\delta}{2}}\|\varphi\|_{\L^{p+2\delta}}^{p+2\delta} +\frac{2^{\delta}(p-1)^2\alpha^2}{4\nu}\|v^n\|_{\L^p}^{p}\\&\quad+\frac{2^{\delta}(p-1)^2\alpha^2}{\nu p}\bigg(\frac{2(p-2)}{p}\bigg)^{\frac{p-2}{2}}\|\varphi\|_{\L^{p(\delta+1)}}^{p(\delta+1)}.
\end{align} Now, we consider the term $I_2$ and estimate it using H\"older's, interpolation and Young's inequalities as 
\begin{align}\label{UEE4}\nonumber
	|I_2|& \leq 2^{\delta}\beta(1+\gamma)\big(|v^n|^{\delta+1}+|\varphi|^{\delta+1},|v^n|^{p-1}\big)\\& \nonumber\leq 2^\delta\beta(1+\gamma)\|v^n\|_{\L^{p+\delta}}^{p+\delta}+2^\delta\beta(1+\gamma)\|\varphi\|_{\L^{p(\delta+1)}}^{\delta+1}\|v^n\|_{\L^p}^{p-1}\\&\nonumber\leq 
	2^\delta\beta(1+\gamma)\|v^n\|_{\L^{p+2\delta}}^{\frac{p+2\delta}{2}}\|v^n\|_{\L^p}^{\frac{p}{2}}+2^\delta\beta(1+\gamma)\|v^n\|_{\L^p}^p+\frac{2^\delta\beta(1+\gamma)}{p}\bigg(\frac{p-1}{p}\bigg)^{p-1}\|\varphi\|_{\L^{p(\delta+1)}}^{p(\delta+1)}\\&\nonumber \leq 
	\frac{\beta}{4}\|v^n\|_{\L^{p+2\delta}}^{p+2\delta}+2^{2\delta}\beta(1+\gamma)^2\|v^n\|_{\L^p}^p+2^\delta\beta(1+\gamma)\|v^n\|_{\L^p}^p\\&\quad+\frac{2^\delta\beta(1+\gamma)}{p}\bigg(\frac{p-1}{p}\bigg)^{p-1}\|\varphi\|_{\L^{p(\delta+1)}}^{p(\delta+1)}.
\end{align}Similarly, we estimate the terms $I_3$ and $I_4$ as 
\begin{align}\label{UEE5}
	|I_3| &\leq \beta\gamma\|\varphi\|_{\L^p}\|v^n\|_{\L^p}^{p-1} \leq \beta\gamma \|v^n\|_{\L^p}^p+\frac{\beta\gamma}{p}\bigg(\frac{p-1}{p}\bigg)^{p-1}\|\varphi\|_{\L^p}^p,\\ \label{UEE6}\nonumber
	|I_4| &\leq 2^{2\delta-1}\beta(2\delta+1)\big(|\varphi||v^n|^{2\delta}+|\varphi|^{2\delta+1},|v^n|^{p-1}\big) \\&\nonumber \leq 2^{2\delta-1}\beta(2\delta+1)\left\{\|\varphi\|_{\L^{p+2\delta}}\|v^n\|_{\L^{p+2\delta}}^{p+2\delta-1}+\|\varphi\|_{\L^{p(2\delta+1)}}^{2\delta+1}\|v^n\|_{\L^p}^{p-1} \right\}\\&\nonumber \leq 
	\frac{\beta}{4}\|v^n\|_{\L^{p+2\delta}}^{p+2\delta}+\frac{1}{p+2\delta}\bigg(\frac{4(p+2\delta-1)}{\beta(p+2\delta)}\bigg)^{p+2\delta-1}\big(2^{2\delta-1}\beta(2\delta+1)\big)^{p+2\delta}\|\varphi\|_{\L^{p+2\delta}}^{p+2\delta}\\&\quad +2^{2\delta-1}\beta(2\delta+1)\|v^n\|_{\L^p}^p.
\end{align}Substituting \eqref{UEE3}-\eqref{UEE6} in \eqref{UEE1} and then using H\"older's inequality, we obtain
\begin{align*}
	&	\frac{1}{p}\frac{\d }{\d t} \|v^n(t)\|_{\L^p}^p +\frac{\nu (p-1)}{2}\||v^n(t)|^{\frac{p-2}{2}}\partial_x v^n(t)\|_{\L^2}^2+\beta\gamma\|v^n(t)\|_{\L^p}^p+\frac{\beta}{4}\|v^n(t)\|_{\L^{p+2\delta}}^{p+2\delta} \\& \leq  \bigg\{   \frac{2^{\delta}(p-1)^2\alpha^2}{4\nu} +2^{2\delta}\beta(1+\gamma)^2+2^\delta\beta(1+\gamma)+2^{2\delta-1}\beta(2\delta+1) \bigg\}\|v^n(t)\|_{\L^p}^p\\&\quad+\bigg\{ \frac{2^\delta\beta(1+\gamma)}{p}\bigg(\frac{p-1}{p}\bigg)^{p-1} +\frac{2^{\delta}(p-1)^2\alpha^2}{\nu p}\bigg(\frac{2(p-2)}{p}\bigg)^{\frac{p-2}{2}}\bigg\}\|\varphi(t)\|_{\L^{p(\delta+1)}}^{p(\delta+1)}\\&\quad+\frac{1}{p+2\delta}\bigg\{	\bigg(\frac{4(p+2\delta-1)}{\beta(p+2\delta)}\bigg)^{p+2\delta-1}\big(2^{2\delta-1}\beta(2\delta+1)\big)^{p+2\delta} +2\bigg(\frac{4(p+2\delta-2)}{\beta(p+2\delta)}\bigg)^{\frac{p+2\delta-2}{2}}\\&\qquad\times\bigg(\frac{2^{\delta}(p-1)^2\alpha^2}{2\nu}\bigg)^{\frac{p+2\delta}{2}}\bigg\}\|\varphi(t)\|_{\L^{p+2\delta}}^{p+2\delta}\\& \leq 
	\frac{\beta}{8}\|v^n(t)\|_{\L^{p+2\delta}}^{p+2\delta}+K_1+K_2\|\varphi(t)\|_{\L^{p(\delta+1)}}^{p(\delta+1)}+K_3\|\varphi(t)\|_{\L^{p+2\delta}}^{p+2\delta},
\end{align*}where $K_1$, $K_2$  and $K_3$ are defined in \eqref{3150}, \eqref{314} and \eqref{315}, respectively.  Integrating the above inequality from $0$ to $t$, we find 
\begin{align}\label{UEE7}\nonumber
	&\|v^n(t)\|_{\L^p}^p +\frac{\nu p(p-1)}{2}\int_0^t\||v^n(s)|^{\frac{p-2}{2}}\partial_x v^n(s)\|_{\L^2}^2\d s+\beta p\gamma\int_0^t\|v^n(s)\|_{\L^p}^p\d s\\&\quad\nonumber+\frac{\beta p}{8}\int_0^t\|v^n(s)\|_{\L^{p+2\delta}}^{p+2\delta}\d s \\& \leq  \|u_0\|_{\L^p}^p+pK_1T+pK_2T\sup_{t\in[0,T]}\|\varphi(t)\|_{\L^{p(\delta+1)}}^{p(\delta+1)}+pK_3T\sup_{t\in[0,T]}\|\varphi(t)\|_{\L^{p+2\delta}}^{p+2\delta},
\end{align}
for all $t\in[0,T]$. The estimates \eqref{3p18} and \eqref{323}  guarantee  that 
\begin{align}\label{WCG1}\nonumber
	\int_0^T\|\partial_tv^n(t)\|_{\H^{-1}}^{\frac{2\delta+2}{2\delta+1}}\d t&\leq 	\int_0^T\|\mathfrak{A}_n(v^n(t))\|_{\H^{-1}}^{\frac{2\delta+2}{2\delta+1}}\d t\nonumber\\&\leq C\int_0^T\bigg[\|\partial_xv^n(t)\|_{\L^2}^{\frac{2\delta+2}{2\delta+1}}+\|v^n(t)+\varphi(t)\|_{\L^{2(\delta+1)}}^{\frac{2(\delta+1)^2}{2\delta+1}}+\|v^n(t)+\varphi(t)\|_{\L^2}^{\frac{2\delta+2}{2\delta+1}}\bigg]\d t\nonumber\\&\quad+C\int_0^T\|v^n(t)+\varphi(t)\|_{\L^{2(\delta+1)}}^{2(\delta+1)}\d t\leq C. 
\end{align}
Using the estimates \eqref{323} and \eqref{WCG1}, and an application of the Banach-Alaoglu theorem ensure the existence of a subsequence of $v^n$ (still denoted by the same symbol) such that 
\begin{equation}\label{3p26}
	\left\{
	\begin{array}{ll}
	v^n\xrightharpoonup{\ast}v&\ \text{ in }\ \mathrm{L}^{\infty}(0,T;\L^2(\1)),\\
	v^n(T)\rightharpoonup v(T)&\ \text{ in }\ \L^2(\1),\\
	v^n\rightharpoonup v&\ \text{ in }\ \mathrm{L}^{2}(0,T;\H_0^1(\1)),\\
	v^n\rightharpoonup v&\ \text{ in }\ \mathrm{L}^{2(\delta+1)}(0,T;\L^{2(\delta+1)}(\1)),\\
	\partial_tv^n\rightharpoonup \partial_tv&\ \text{ in }\ \mathrm{L}^{\frac{2\delta+2}{2\delta+1}}(0,T;\H^{-1}(\1)), \\
	\mathfrak{A}_n(v^n)\rightharpoonup\mathfrak{A}_0&\ \text{ in }\ \mathrm{L}^{\frac{2\delta+2}{2\delta+1}}(0,T;\H^{-1}(\1)). 
	\end{array}
	\right.
\end{equation}
Furthermore, from the energy estimate \eqref{UEE7}, we also have 
\begin{equation}
	v^n\xrightharpoonup{\ast}v\ \text{ in }\ \mathrm{L}^{\infty}(0,T;\L^p(\1)), \ \text{ for all }\ p\geq 2\delta+1.
	\end{equation}
Since the embedding $\H_0^1(\1)\subset\L^2(\1)\subset\H^{-1}(\1)$ is continuous, and the embedding $\H_0^1(\1)\subset\L^2(\1)$ is compact, we use the convergence given in \eqref{3p26} to deduce along Aubin-Lions compactness  lemma (cf. \cite{JSi}) that 
\begin{align}
	v^n\to v\ \text{ in }\ \L^2(0,T;\L^2(\1)). 
\end{align}
Along a further subsequence (still denoted by the same symbol), we also have 
\begin{align}\label{3p28}
	v^n(t,x)\to v(t,x)\ \text{ for a.e. }\ (t,x)\in[0,T]\times[0,1]. 
\end{align}
Using the convergence in \eqref{3p26}, one can pass the limit in \eqref{4p11} to obtain the following system  for a.e. $t\in[0,T]$:
\begin{equation}\label{3p29}
	\left\{
	\begin{aligned}
		\langle\partial_tv(t),\varphi\rangle&=-\langle\mathfrak{A}_0(t),\varphi\rangle,\\
		(v(0),\varphi)&=(u_0,\varphi),
	\end{aligned}
	\right.
\end{equation}
for all $\varphi\in\H_0^1(\1)$. Now, it is left to show that $$\mathfrak{A}(v(t))=\mathfrak{A}_0(t)\ \text{ for a.e. }\ t\in[0,T],$$ in order to obtain a weak solution to the system \eqref{318}. Since $v \in \mathrm{L}^{2}(0,T;\H_0^1(\1))$ and $\partial_tv\in \mathrm{L}^{\frac{2\delta+2}{2\delta+1}}(0,T;\H^{-1}(\1))$ imply that  $t\mapsto\|v(t)\|_{\L^2}^2$ is absolutely continuous with $\frac{1}{2}\frac{\d}{\d t}\|v(t)\|_{\L^2}^2= \langle\partial_tv(t),v(t)\rangle$ for a.e. $t\in[0,T]$ and $v\in\C([0,T];\L^2(\1))$ (Theorem 7.2 and Exercise 8.2, \cite{JCR}) satisfies the following energy equality:
\begin{align}\label{352}
	\|v(t)\|_{\L^2}^2=\|v_0\|_{\L^2}^2-2\int_0^t\langle\mathfrak{A}_0(s),v(s)\rangle\d s, 
\end{align}
for all $t\in[0,T]$. In fact, with the extra regularity given in \eqref{UEE7}, one can easily obtain $\partial_tv\in \mathrm{L}^{2}(0,T;\H^{-1}(\1))$.   Since $\|\mathrm{P}_nu_0-u_0\|_{\L^2}\to 0$ as $n\to\infty$, and using the fact that norm is weakly lower semicontinuous, from \eqref{337} and \eqref{352}, we deduce
\begin{align}\label{3p32}
	\liminf_{n\to\infty}\int_0^T\langle \mathfrak{A}_n(v^n(t)),v^n(t)\rangle\d t&=-\frac{1}{2}\limsup_{n\to\infty}	\|v^n(T)\|_{\L^2}^2+\frac{1}{2}\liminf_{n\to\infty}\|v_0^n\|_{\L^2}^2\nonumber\\&\leq -\frac{1}{2}\|v(T)\|_{\L^2}^2+\frac{1}{2}\|v_0\|_{\L^2}^2=\int_0^T \langle\mathfrak{A}_0(t),v(t)\rangle\d t.
\end{align}
From the local monotonicity condition \eqref{320}, we infer  
\begin{align}\label{3p33}
	&\liminf_{n\to\infty}	\int_0^T\langle\mathfrak{A}(v^n(t))-\mathfrak{A}(u(t)),v^n(t)-u(t)\rangle\d t\nonumber\\&\quad+C\liminf_{n\to\infty}\int_0^T\left(1+\|v^n(t)\|_{\L^{2\delta}}^{4\delta}+\|u(t)\|_{\L^{2\delta}}^{4\delta}\right)\|v^n(t)-u(t)\|_{\L^2}^2\d t\geq 0, 
\end{align}
for some $u\in \mathrm{L}^{\infty}(0,T;\L^2_m(\1))$,  $m<n$. Using the fact that $v^n\in\mathrm{L}^{\infty}(0,T;\L^{p}(\1))$, for all $p\geq 2\delta+1$, we obtain 
\begin{align}\label{3p31}
	&	\int_0^T\left(1+\|v^n(t)\|_{\L^{2\delta}}^{4\delta}+\|u(t)\|_{\L^{2\delta}}^{4\delta}\right)\|v^n(t)-v(t)\|_{\L^2}^2\d t\nonumber\\&\leq CT\sup_{t\in[0,T]}\left(1+\|v^n(t)\|_{\L^{2\delta}}^{4\delta}+\|u(t)\|_{\L^{2\delta}}^{4\delta}\right)\left(\|v^n(t)\|_{\L^2}^2+\|u(t)\|_{\L^2}^2\right)\leq C,
\end{align}
where $C$ is independent of $n$ (cf. \eqref{UEE7}). Thus, using the above estimate and  the convergence \eqref{3p28}, from the Dominated Convergence Theorem (DCT), we deduce 
\begin{align}\label{3p34}
	&	\lim_{n\to\infty}\int_0^T\left(1+\|v^n(t)\|_{\L^{2\delta}}^{4\delta}+\|u(t)\|_{\L^{2\delta}}^{4\delta}\right)\|v^n(t)-u(t)\|_{\L^2}^2\d t\nonumber\\&= \int_0^T\left(1+\|v(t)\|_{\L^{2\delta}}^{4\delta}+\|u(t)\|_{\L^{2\delta}}^{4\delta}\right)\|v(t)-u(t)\|_{\L^2}^2\d t.
\end{align} 
Therefore, using \eqref{3p32} and \eqref{3p34} in  \eqref{3p33}, we further have 
\begin{align}\label{3p35}
&	\int_0^T\langle\mathfrak{A}_0(t)-\mathfrak{A}(u(t)),v(t)-u(t)\rangle\d t\nonumber\\&\quad+\int_0^T\left(1+\|v(t)\|_{\L^{2\delta}}^{4\delta}+\|u(t)\|_{\L^{2\delta}}^{4\delta}\right)\|v(t)-u(t)\|_{\L^2}^2\d t\geq 0.
\end{align}
The above estimate holds for any $u\in \L^{\infty}(0,T;\L^2_m(\1)),$ for any $m\in\mathbb{N}$. It is clear by a density argument that the above inequality remains true for any $u\in \L^{\infty}(0,T;\L^p(\1))\cap\mathrm{L}^2(0,T;\H_0^1(\1))\cap\mathrm{L}^{2(\delta+1)}(0,T;\L^{2(\delta+1)}(\1))$, for $p\geq 2\delta+1$.  

Let us now take $v=u+\lambda w$, $\lambda>0$, where  $w\in \L^{\infty}(0,T;\L^p(\1))\cap\mathrm{L}^2(0,T;\H_0^1(\1))\cap\mathrm{L}^{2(\delta+1)}(0,T;\L^{2(\delta+1)}(\1)),$ for $p\geq 2\delta+1$,   and substitute it in \eqref{3p35} to find 
\begin{align*}
	&	\int_0^T\langle\mathfrak{A}_0(t)-\mathfrak{A}(v(t)-\lambda w(t)),w(t)\rangle\d t\nonumber\\&\quad+\lambda\int_0^T\left(1+\|v(t)\|_{\L^{2\delta}}^{4\delta}+\|v(t)-\lambda w(t)\|_{\L^{2\delta}}^{4\delta}\right)\|w(t)\|_{\L^2}^2\d t\geq 0, 
\end{align*}
Letting $\lambda\to 0$  and then by using the hemicontinuity property \eqref{DC5}  of $\mathfrak{A}(\cdot)$, we obtain 
\begin{align}
	\int_0^T\langle\mathfrak{A}_0(t)-\mathfrak{A}(v(t)),w(t)\rangle\d t\geq 0, 
\end{align}
where we have performed a calculation similar to \eqref{3p31}. Since $w$ is arbitrary, we deduce 
$$\mathfrak{A}_0(t)=\mathfrak{A}(v(t)), \ \text{ for a.e. } t\in[0,T],$$  as required.  It should be  noted that  an application of interpolation inequality yields 
\begin{align}\label{3p61}
	&	\int_0^T\left(1+\|v(t)\|_{\L^{2\delta}}^{4\delta}+\|w(t)\|_{\L^{2\delta}}^{4\delta}\right)\|w(t)\|_{\L^2}^2\d t\nonumber\\&\leq \sup_{t\in[0,T]}\|w(t)\|_{\L^2}^2\int_0^T\left(1+\|v(t)\|_{\L^{2\delta+p}}^{\frac{4\delta^2-p^2}{\delta}}\|v(t)\|_{\L^p}^{\frac{p^2}{\delta}}+\|w(t)\|_{\L^{2\delta+p}}^{\frac{4\delta^2-p^2}{\delta}}\|w(t)\|_{\L^p}^{\frac{p^2}{\delta}}\right)\d t\leq C, 
\end{align}
for all $\delta\leq p\leq 2\delta$.  Therefore, for $\delta\in[1,2]$, one can obtain the existence of a weak solution for  any $u_0\in\L^2(\1)$, while for $2<\delta<\infty$, a weak solution exists for any $u_0\in\L^p(\1)$ for $p\geq \delta$. 
Hence, the evolution equation \eqref{316} has a weak solution $v\in\C([0,T];\L^2(\1)),$  such that the energy equality \eqref{352} is satisfied by $v(\cdot)$. That is, 
\begin{align}
	&\|v(t)\|_{\L^2}^2+2\nu\int_0^t\|\partial_xv(t)\|_{\L^2}^2\d s+2\beta\gamma\int_0^t\|v(s)+\varphi(s)\|_{\L^2}^2\d s+2\beta\int_0^t\|v(s)+\varphi(s)\|_{\L^{2(\delta+1)}}^{2(\delta+1)}\d s\nonumber\\&=\|v_0\|_{\L^2}^2+2\alpha\int_0^t(v(s)+\varphi(s))^{\delta}\partial_x(v(s)+\varphi(s))\varphi(s)\d s\nonumber\\&\quad+2\beta(1+\gamma)\int_0^t(v(s)+\varphi(s))^{\delta+1}(v(s)+\varphi(s))\d s-2\beta\int_0^t\mathfrak{c}(v(s)+\varphi(s))\varphi(s)\d s, 
\end{align}
for all $t\in[0,T]$.

\vskip 0.2cm 
\noindent\textbf{Step 4:} 
The uniqueness of weak solution can be established in the following way: If $v_1(\cdot)$ and $v_2(\cdot)$ are two weak solutions of \eqref{316} with the same initial data $u_0$, then from \eqref{318}, we infer  for all $t\in[0,T]$ that 
\begin{align*}
		(v_1(t)-v_2(t),\zeta)=-\int_0^t\langle\mathfrak{A}(v_1(s))-\mathfrak{A}(v_2(s)),\zeta\rangle\d s,
\end{align*}
for all $\zeta\in\H_0^1(\1)$.  Since the energy equality \eqref{352} holds true, by using \eqref{320},  we further have for all $t\in[0,T]$
\begin{align}\label{361}
	&\|	v_1(t)-v_2(t)\|_{\L^2}^2\nonumber\\&=-2\int_0^t\langle\mathfrak{A}(v_1(s))-\mathfrak{A}(v_2(s)),v_1(s)-v_2(s)\rangle\d s\nonumber\\&\leq-\nu\int_0^t\|v_1(s)-v_2(s)\|_{\H_0^1}^2\d s-2\beta\gamma\int_0^t\|v_1(s)-v_2(s)\|_{\L^2}^2\d s-\frac{\beta}{2^{2\delta}}\int_0^t\|v_1(s)-v_2(s)\|_{\L^{2(\delta+1)}}^{2(\delta+1)}\d s\nonumber\\&\quad+2\int_0^t\bigg\{\tilde{L}+\frac{27\alpha^4}{32\nu^3}\big(\|v_1(s)+\varphi(s)\|_{\L^{2\delta}}+\|v_2(s)+\varphi(s)\|_{\L^{2\delta}}\big)^{4\delta}\bigg\}\|v_1(s)-v_2(s)\|_{\L^2}^2\d s, 
\end{align}
where $\tilde{L}=2^{2\delta-1}\beta(1+\gamma)^2(\delta+1)^2$. Using the fact that $v_1,v_2\in \L^{\infty}(0,T;\L^p(\1))$, for $p\geq 2\delta+1$, uniqueness follows immediately by an application of Gronwall's inequality in \eqref{361}.  Note that using the estimate \eqref{3p61},  the uniqueness of weak solution holds if $u_0\in\L^2(\1)$ for $\delta\in[1,2]$, and if  $u_0\in\L^p(\1)$ for any $p\geq \delta,$ for $2<\delta<\infty$. 
By the uniqueness of weak solution of \eqref{316}, we obtain that $v$  is also the unique solution to \eqref{4.13}. 
\end{proof}

\begin{theorem}\label{theorem4.3}
	Let us assume that $u_0\in \L^p(\1),$ for $p\geq  2\delta+1$. Then there exists a unique $\L^p(\1)$-valued $\mathscr{F}_t$-adapted continuous process $u(\cdot)$ satisfying \eqref{4.02}  such that 
	\begin{align}\label{411}
		\E\left(\sup_{t\in[0,T]}\|u(t)\|_{\L^p}^p\right)\leq C(T).
	\end{align}
\end{theorem}
\begin{proof}
{Let us first prove  that Proposition \ref{theorem4.2} and Corollary \ref{cor3.4}   provide a proof of the uniqueness of mild solution to \eqref{1.01}. Indeed, suppose that $u_1$ and $u_2$  are  two mild solutions to \eqref{1.01} in the interval $[0,T]$. One can easily see that $u_1$ and	$u_2$ are solutions of the problem \eqref{4.05}  in the stochastic interval $[0,\tau^n\wedge\sigma^n]$, for	every positive number $n$, where 
	\begin{align*}
		\tau^n:=\inf_{t\geq 0}\left\{t:\|u_1(t)\|_{\L^p}\geq n\right\}\wedge T\ \mbox{ and }\ \sigma^n:=\inf_{t\geq 0}\left\{t:\|u_2(t)\|_{\L^p}\geq n\right\}\wedge T. 
	\end{align*}
 Corollary \ref{cor3.4}  implies that $u_1(t)=u_2(t)$ in the stochastic interval $[0,\tau^n\wedge\sigma^n]$. Therefore, 
 we deduce that almost surely $u_1(t,x) = u_2(t,x) $ for a.e. $x\in[0,1] $,  for all $t\in[0,T]$, since $\mathbb{P}\left(\tau^n\wedge\sigma^n<T\right)\to 0$ as $n\to\infty$.  
 }

In order to prove global existence, we follow similar steps as in section 3, \cite{GDDG} and proof of Theorem 2.3, \cite{IGDN}.   For  the existence of a mild solution to \eqref{1.01},  we first consider the solution $u^n(\cdot,\cdot)$ of the truncated integral equation \eqref{4.05} for any $n\geq0$. Define a sequence of stopping times
\begin{align*}
	\tau^n:=\inf_{t\geq 0}\{t:\|u^n(t)\|_{\L^p}\geq n\}\wedge T,
\end{align*}
for $n\in\mathbb{N}$. {Since $u^n(t,\cdot)$ is the unique mild solution of \eqref{4.02}, for any $t\in[0,\tau^n]$ (Corollary \ref{cor3.4}) ,}  it follows that  for $m\geq n$,  $\tau^n\leq \tau^m$ and $u^m(t,\cdot) = u^n(t,\cdot), $  for $t\leq \tau^n$. Therefore we can define a process $u$ by setting $u(t)=u^n(t)$ for $t\in[0,\tau^n]$. Let us set $\tau^\infty=\sup\limits_{n\geq 0}\tau^n$. Then Proposition \ref{theorem4.2}  tells us that we have constructed a solution to \eqref{4.05}  in the random interval $[0, \tau^{\infty})$, and it is unique by Corollary \ref{cor3.4}. To conclude, we just need to show that $\tau^{\infty}=T$, $\mathbb{P}$-a.s., that is, $\mathbb{P}(\tau^{\infty}<T)=0$.

Let us consider  the stochastic process
\begin{align}\label{4.11}
	\varphi(t,x)= \int_{0}^{t}\int_{0}^{1}G(t-s,x,y)g(s,y,u(s,y))\chi_{\{t<\tau^{\infty}\}}\W(\d s,\d y), 
\end{align}where the function $g(\cdot,\cdot,\cdot)$ satisfies the Hypothesis \ref{H1} and $(u(\cdot),\tau^{\infty})$ is the solution to the integral equation \eqref{4.02}. Also for any $t\leq \tau^n,$ we have 
\begin{align*}
\varphi(t,x)=\varphi^n(t,x)=\int_{0}^{t}\int_{0}^{1}G(t-s,x,y)g(s,y,u^n(s,y))\W(\d s,\d y).
\end{align*} Since $g(\cdot,\cdot,\cdot)$  is uniformly bounded in $n$ (Hypothesis \ref{H1}), 
using Corollary 4.3, \cite{IG}, for every $p\geq 1, T>0$ there exists a constant $C>0$ such that 
   \begin{align*}
   	\E\big(|\varphi(t,x)-\varphi(s,y)|^{2p}\big) \leq C\big(|t-s|^{\frac{1}{4}}+|x-y|^{\frac{1}{2}}\big)^{2p},
   \end{align*}for all $s,t\in[0,T]$ and $x,y\in[0,1]$. Hence by Kolmogorov's continuity theorem  (see \cite{JBW}), there exists a version of $\varphi(\cdot,\cdot)$ with $\P$-a.s. trajectories being H\"older continuous  with exponents smaller than $\frac{1}{4}-\frac{1}{2p}$ (in time)  and $\frac{1}{2}-\frac{1}{2p}$ (in space). From Lemma \ref{lemma3.5}, we infer that if $p\geq 2$ and $q>4$
\begin{align}\label{366}
\sup_{n\geq 1}	\E\left[\sup_{t\in[0,T]}\|\varphi^n(t)\|_{\L^p}^q\right]\leq C\sup_{n\geq 1}\int_0^T\E\left[\|\varphi^n(t)\|_{\L^p}^q\right]\d t\leq CK^qT,
\end{align}
by using  Hypothesis \ref{H1}. 

 Let us now show that $\tau^\infty=T,$ $\mathbb{P}$-a.s., or  in other words $\P\left\{\omega\in\Omega:\tau^\infty(\omega)=T\right\}=1$. Let $u(\cdot,\cdot)$ be a (possibly exploding) solution  of \eqref{1.01} and set $v(\cdot,\cdot)=u(\cdot,\cdot)-\varphi(\cdot,\cdot)$. Since $\varphi$ satisfies the assumptions of Proposition \ref{lem3.4}, then $v(\cdot,\cdot)$ verifies the integral equation given in \eqref{4.13}.    From the estimates \eqref{3.13} and  \eqref{366}, we have 
 \begin{align*}
 \E\left(\sup_{t\in[0,T]}\|u^n(t)\|_{\L^p}^p\right)&\leq 2^{p-1}\left[\|u_0\|_{\L^p}^p+pK_1T+pK_2T\E\left(\sup_{t\in[0,T]}\|\varphi^n(t)\|_{\L^{p(\delta+1)}}^{p(\delta+1)}\right)\right.\nonumber\\&\qquad\left.+pK_3T\E\left(\sup_{t\in[0,T]}\|\varphi^n(t)\|_{\L^{p+2\delta}}^{p+2\delta}\right)+\E\left(\sup_{t\in[0,T]}\|\varphi^n(t)\|_{\L^p}^p\right)\right]\nonumber\\&\leq C\left(\|u_0\|_{\L^p}^p+1\right),
 \end{align*} for all $p\geq 2\delta+1$. Since $\tau^n\leq \tau^{\infty}$, we have 
\begin{align*}
\P(\tau^{\infty}<T)\leq \P(\tau^n\;{<}\;T) = \P \left(\sup_{t\in[0,T]}\|u^n(t)\|_{\L^p}^p \geq n^p\right),
\end{align*}  and an application of Markov's inequality yields 
\begin{align*}
\P(\tau^n\;{<}\; T) \leq \frac{1}{n^p}\E\left(\sup_{t\in[0,T]}\|u^n(t)\|_{\L^p}^p\right)  \leq \frac{C\left(\|u_0\|_{\L^p}^p+1\right)}{n^p},
\end{align*}for some constant $C=C(\alpha,\beta,\gamma,\delta,\nu,p,T,K)$. Taking $n\to\infty$, we obtain $\tau^\infty=T,\;\P$-a.s. Hence, we obtain the global existence  and uniqueness of the mild solution of the integral equation \eqref{4.02}.
\end{proof}
\begin{remark}
	Using the existence and uniqueness of solution and Lemmas \ref{lemma3.3}, \ref{lemma3.5}, we conclude that $u(\cdot)$ (solution of SGBH equation \eqref{1.01}) has a version having continuous paths on $(0,T]$ with values in $\W^{\frac{1+\e}{p},p}(\1),$ for each $p> \max\big\{6,2\delta+1\big\}$ and $0<\e<\frac{p}{2}-3$.
\end{remark}

	\section{Comparison Theorem}\label{sec5}\setcounter{equation}{0}
	In this section, we discuss a comparison result, which plays  an important role in the proof of density for solutions of the system \eqref{1.01}-\eqref{1.1c}. A positivity result for the solutions of deterministic GBH equation has been obtained in Theorem 3.4, \cite{MTMAK}. A good number of works are available in the literature regarding comparison theorem for several SPDEs, we refer  the readers to\cite{MP,IG,CGRM,PK,RMCS}, etc. and the references therein.

	\begin{theorem}[Comparison theorem]\label{theorem5.2} Assume that two initial data $u_0,v_0\in \L^p(\1)$, for $p\geq2\delta+1$ with the condition $u_0(x)\leq v_0(x)$ for a.e. $x\in[0,1]$. Denote the unique solution $u(t,x)$ (resp. $v(t,x)$) of the SGBH equation  \eqref{1.01} corresponding to the initial data $u_0(x)$ (resp. $v_0(x)$) in the interval $[0,T]$. Then almost surely $u(t,x)\leq v(t,x),$ for all $t\in[0,T]$ for a.e. $x\in[0,1]$. 
	\end{theorem}

Before going to the proof of comparison theorem, we first establish the following technical lemma, a variant of  Lemma 4.6, \cite{IG} and for completeness, we provide a proof here:
	\begin{lemma}\label{lemma5.1}
	Let $h_n=\{h_n(r):r\in\R\}$ be any  random field for every integer $n\geq 1$ such that for every $M\geq 0$ 
	\begin{align}\label{5.001}
		\int_{0}^{T}\int_{0}^{1}\sup_{|r|\leq M} |h_n(r)-h(r)|\d x\d t \to 0,  \ \P\text{-a.s.},
	\end{align} as $n\to\infty$, where $h$ is some random field.  Assume moreover that for some constant $C$
	\begin{align*}
		|h_n(r)|+|h(r)|& \leq C\big(|r|+|r|^{\delta+1}+|r|^{2\delta+1}\big),\\
		|h_n(r)-h_n(p)| &\leq C\big(1+(|r|^\delta+|p|^\delta)+(|r|^{2\delta}+|p|^{2\delta})\big)|r-p|,
	\end{align*}for $r,p\in\R$.
	
	Let $u_n=\{u_n(t,x):(t,x)\in[0,T]\times[0,1] \},$ $n\geq 1$, be a random field such that $\P$-a.s.,
	\begin{equation}\label{5.002}
		\left\{
		\begin{aligned}
			u_n&\xrightharpoonup{\ast}u\ \text{ in }\ \L^{\infty}(0,T;\L^2(\1)),\\
			u_n&\xrightharpoonup{}u\ \text{ in }\ \L^{2}(0,T;\H_0^1(\1)),\\
			u_n&\xrightharpoonup{}u\ \text{ in }\ \L^{2(\delta+1)}(0,T;\L^{2(\delta+1)}(\1)).
		\end{aligned}\right.
	\end{equation}
		Then, we have 
	\begin{align*}
		I:=	\int_{0}^{T}\int_{0}^{1}|h_n(u_n(t,x))-h(u(t,x))|\d x\d t \to 0, \ \P\text{-a.s.},\ \text{ as }\ n\to\infty. 
	\end{align*}
\end{lemma}
\begin{proof}
	One can easily see that 
	\begin{align}\label{5.003}\nonumber
		I&\leq   \int_{0}^{T}\int_{0}^{1} \chi_{\{|u_n|\leq M\}}|h_n(u_n)-h(u_n)|\d x\d t+ \int_{0}^{T}\int_{0}^{1} \chi_{\{|u_n|> M\}}|h_n(u_n)-h(u_n)|\d x\d t \\&\quad+\int_{0}^{T}\int_{0}^{1}|h(u_n)-h(u)|\d x \d t.
	\end{align}For the penultimate term in the right hand side of above inequality, we compute 
	\begin{align*}
		I_1:&=	\int_{0}^{T}\int_{0}^{1} \chi_{\{|u_n|> M\}}|h_n(u_n)-h(u_n)|\d x\d t \\&\leq  \int_{0}^{T}\big(|u_n|+|u_n|^{\delta+1}+|u_n|^{2\delta+1}\big)\d x\d t<+\infty,
	\end{align*}
and $\lim\limits_{M\to\infty}\chi_{\{|u_n|> M\}}=0$. Since $M>0$ is arbitrary, an application of DCT  yields that $I_1\to 0$ as $M\to\infty$. Since the embedding of $\H_0^1(\1)\subset\L^2(\1)$ is compact, the convergence given in \eqref{5.002} and an application of Aubin-Lions compactness lemma imply that 
\begin{align*}
	u_n\to u\ \text{ in }\ \mathrm{L}^2(0,T;\L^2(\1)), \ \mathbb{P}\text{-a.s.}
	\end{align*}
Using the assumptions on $h(\cdot)$  and H\"older's inequality in the final term, we obtain 
	\begin{align*}
		&\int_{0}^{T}\int_{0}^{1} |h(u_n)-h(u)|\d x\d t \\& \leq C\int_{0}^{T}\left\{\big(\|u_n\|_{\L^{2\delta}}+\|u\|_{\L^{2\delta}}\big)^\delta\|u_n-u\|_{\L^2}+\|u_n-u\|_{\L^2}\right.\\&\qquad\left.+\big(\|u_n\|_{\L^{2(\delta+1)}}+\|u\|_{\L^{2(\delta+1)}}\big)^{2\delta} \|u_n-u\|_{\L^2}^{\frac{1}{\delta}}\|u_n-u\|_{\L^{2(\delta+1)}}^{\frac{\delta-1}{\delta}}  \right\}\d t \\& \leq  C\bigg\{\bigg[\bigg(\int_{0}^{T}\big(\|u_n\|_{\L^{2\delta}}+\|u\|_{\L^{2\delta}}\big)^{2\delta} \d t\bigg)^{\frac{1}{2}}+T^{\frac{1}{2}}\bigg]\bigg(\int_{0}^{T}\|u_n-u\|_{\L^2}^2\d t\bigg)^{\frac{1}{2}}\\&\qquad+\bigg(\int_{0}^{T}\|u_n-u\|_{\L^2}^2\d t\bigg)^{\frac{1}{2\delta}}\bigg(\int_{0}^{T} \big(\|u_n\|_{\L^{2(\delta+1)}}^{2(\delta+1)}+\|u\|_{\L^{2(\delta+1)}}^{2(\delta+1)}\big)\d t\bigg)^{\frac{2\delta-1}{2\delta}}\bigg\}.
	\end{align*}
	By using the above estimate, \eqref{5.001}  and \eqref{5.002} in \eqref{5.003}, one can complete the proof. 
\end{proof}

	\begin{proof}[Proof of Theorem \ref{theorem5.2} ] Proof of this Theorem is divided in the following steps.
		\vskip 0.2 cm
		\noindent 
 {\bf Step 1:}
		By  Theorem \ref{theorem4.3}, SGBH equation \eqref{1.01} with the initial conditions $u_0$ and $v_0$ admits  unique mild solutions $u$ and $v,$ respectively, which belongs to the space $\C([0,T];\L^p(\1))$ for $p\geq  2\delta+1$.
	
	Let us define a cutoff function for any $x\in\R$ by
	\begin{equation}\label{53}
		\eta_n(x)=\begin{cases}
			1, &\text{ if }\ x\leq n,\\
			n+1-x, &\text{ if } \ n<x\leq n+1\\
			0, &\text{ if }\  x> n+1.
		\end{cases},
	\end{equation} Define  nonlinear operators with the help of above cutoff function as 
\begin{align}\label{5.01}
\p_n(u)=\eta_n(|u|)\p(u) \ \text{ and }\  \c_n(u)=\eta_n(|u|)\c(u).
\end{align}
		Let $\{\phi_k\}$ be an orthonormal basis of $\L^2(\1)$ such that $\phi_k$ is bounded uniformly for $k\geq 1$. Define \begin{align*}	
		\W^k(t) := \int_{0}^{t}\int_{0}^{1}\phi_k(x)\W( \d s, \d x),\end{align*}where $\W(\cdot,\cdot)$ is an $\mathscr{F}_t$-adapted Brownian sheet. Then $\W^k(\cdot)$, $k\geq 1$ are independent $\mathscr{F}_t$-adapted Wiener processes.
		Fix $n$ and consider the evolution equation 
		\begin{align}\label{5.02}
			\d u_n (t)= A^n(t,u_n(t))\d t + \sum_{k=1}^{n}\mathcal{B}^k(t,u_n(t))\d \W^k(t), \; u_n(0)=u_0,
		\end{align}in the triplet $\H_0^1(\1)\hookrightarrow \L^2(\1)\hookrightarrow\H^{-1}(\1)$. The nonlinear operators $A^n$ and $\mathcal{B}^k$  form $\H_0^1(\1)$ into $\H^{-1}(\1)$ and into $\L^2(\1)$, respectively, are defined by 
		\begin{align*}
			\langle A^n(t,\psi),\phi \rangle &:= -\nu \int_{0}^{1}\psi'(x)\phi'(x)\d x +\beta\int_{0}^{1}\c_n(\psi)(t,x)\phi(x)\d x \\&\quad+\frac{\alpha}{\delta+1}\int_{0}^{1}\p_n(\psi)(t,x)\phi'(x)\d x, 	\\
			(\mathcal{B}^k(t,\psi),h) &:= \int_{0}^{1}g(t,x,\psi(x))\phi_k(x)h(x)\d x,
		\end{align*}for all $\psi,\phi\in\H_0^1(\1),\; h\in \L^2(\1)$. Here $\langle\cdot,\cdot\rangle$ and $(\cdot,\cdot)$ denote the duality pairing between $\H_0^1(\1)$ and $\H^{-1}(\1)$ and scalar product in $\L^2(\1),$ respectively. Then $u_n(\cdot)$ is called a mild solution of \eqref{5.02} if it satisfies the following integral equation:\begin{align*}
		u_n(t)&=G(t,x;u_0)+\frac{\alpha}{\delta+1}\int_{0}^{t}\int_{0}^{1}\frac{\partial G}{\partial y}(t-s,x,y)\p_n(u_n(s,y))\d y\d s \\&\quad+\beta\int_{0}^{t}\int_{0}^{1}G(t-s,x,y)\c_n(u_n(s,y))\d y\d s \\&\quad+\sum_{k=1}^{n}\int_{0}^{t}\int_{0}^{1}G(t-s,x,y)g(s,y,u_n(s,y))\phi_k(y)\d y\d \W^k( s),
	\end{align*} $\P$-a.s., for all $t\in[0,T]$. Using fixed point arguments, as we have done in Theorem \ref{theorem4.3},  one can obtain the unique mild solution $u_n \in \C([0,T];\L^p(\1)),$ for $p\geq 2\delta+1$ of the above integral equation.

	Let $v_n$ denote the solution of \eqref{5.02} with the initial data $v_0$ in place of $u_0$. Setting $w_n=u_n-v_n$, our aim is to show that $\mathbb{P}$-a.s., for all $ t\in [0,T]$, 
		\begin{align}\label{5.03}
			|w_n(t,x)|_+ =0, \text{ for a.e. } x \in [0,1],
		\end{align}here $|\cdot|_+ = \max \{\cdot,0\}$. In order to show \eqref{5.03}, the ideas have been borrowed   from the works \cite{MP,IG}, especially the construction of a functional $\Psi$ defined below (cf. \cite{MP}) and the final part of the proof (cf. \cite{IG}). For every integer $k\geq 1$, we define a functional $\Psi:\L^2(\1)\to \R$ by $\Psi_k(h):=\int_{0}^{1}\psi_k(h(x))\d x $, where
		\begin{align*}
			\psi_k(x):= \chi_{\{x\geq 0\}} \int_{0}^{x}\int_{0}^{y}\rho_k(z)\d z \d y, 
		\end{align*} and 
		\begin{equation*}\rho_k(z):=	\left\{
			\begin{aligned}
				&2kz, &&\text{ for } z\in \bigg[0,\frac{1}{k}\bigg],\\
				&2\chi_{\{z\geq 0\}}, &&\text{ for } z\notin \bigg[0,\frac{1}{k}\bigg].	 
			\end{aligned}
			\right.
		\end{equation*}
	Then, one can show that (cf. \cite{MP})
	\begin{itemize}
		\item [$\bullet$] $\psi_k \in \C^2(\R)$,
		\item [$\bullet$] $0 \leq \psi_k'(x)\leq 2|x|_+$ \ and \ $0 \leq \psi''_k(x)\leq 2 \chi_{\{x\geq 0\}}$, 
		\item [$\bullet$] $\psi_k(x)\to |x|_+^2$\  for \ $k\to \infty$. 
		\end{itemize}
		Also the function $\Psi_k$ is twice Fr\'echet differentiable at every $h \in \L^2(\1)$, the first derivative $\Psi'_k(h)$ is a continuous linear functional on $\L^2(\1)$ and the second derivative $\Psi''_k(h)$ is a continuous symmetric bilinear form on $\L^2(\1)\times \L^2(\1)$ given by  
		\begin{align*}
			(\Psi'_k(h),h_1)= \int_{0}^{1}\psi'_k(h(x))h_1(x)\d x \ \text{ and } \ \Psi''_k(h)(h_1,h_2)= \int_{0}^{1}\psi''_k(h(x))h_1(x)h_1(x)\d x, 
		\end{align*}respectively for $h_1,h_2\in \L^2(\1)$.
	
		Applying It\^o's formula, we find
		\begin{align}\label{5.04}
			\Psi_k(w_n(t))= \Psi_k(w_n(0))+\int_{0}^{t}\mathtt{A}(s)\d s +\frac{1}{2}\int_{0}^{t}\mathtt{B}(s)\d s +M_{n}(t),
		\end{align}where \begin{align*}
			\mathtt{A}(s)&:= \langle A^n(s,u_n)-A^n(s,v_n), \psi_k' (w_n(s))\rangle,    \\
			\mathtt{B}(s)&:= \sum_{i=1}^{n}\big(\psi''_k(w_n(s))(g(u_n(s))-g(v_n(s)))\phi_i,(g(u_n(s))-g(v_n(s))\phi_i)\big),
		\end{align*} and $M_n(t)$ is a continuous local martingale starting from zero given by 
		\begin{align*}
			M_n(t)=\sum_{i=1}^{n}\int_{0}^{t}\big(\psi_k'(w_n(s)),(g(u_n(s))-g(v_n(s)))\phi_i\big)\d \W^i(s).\end{align*}  One can re-write $\mathtt{A}$  as
		\begin{align}\label{5.05}
			\mathtt{A}(s)= -	\nu\mathcal{A}^{(1)}(s)+\beta \mathcal{A}^{(2)}(s)+\frac{\alpha}{\delta+1}	\mathcal{A}^{(3)}(s),
		\end{align}where \begin{align}\label{5.06} \nonumber
			\mathcal{A}^{(1)}(s)&:= \bigg(\frac{\partial}{\partial x}w_n(s,x),\frac{\partial}{\partial x}(\psi'_k(w_n(s,x)))\bigg) \\&\;\nonumber= \int_{0}^{1}\bigg(\frac{\partial}{\partial x}w_n(s,x)\bigg)^2 \psi''_k(w_n(s,x))\d x \geq 0,\\ \nonumber
			\mathcal{A}^{(2)}(s) &:=\big((\c_n(u_n)-\c_n(v_n))(s,x),\psi_k'(w_n(s,x))\big) \nonumber\\&\; =
			\int_{0}^{1}  \big((\c_n(u_n)-\c_n(v_n))(s,x)\big)\psi_k'(w_n(s,x))\d x,
			\nonumber	\\ 
			\mathcal{A}^{(3)}(s)&:= \bigg((\p_n(u_n)-\p_n(v_n))(s,x),\frac{\partial }{\partial x}\psi'_k(w_n(s,x))\bigg) \nonumber\\& \;
			=\int_{0}^{1}(\p_n(u_n)-\p_n(v_n))(s,x)\frac{\partial}{\partial x}w_n(s,x) \psi''_k(w_n(s,x))\d x,
		\end{align}
for a.e. $s\in[0,T]$.	Without loss of generality, one may assume that $|v_n|\leq |u_n|$. Using the definition of the cutoff function given in \eqref{53},  we divide the further calculations into the following six cases. 
	\vskip 0.2 cm
	\noindent 
		{\bf Case 1:} If $|u_n|\leq n$ and $|v_n|\leq n,$ then $\p_n(\cdot)=\p(\cdot)$ and $\c_n(\cdot)=\c(\cdot)$. 
	\begin{align*}
		\mathcal{A}^{(2)}&= \int_{0}^{1}\left[ (1+\gamma)(u_n^{\delta+1}-v_n^{\delta+1})(x)-\gamma(u_n-v_n)(x)-(u_n^{2\delta+1}-v_n^{2\delta+1})(x)\right]\psi_k'(w_n(x))\d x.
	\end{align*}
	Using Taylor's formula and boundedness of $\psi_k'$, we estimate the above term as 
	\begin{align}\label{5.07}\nonumber
		\mathcal{A}^{(2)} &\leq 2\int_{0}^{1}\left[2^{\delta}n^{\delta}(\gamma+1)(\delta+1)|(u_n-v_n)(x)|+\gamma|(u_n-v_n)(x)|\right.\nonumber\\&\qquad\left.+2^{2\delta} n^{2\delta}(2\delta+1)|(u_n-v_n)(x)|\right]|(u_n-v_n)(x)|_+\d x 
		\nonumber\\&\leq 2\big(2^{\delta}n^{\delta}(\gamma+1)(\delta+1)+\gamma+2^{2\delta} n^{2\delta}(2\delta+1)\big)\int_{0}^{1}|w_n(x)|_+^2\d x.
	\end{align}Again, using Taylor's formula, Young's inequality and boundedness of $\psi_k''$, we estimate the term $\mathcal{A}^{(3)}$ as
	\begin{align}\label{5.08}\nonumber
		\mathcal{A}^{(3)} & \leq 2^{\delta} n^{\delta}(\delta+1)\int_{0}^{1}|(u_n-v_n)(x)|\psi_k''(w_n(x))\bigg|\frac{\partial}{\partial x}(w_n(x))\bigg|\d x\\& \leq 
		\frac{\nu}{4}\int_{0}^{1}\psi_k''(w_n(x))\bigg(\frac{\partial}{\partial x}(w_n(x))\bigg)^2\d x+\frac{2^{2\delta+1}n^{2\delta}(\delta+1)^2}{\nu}\int_{0}^{1}|w_n(x)|_+^2\d x.
	\end{align}
	\vskip 0.2 cm
	\noindent 
	{\bf Case 2:} If $|u_n|> n+1,\; |v_n|> n+1$ then $\p_n(\cdot)=\c_n(\cdot)=0$. 
	\vskip 0.2 cm
	\noindent 
	{\bf Case 3:} If $n< |u_n|\leq n+1,\; n<|v_n|\leq n+1,$ then $\p_n(\cdot)=(n+1-|\cdot|)\p(\cdot)$ and $\c_n(\cdot)=(n+1-|\cdot|)\c(\cdot)$. Using H\"older's inequality and boundedness of $\psi_k'$ in the term $\mathcal{A}^{(2)}$, we obtain
	\begin{align}\label{5.09} \nonumber
		\mathcal{A}^{(2)}&= \big((n+1-|u_n|)\c(u_n)-(n+1-|v_n|)\c(v_n),\psi_k'(w_n)\big) \\& \nonumber
		= (n+1)\big(\c(u_n)-\c(v_n),\psi_k'(w_n)\big)+\big(-|u_n|\c(u_n)+|v_n|\c(v_n),\psi_k'(w_n)\big) \\&\nonumber 
		= (n+1)\big(\c'(\theta u_n+(1-\theta) v_n)(u_n-v_n),\psi_k'(w_n)\big) - \big((|u_n|-|v_n|)\c(u_n),\psi_k'(w_n)\big)\\&\nonumber\quad -\big(|v_n|\c'(\theta u_n+(1-\theta)v_n)(u_n-v_n),\psi_k'(w_n)\big)
		\\& 
		\leq C(\gamma,\delta,n)\int_{0}^{1}|w_n(x)|_+^2\d x.
	\end{align}Similarly, using Young's inequality and boundness of $\psi_k''$, we compute $\mathcal{A}^{(3)} $ as
	\begin{align}\label{5.10}\nonumber
		\mathcal{A}^{(3)}&=\bigg((n+1-|u_n|)u_n^{\delta+1}-(n+1-|v_n|)v_n^{\delta+1},\psi_k''(w_n)\frac{\partial}{\partial x}w_n\bigg) \\&\nonumber
		= (n+1)\bigg(u_n^{\delta+1}-v_n^{\delta+1},\psi_k''(w_n)\frac{\partial}{\partial x}w_n\bigg)-\bigg((|u_n|-|v_n|)u_n^{\delta+1},\psi_k''(w_n) \frac{\partial}{\partial x}w_n\bigg)\\&\nonumber\quad-\bigg(|v_n|(u_n^{\delta+1}-v_n^{\delta+1}),\psi_k''(w_n)\frac{\partial}{\partial x}w_n\bigg)\\&
		\leq C(\delta,n)\int_{0}^{1}|w_n(x)|_+^2\d x+\frac{\nu}{4}\int_{0}^{1}\psi_k''(w_n(x))\bigg(\frac{\partial}{\partial x}w_n(x)\bigg)^2\d x.
	\end{align}
	\vskip 0.2 cm
	\noindent 
	{\bf Case 4:} If $|u_n|\geq n+1,\; |v_n|\leq n$ then $\p_n(u_n)=\c_n(u_n)=0,\; \p_n(v_n)=\p(v_n)$ and $\c_n(v_n)=\c(v_n)$ also  $1<|u_n|-|v_n|\leq |u_n-v_n|$. A simple calculation as we performed  in the previous parts and bound of $\psi_k'$ help us to obtain
	\begin{align}\label{5.11}\nonumber
		\mathcal{A}^{(2)} &= -\big(\c(v_n),\psi_k'(w_n)\big) \leq \big(|\c(v_n)|,\psi_k'(w_n)\big)  \leq C(\gamma,\delta,n)\int_{0}^{1}1\cdot|u_n-v_n|_+\d x \\& \leq C(\gamma,\delta,n)\int_{0}^{1}|w_n(x)|_+^2\d x.
	\end{align}Again, applying Young's inequality and the bound of $\psi_k''$, we find
	\begin{align}\label{5.12}\nonumber
		\mathcal{A}^{(3)}&= -\bigg(v_n^{\delta+1},\psi_k''(w_n)\frac{\partial}{\partial x}w_n\bigg) \leq C(\delta,n)\int_{0}^{1}1\cdot\psi_k''(w_n)\frac{\partial}{\partial x}w_n\d x \\& \leq C(\nu,\delta,n)\int_{0}^{1}|w_n(x)|_+^2\d x+\frac{\nu}{4}\int_{0}^{1}\psi_k''(w_n(x))\bigg(\frac{\partial}{\partial x}w_n(x)\bigg)^2\d x.
	\end{align}
	\vskip 0.2 cm
	\noindent 
	{\bf Case 5:} If $|u_n|> n+1,\; n< |v_n|\leq n+1,$ then $\p_n(u_n)=\c_n(u_n)=0$, $\p_n(v_n)=(n+1-|v_n|)\p(v_n)$,  $\c_n(v_n)=(n+1-|v_n|)\c(v_n)$ and $1<|u_n|-|v_n|\leq |u_n-v_n|$. A similar calculation as in the case 4 helps us to estimate the term $\mathcal{A}^{(2)}$ as 
	\begin{align}\label{5.13}\nonumber
		\mathcal{A}^{(2)} &= -\big((n+1-|v_n|)\c(v_n),\psi_k'(w_n)\big) \leq C(\gamma,\delta,n) \int_{0}^{1}1\cdot \psi_k'(w_n)\d x \\& \leq C(\gamma,\delta,n)\int_{0}^{1}|w_n(x)|_+^2\d x.
	\end{align}Using Young's inequality and the bound of $\psi_k''$, we estimate $\mathcal{A}^{(3)}$ as
	\begin{align}\label{5.14}\nonumber
		\mathcal{A}^{(3)} &= -\bigg((n+1-|v_n|)v_n^{\delta+1},\psi_k''(w_n)\frac{\partial}{\partial x}w_n\bigg)	 \\& \nonumber\leq \int_{0}^{1}\big|n+1-|v_n|\big||v_n|^{\delta+1}\psi_k''(w_n)\bigg|\frac{\partial}{\partial x}w_n\bigg|\d x \\&\nonumber\leq 
		C(\delta,n)\int_{0}^{1}1\cdot \psi_k''(w_n)\bigg|\frac{\partial}{\partial x}w_n\bigg|\d x \\&
		\leq C(\nu,\delta,n)\int_{0}^{1}|w_n(x)|_+^2\d x+\frac{\nu}{4}\int_{0}^{1}\psi_k''(w_n(x))\bigg(\frac{\partial}{\partial x}w_n(x)\bigg)^2\d x.
	\end{align}
	\vskip 0.2 cm
	\noindent 
	{\bf Case 6:} If $n<|u_n|\leq n+1,\; |v_n|\leq n$ then $\p_n(u_n)=(n+1-|u_n|)\p(u_n),\; \c_n(u_n)=(n+1-|u_n|)\c(u_n)$, $\p_n(v_n)=\p(v_n)$ and $\c_n(v_n)=\c(v_n)$.	Using the conditions on $u_n, v_n$ and  bounds of $\psi_k'$, we find		
	\begin{align}\label{5.15}\nonumber
		\mathcal{A}^{(2)}& = \big((n+1-|u_n|)\c(u_n)-(n+1-|v_n|)\c(v_n)+(n-|v_n|)\c(v_n),\psi_k'(w_n)\big) \\& \nonumber= (n+1)\big(\c(u_n)-\c(v_n),\psi_k'(w_n)\big)-\big(|u_n|\c(u_n)-|v_n|\c(v_n),\psi_k'(w_n)\big)\\&\nonumber \quad+\big((n-|v_n|)\c(v_n),\psi_k'(w_n)\big) \\& \nonumber\leq C(\gamma,\delta,n)\big(|w_n|,\psi_k'(w_n)\big)-\big((|u_n|-|v_n|)\c(u_n)-|v_n|(\c(u_n)-\c(v_n)),\psi_k'(w_n)\big)\\&\nonumber\quad+ \big((n-|v_n|)\c(v_n),\psi_k'(w_n)\big) \\&\nonumber \leq C(\gamma,\delta,n)\big(|w_n|,\psi_k'(w_n)\big)+\big(|u_n-v_n||\c(u_n)|,\psi_k'(w_n)\big)\\&\nonumber\quad+\big(|v_n|(\c(u_n)-\c(v_n)),\psi_k'(w_n)\big)+\big(|u_n-v_n||\c(v_n)|,\psi_k'(w_n)\big)\\& 
		\leq C(\gamma,\delta,n) \int_{0}^{1}|w_n(x)|_+^2\d x.
	\end{align} Again with the help of the bounds of $u_n, v_n$ and estimate on $\psi_k''$, we get
	\begin{align}\label{5.16}\nonumber
		\mathcal{A}^{(3)}&=\bigg((n+1-|u_n|)u_n^{\delta+1}-(n+1-|v_n|)v_n^{\delta+1}+(n-|v_n|)v_n^{\delta+1},\psi_k''(w_n)\frac{\partial}{\partial x}w_n\bigg) \\& \nonumber=
		(n+1)\bigg(u_n^{\delta+1}-v_n^{\delta+1},\psi_k''(w_n)\frac{\partial}{\partial x}w_n\bigg)\\&\nonumber\quad-\bigg((|u_n|-|v_n|)u_n^{\delta+1}-|v_n|(u_n^{\delta+1}-v_n^{\delta+1}),\psi_k''(w_n)\frac{\partial}{\partial x}w_n\bigg)\\&\nonumber\quad+\bigg((n-|v_n|)v_n^{\delta+1},\psi_k''(w_n)\frac{\partial}{\partial x}w_n\bigg) \\& \nonumber\leq 
		C(\delta,n)\bigg(|w_n|,\psi_k''(w_n)\bigg|\frac{\partial}{\partial x}w_n\bigg|\bigg)	\\&\nonumber\quad+\bigg(|u_n-v_n||u_n|^{\delta+1}+|v_n||u_n^{\delta+1}-v_n^{\delta+1}|,\psi_k''(w_n)\bigg|\frac{\partial}{\partial x}w_n\bigg|\bigg)	\\&\nonumber\quad	+ \bigg(|u_n-v_n||v_n|^{\delta+1},\psi_k''(w_n)\bigg|\frac{\partial}{\partial x}w_n\bigg|\bigg) \\&
		\leq C(\delta,n) \int_{0}^{1}|w_n(x)|_+^2\d x +\frac{\nu}{4}\int_{0}^{1}\psi_k''(w_n(x))\bigg(\frac{\partial}{\partial x}w_n(x)\bigg)^2\d x.
	\end{align}	

	Using the boundedness of $\phi_i$, estimates on $\psi''_k$ and Hypothesis \ref{H1} on $g$, we find a constant $C$ such that 
		\begin{align}\label{5.17}
			\mathtt{B}(s)\leq C(n,L)\int_{0}^{1}|w_n(s,x)|_+^2\d x.	
		\end{align}
	  From \eqref{5.03}-\eqref{5.17}, for every $n,$ we can find a constant $C=C(\nu,\alpha,\beta,\gamma,\delta,n)$ such that for all $k$ and  $t\in[0,T]$ 
		\begin{align*}
			\E\left[\Psi_k(w_n(t\wedge\tau))\right] \leq C\int_{0}^{t}\E\bigg(\int_{0}^{1}|w_n(s\wedge\tau,x)|_+^2\d x\bigg)\d s,
		\end{align*} for some stopping time $\tau$. Using DCT,  we are able to pass the limit $k\to\infty$  
		\begin{align*}
			\E\left[\Psi(w_n(t\wedge\tau))\right] \leq C\int_{0}^{t}\E\left[\Psi(w_n(s\wedge\tau))\right]\d s,
		\end{align*}where $\Psi(w_n(\cdot)):= \int_{0}^{1}|w_n(\cdot,x)|_+^2\d x$.
		Hence we obtain \eqref{5.02} by applying Gronwall's lemma, replacing the stopping time $\tau$ by a sequence of stopping times, which localize the process $\Psi(w_n(\cdot))$. Consequently,  we deduce that 
		\begin{align}\label{521}
		 u_n(t,x)\leq v_n(t,x), \ \mathbb{P}\text{-a.s., for all } t\in[0,T]\ \text{ and for a.e. }\ x\in[0,1].
		\end{align}
		\vskip 0.2 cm 
		\noindent{\bf Step 2:}
	Let us set $\v_n=u_n-\z_n,$ where $\z_n(t,x)= \int_{0}^{t}\int_{0}^{1}G(t-s,x,y)g(s,y,u_n(s,y)) \W(\d s,\d y)$. Since $g(\cdot,\cdot,\cdot)$ is uniformly bounded in $n$, using Corollary 4.3,  \cite{IG} for every $p\geq 1, \; T>0$, there exists a constant $C$ such that for all $n\geq1 $
	\begin{align*}
		\E\big(|\z_n(t,x)-\z_n(s,y)|^{2p}\big) \leq C\big(|t-s|^{\frac{1}{4}}+|x-y|^{\frac{1}{2}}\big)^{2p},
	\end{align*}for all $0\leq s,t\leq T,\; 0\leq x,y\leq 1$. Hence by Kolmogorov's continuity theorem  (see \cite{GRR,JBW}) we have
\begin{align*}
	\sup_{n\geq 1}\E\left(\sup_{(t,x)\in[0,T]\times[0,1]}|\z_n(t,x)|^p\right)<\infty,
\end{align*} therefore, $\z_n^*:=\sup\limits_{(t,x)\in[0,T]\times[0,1]}|\z_n(t,x)|$ is bounded in probability, uniformly in $n$. Firstly, we establish that $\sup\limits_{t\in[0,T]}\|u_n(t)\|_{\L^2}$ is also bounded in  probability, uniformly in $n$, with $u_n=\v_n+\z_n,$ where $\v_n$ is the unique mild  solution of the following equation:
\begin{align*}
	\frac{\partial \v_n}{\partial t}= \nu \frac{\partial^2 \v_n}{\partial x^2}+\beta \c_n(\v_n+\z_n)-\frac{\alpha}{\delta+1}\frac{\partial}{\partial x}\p_n(\v_n+\z_n), 
\end{align*}with the boundary conditions $\v_n(t,0)=\v_n(t,1)=0, \; t\in[0,T]$ and the initial condition $\v_n(0,x)=u_0(x),\; x\in[0,1]$. 
We have the following energy equality (see Proposition \ref{lem3.4})
\begin{align}\label{5.18}\nonumber
	\|\v_n(t)\|_{\L^p}^p&=\|u_0\|_{\L^p}^p-\nu p(p-1)\int_{0}^{t}\||\v_n(s)|^{\frac{p-2}{2}}\partial_x \v_n(s)\|_{\L^2}^2\d s \\&\nonumber\quad+\beta p \int_{0}^{t}\big(\c_n(\v_n(s)+\z_n(s)),|\v_n(s)|^{p-2}\v_n(s)\big)\d s\\& \quad-\frac{\alpha p}{\delta+1}\int_{0}^{t}\big(\partial_x\p_n(\v_n(s)+\z_n(s)), |\v_n(s)|^{p-2}\v_n(s)\big)\d s,
\end{align}
for all $t\in[0,T]$. 
Using the definition of cutoff function, we consider the following three cases: 
\vskip 0.2 cm
\noindent
{\bf Case 1:} If $|\v_n+\z_n|\leq n,$ then $\p_n(\cdot)=\p(\cdot)$ and $\c_n(\cdot)=\c(\cdot)$.
Using \eqref{UEE4}-\eqref{UEE6}, we can estimate the term $\beta\big(\c_n(\v_n+\z_n),|\v_n|^{p-2}\v_n\big)$ as
\begin{align}\label{5.19}\nonumber
	&\beta\big(\c_n(\v_n+\z_n),|\v_n|^{p-2}\v_n\big)\\&\nonumber\leq \frac{\beta}{2}\|\v_n\|_{\L^{p+2\delta}}^{p+2\delta}+\big\{2^{2\delta}\beta(1+\gamma)^2+2^\delta\beta(1+\gamma)+\beta\gamma+2^{2\delta-1}\beta(2\delta+1)\big\}\|\v_n\|_{\L^p}^p\\&\nonumber\quad +\frac{2^\delta\beta (1+\gamma)}{p}\bigg(\frac{p-1}{p}\bigg)^{p-1}\|\zeta_n\|_{\L^{p(\delta+1)}}^{p(\delta+1)}\\&\quad+\frac{1}{p+2\delta}\bigg(\frac{p(p+2\delta-1)}{\beta(p+2\delta)}^{p+2\delta-1}\bigg)^{p+2\delta-1}\big(2^{2\delta-1}\beta(2\delta+1)\big)^{p+2\delta}\|\zeta_n\|_{\L^{p+2\delta}}^{p+2\delta}.
\end{align}	From \eqref{UEE3}, we have
\begin{align}\label{5.20}\nonumber
	&\frac{\alpha}{\delta+1}\big|\big(\partial_x(\v_n+\z_n)^{\delta+1},|\v_n|^{p-2} \v_n\big)\big|
	\\&\nonumber\leq \frac{\nu}{2}\||\v_n|^{\frac{p-2}{2}}\partial_x\v_n\|_{\L^2}^2+\frac{\beta}{4}\|\v_n\|_{\L^{p+2\delta}}^{p+2\delta} + \frac{2}{p+2\delta}\bigg(\frac{4(p+2\delta-2)}{\beta(p+2\delta)}\bigg)^{\frac{p+2\delta-2}{2}}\\&\nonumber\qquad\times\bigg(\frac{2^{\delta}(p-1)^2\alpha^2}{2\nu}\bigg)^{\frac{p+2\delta}{2}}\|\zeta_n\|_{\L^{p+2\delta}}^{p+2\delta}+ \frac{2^\delta (p-1)^2\alpha^2}{4\nu}\|\v_n\|_{\L^p}^p\\&\quad +\frac{2^\delta (p-1)^2\alpha^2}{\nu p}\bigg(\frac{2(p-2)}{p}\bigg)^{\frac{p-2}{2}}\|\zeta_n\|_{\L^{p(\delta+1)}}^{p(\delta+1)}.
\end{align}
\vskip 0.2 cm
\noindent
{\bf Case 2:} If $|\v_n+\z_n|> n+1$, then $\p_n(\v_n+\z_n)=\c_n(\v_n+\z_n)=0$. 
\vskip 0.2 cm
\noindent
{\bf Case 3:} If $ n< |\v_n+\z_n|\leq n+1$, then $\p_n(\v_n+\z_n)=m\p(\v_n+\z_n)$ and $\c_n(\v_n+\z_n)=m\c(\v_n+\z_n)	$, where $m=n+1-|\v_n+\z_n|$ with $0\leq m<1$. We estimate the term $\beta\big(m\c_n(\v_n+\z_n),\v_n\big)$ as
\begin{align}\label{5.21}\nonumber
	&\beta\big(m\c_n(\v_n+\z_n),|\v_n|^{p-2}\v_n\big) \\&\nonumber=  
	\beta(1+\gamma)\big(m(\v_n+\z_n)^{\delta+1},|\v|^{p-2}\v_n\big)-\beta\gamma\big(m(\v_n+\z_n),|\v_n|^{p-2}\v_n\big)\\&\nonumber\quad-\beta\big(m(\v_n+\z_n)^{2\delta+1},|\v_n|^{p-2}\v_n\big)
	\\&=: I_1+I_2+I_3.
\end{align}Using similar arguments to \eqref{UEE4}, we obtain 
\begin{align}\label{5.021}\nonumber
	|I_1|&\leq 2^\delta	\beta(1+\gamma)\big(m^{\frac{1}{2}}(|\v_n|^{\delta+1}+|\z_n|^{\delta+1}),m^{\frac{1}{2}}|\v_n|^{p-1}\big)\\&\nonumber = 2^\delta\beta(1+\gamma)\big\{\big(m^{\frac{p+3\delta}{2(p+2\delta)}}m^{\frac{p+\delta}{2(p+2\delta)}}|\v_n|^{\delta+1},m^{\frac{p+\delta}{2(p+2\delta)}}|\v|^{p-1}\big)+\big(m^\frac{1}{2}|\z_n|^{\delta+1},m^\frac{1}{2}|\v_n|^{p-1}
	\big)\big\}	
	 \\&\nonumber\leq 2^\delta\beta(1+\gamma)\|m^\frac{1}{2(p+2\delta)}\v_n\|_{\L^{p+\delta}}+2^\delta\beta(1+\gamma)\|m^\frac{1}{2(\delta+1)}\z_n\|_{\L^{p(\delta+1)}}^{\delta+1}\|m^{\frac{1}{2(p-1)}}\v_n\|_{\L^p}^{p-1} \\&\nonumber\leq 
	2^\delta\beta(1+\gamma)\|m^{\frac{1}{2(p+2\delta)}}\v_n\|_{\L^{p+2\delta}}^{\frac{p+2\delta}{2}}\|m^{\frac{1}{2(p+\delta)}}\v_n\|_{\L^p}^\frac{p}{2}+2\delta\beta(1+\gamma)\|m^{\frac{1}{2(p-1)}}\v\|_{\L^p}^p\\&\nonumber\quad +\frac{2^\delta\beta(1+\gamma)}{p}\bigg(\frac{p-1}{p}\bigg)^{p-1}\|m^{\frac{1}{2(\delta+1)}}\z_n\|_{\L^{p(\delta+1)}}^{p(\delta+1)}\\&\nonumber\leq 
	\frac{\beta}{4}\|m^{\frac{1}{2(p+2\delta)}}\v_n\|_{\L^{p+2\delta}}^{p+2\delta}+2^{2\delta}\beta(1+\gamma)^2\|\v_n\|_{\L^p}^p+2^\delta\beta(1+\gamma)\|\v_n\|_{\L^p}^p\\&\quad+\frac{2^\delta\beta(1+\gamma)}{p}\bigg(\frac{p-1}{p}\bigg)^{p-1}\|\z_n\|_{\L^{p(\delta+1)}}^{p(\delta+1)},
		\end{align}where we have used the fact that $m\in[0,1)$. Using similar arguments to \eqref{UEE5}, the term $I_2$ of \eqref{5.21} can be handle in the following way:
	\begin{align}\label{5.022}\nonumber
		I_2 &=-\beta\gamma \big(m\v_n,|\v_n|^{p-2}\v_n\big)-\beta\gamma \big(m\z_n,|\v_n|^{p-2}\v_n\big)\\&\nonumber =-\beta\gamma\|m^{\frac{1}{2p}}\v_n\|_{\L^p}^p-\beta\gamma \big(m^{\frac{p+1}{2p}}\z_n,m^{\frac{p-1}{2p}}|\v_n|^{p-2}\v_n\big)\\&\nonumber\leq -\beta\gamma\|m^{\frac{1}{2p}}\v_n\|_{\L^p}^p+\beta\gamma \|m^{\frac{p+1}{2p}}\z_n\|_{\L^p}\|m^{\frac{1}{2p}}\v_n\|_{\L^p}^{p-1} \\&\leq -\frac{\beta\gamma}{2}\|m^{\frac{1}{2p}}\v_n\|_{\L^p}^p+\frac{\beta\gamma}{p}\bigg(\frac{p-1}{p}\bigg)^{p-1}\|\z_n\|_{\L^p}^p.
		\end{align} Applying similar calculations to \eqref{UEE6}, we estimate the term $I_3$ of \eqref{5.21} as 
	\begin{align}\label{5.023}\nonumber
		I_3&= -\beta \big(m\v_n^{2\delta+1}+(2\delta+1)m\z_n(\theta\v_n+(1-\theta)\z_n)^{2\delta},|\v_n|^{p-2}\v_n\big)\\&\nonumber = -\beta \big(m\v_n^{2\delta+1},|\v_n|^{p-2}\v_n\big)-\beta(2\delta+1)\big(m\z_n(\theta\v_n+(1-\theta)\z_n)^{2\delta},|\v_n|^{p-2}\v_n\big)\\&\nonumber= 
		-\beta \|m^{\frac{1}{2(p+2\delta)}}\v_n\|_{\L^{p+2\delta}}^{p+2\delta}-\beta(2\delta+1)\big(m\z_n(\theta\v_n+(1-\theta)\z_n)^{2\delta},|\v_n|^{p-2}\v_n\big) \\&\nonumber\leq 	-\beta \|m^{\frac{1}{2(p+2\delta)}}\v_n\|_{\L^{p+2\delta}}^{p+2\delta} -2^{2\delta-1}\beta(2\delta+1)\big(m|\z_n||\v_n|^{2\delta}+m|\z_n|^{2\delta},|\v_n|^{p-1}\big)\\&\nonumber=	-\beta \|m^{\frac{1}{2(p+2\delta)}}\v_n\|_{\L^{p+2\delta}}^{p+2\delta} -2^{2\delta-1}\beta(2\delta+1)\big(m^{\frac{p+2\delta+1}{2(p+2\delta)}}|\z_n|m^{\frac{p+2\delta-1}{2(p+2\delta)}}|\v_n|^{2\delta}+m|\z_n|^{2\delta},|\v_n|^{p-1}\big)\\&\nonumber\leq -\beta \|m^{\frac{1}{2(p+2\delta)}}\v_n\|_{\L^{p+2\delta}}^{p+2\delta} +2^{2\delta-1}\beta(2\delta+1)\|m^{\frac{p+2\delta+1}{2(p+2\delta)}}\z_n\|_{\L^{p+2\delta}}\|m^{\frac{1}{2(p+2\delta)}}\v_n\|_{\L^{p+2\delta}}^{p+2\delta-1}\\&\nonumber\quad+\|\z_n\|_{\L^{p(2\delta+1)}}^{2\delta+1}\|\v_n\|_{\L^p}^{p-1}\\&\nonumber\leq -\frac{3\beta}{4}\|m^{\frac{1}{2(p+2\delta)}}\v_n\|_{\L^{p+2\delta}}^{p+2\delta}+\frac{1}{p+2\delta}\bigg(\frac{4(p+2\delta+1)}{\beta(p+2\delta)}\bigg)^{p+2\delta-1}\big(2^{2\delta-1}\beta(2\delta+1)\big)^{p+2\delta}\|\z_n\|_{\L^{p+2\delta}}^{p+2\delta}\\&\quad +2^{2\delta-1}\beta(2\delta+1)\|\v_n\|_{\L^p}^p.
	\end{align}In similar ay as in \eqref{5.20}, and the bound $0\leq m<1$, we estimate the term $\frac{\alpha}{\delta+1}\big(m(\v_n+\z_n)^{\delta+1},\partial_x \v_n\big)$as
\begin{align}\label{5.22}\nonumber
	-&\frac{\alpha}{\delta+1}\big(\partial_xm(\v_n+\z_n)^{\delta+1},|\v_n|^{p-2} \v_n\big)\\&\nonumber=
	\alpha(p-1)\big(m\zeta_n(\theta\v_n+(1-\theta)\zeta_n)^\delta,|\v_n|^{p-2}\partial_x\v_n\big)\\&\nonumber=
	\alpha(p-1)\big(m^{\frac{1}{2}}|\v_n|^{\frac{p-2}{2}}\zeta_n(\theta\v_n+(1-\theta)\zeta_n)^\delta,m^{\frac{1}{2}}|\v_n|^{\frac{p-2}{2}}\partial_x\v_n\big)\\&\nonumber\leq 
	\alpha(p-1)\|m^{\frac{1}{2}}|\v_n|^{\frac{p-2}{2}}\zeta_n(\theta\v_n+(1-\theta)\zeta_n)^\delta\|_{\L^2}\|m^{\frac{1}{2}}|\v_n|^{\frac{p-2}{2}}\partial_x\v_n\|_{\L^2}	\\&\nonumber\leq \frac{\nu}{2}\||\v_n|^{\frac{p-2}{2}}\partial_x\v_n\|_{\L^2}^2+\frac{\beta}{4}\|m^{\frac{1}{2(p+2\delta)}}\v_n\|_{\L^{p+2\delta}}^{p+2\delta} + \frac{2}{p+2\delta}\bigg(\frac{4(p+2\delta-2)}{\beta(p+2\delta)}\bigg)^{\frac{p+2\delta-2}{2}}\\&\nonumber\qquad\times\bigg(\frac{2^{\delta}(p-1)^2\alpha^2}{2\nu}\bigg)^{\frac{p+2\delta}{2}}\|\zeta_n\|_{\L^{p+2\delta}}^{p+2\delta}+ \frac{2^\delta (p-1)^2\alpha^2}{4\nu}\|\v_n\|_{\L^p}^p\\&\quad +\frac{2^\delta (p-1)^2\alpha^2}{\nu p}\bigg(\frac{2(p-2)}{p}\bigg)^{\frac{p-2}{2}}\|\zeta_n\|_{\L^{p(\delta+1)}}^{p(\delta+1)}.
\end{align}Let us now discuss the rest of the proof for case 1 only, for the other cases, one can obtain a similar result by using \eqref{5.21}-\eqref{5.22}, instead of \eqref{5.19} and \eqref{5.20}. Substituting \eqref{5.19}, \eqref{5.20} in \eqref{5.18}, and using similar arguments to \eqref{UEE7}, we obtain
\begin{align}\label{5.23}\nonumber
		&\|\v_n(t)\|_{\L^p}^p+\frac{\nu p(p-1)}{2}\int_{0}^{t}\||\v_n(s)|^{\frac{p-2}{2}}\partial_x \v_n(s)\|_{\L^2}^2\d s+\beta p \gamma \int_{0}^{t} \|\v_n(s)\|_{\L^{p}}^p\d s\\&\nonumber\quad +\frac{\beta p}{8}\int_0^t \|\v_n(s)\|_{\L^{p+2\delta}}^{p+2\delta}\d s \\&\leq\|u_0\|_{\L^p}^p+pK_1 T+pK_2T\sup_{t\in[0,T]}\|\z_n(t)\|_{\L^{p(\delta+1)}}^{p(\delta+1)}+pK_3T\sup_{t\in[0,T]}\|\z_n(t)\|_{\L^{p+2\delta}}^{p+2\delta},
	\end{align}where the constants $K_i,$ for $i=1,2,3$ are defined in \eqref{3150}-\eqref{315}, respectively.
From \eqref{5.23}, we find
\begin{align*}
\sup\limits_{t\in[0,T]}	\|\v_n(t)\|_{\L^p}^p\leq  \|u_0\|_{\L^p}^p+pK_1 T+pK_2T|\z_n^*|^{p(\delta+1)}+pK_3T|\z_n^*|^{p+2\delta},
	\end{align*}holds for all $n\geq 1$. From the above expression,  we obtain that the sequence $\sup\limits_{t\in[0,T]}\|\v_n(t)\|_{\L^p}$ is bounded in probability, uniformly in $n,$ since the sequence $\z_n^*$ is bounded in probability, uniformly in $n$. Therefore the sequence $\sup\limits_{t\in[0,T]}\|u_n(t)\|_{\L^p}$ is also bounded in probability, uniformly in $n$.

By Lemmas \ref{lemma3.1}, \ref{lemma3.2} and \ref{lemma3.6}, the sequences of $\L^p(\1)$-valued stochastic processes $J_1^n(t), J_2^n(t)$ defined by
\begin{align*}
	J_1^n(t)&= \int_{0}^{t}\int_{0}^{1}G(t-s,x,y)\c_n(u_n(s,y))\d y \d s,\\
	J_2^n(t)&= \int_{0}^{t}\int_{0}^{1} \frac{\partial G}{\partial y }(t-s,x,y)\p_n(u_n(s,y))\d y \d s,
\end{align*}are weakly compact in $\C([0,T];\L^p(\1)), \text{ for } p\geq2\delta+1$. The process
\begin{align*}
	J_0(t)=\int_{0}^{1}G(t,x,y)u_0(y)\d y, \; \text{ for } t\in[0,T]
\end{align*}is in $\C([0,T];\L^p(\1)), \text{ for } p\geq1$. Also, the sequence of the processes
\begin{align*}
	J_3^n(t)= \int_{0}^{t}\int_{0}^{1}G(t-s,x,y)g(s,y,u_n(s,y)) \W(\d s,\d y), \; t\in[0,T], \; x\in[0,1]
\end{align*}is weakly compact in $\C([0,T]\times[0,1])$. Thus the sequence of the processes 
\begin{align*}
	u_n(t)=J_0(t)+\beta J_1^n(t)+\frac{\alpha}{\delta+1} J_2^n(t)+J_3^n(t), \; t\in[0,T]
\end{align*}is weakly compact in the space $\C([0,T];\L^p(\1)), \text{ for } p\geq  2\delta+1$.

	\vskip 0.2 cm 
\noindent{\bf Step 3:}
 In order to conclude the proof of our main result, we apply  Skorokhod's representation  theorem and  Lemma 4.1, \cite{IG} (see Lemma 1.1, \cite{IGNK} also).  For a given pair of subsequences $u_m$ and $u_l$, there exist subsequences $m_k$ and $l_k$ (denoted by $u_{m_k},u_{l_k}$) and a sequence of random variables $z_k:=(\tilde{u}_k,\bar{u}_k,\bar{\W}_k)$, $k=1,2,\ldots$, in $\mathbb{B}:=	\C([0,T];\L^p(\1))\times\C([0,T];\L^p(\1))\times	\C([0,T]\times[0,1])$, for $p\geq 2\delta+1$ in some probability space $(\bar{\Omega},\mathcal{\bar{F}},\bar{\P})$ such that the sequence $z_k$ converges almost surely in $\mathbb{B}$ to the random variable $z:=	(\tilde{u},\bar{u},\bar{\W})$ as $k\to\infty$, with the same distributions of $z_k$ and $(u_{m_k},u_{l_k},\W)$. Here, the two random fields $\bar{\W}$ and $\bar{\W}_k$ are Brownian fields defined on  different stochastic bases $\bar{\Theta}=(\bar{\Omega},\bar{\mathscr{F}}, (\bar{\mathscr{F}}_t)_{t\geq 0}, \bar{\P})$ and $\bar{\Theta}_k=(\bar{\Omega},\bar{\mathscr{F}}, (\bar{\mathscr{F}}_t^k)_{t\geq 0}, \bar{\P})$, respectively, where $\bar{\mathscr{F}}_t$ and $\bar{\mathscr{F}}_t^k$ are the completion of the $\sigma$-fields generated by $z(s,x)$ and $z_k(s,x),$ for all $s\leq t,\; x\in[0,1]$, respectively.	For every smooth function $\phi\in\C^2([0,1])$, with $\phi(0)=\phi(1)=0$, we have 
	\begin{align}\label{5.24}\nonumber
		&\int_{0}^{1}u_n(t,x)\phi(x)\d x= \int_{0}^{1}u_0(x)\phi(x)d x+\nu\int_{0}^{t}\int_{0}^{1}u_n(s,x)\phi''(x)\d x\d s\\&\nonumber\quad+\beta\int_{0}^{t}\int_{0}^{1}\c_n(u_n(s,x))\phi(x)\d x\d s+\frac{\alpha}{\delta+1}\int_{0}^{t}\int_{0}^{1}\p_n(u_n(s,x))\phi'(x)\d x\d s\\&\quad+\int_{0}^{t}\int_{0}^{1}g(s,x,u_n(s,x))\phi(x)\W(\d s,\d x), \ \mathbb{P}\text{-a.s.},
	\end{align}	for all $t\in[0,T]$. Also the above equation \eqref{5.24} holds true if we replace $u_n$ and $\W$ by $\tilde{u}_k$ and $\bar{\W}_k$, then on passing the limit $k\to\infty$ with the help of Corollary 4.5, \cite{IG} and Lemma \ref{lemma5.1},  one can deduce that $\tilde{u}$ solves the SGBH equation \eqref{1.01} on the stochastic basis $\bar{\Theta}$ with the Wiener process $\bar{\W}$. Similarly $\bar{u}$ also solves the SGBH equation \eqref{1.01} on the basis $\bar{\Theta}$ with the Wiener process $\bar{\W}$. We already know that the solution is unique, and hence we  obtain $\tilde{u}=\bar{u}$. Now applying Lemma 4.1,  \cite{IG},  we obtain that $u_n$ converges in $\C([0,T];\L^p(\1))$, for $p\geq  2\delta+1$ in probability to some random element $u\in \C([0,T];\L^p(\1))$, for $p\geq  2\delta+1$.  Similarly, one can pass the limit for $v_n$.
Applying Corollary 4.5,  \cite{IG} and Lemma \ref{lemma5.1}, one can pass the limit $n\to\infty$ in 
\eqref{5.24} to find that $u$ and $v$  are  solutions of the SGBH equation \eqref{1.01} with initial data $u_0$ and $v_0,$ respectively and from \eqref{521}, one can deduce that  $u(t,x)\leq v(t,x)$, $\mathbb{P}$-a.s., for all $t\in[0,T]$ and a.e. $x\in[0,1]$. 
\end{proof}

\section{Weak Differentiability of the Solution}\label{sec6}\setcounter{equation}{0}
In this section, we discuss the weak differentiability of the solution to SGBH equation. We first start with the general theory, that is, we recall some basic facts about  stochastic calculus of variations (or Malliavin calculus) for the Brownian sheet $\W(\cdot,\cdot)$. The interested one are referred to see the monographs \cite{ND,NDNE}, etc. for a detailed discussion on  Malliavin calculus. In the sequel, we assume that the initial data $u_0\in\C([0,1])$ and $H=\L^2([0,T]\times[0,1])$. 

\subsection{Elements of Malliavin calculus}
Let $\S$ denote the class of smooth and cylindrical random variables of the form 
\begin{align}\label{6.01}
	F=f(\W(h_1),\ldots,\W(h_n)),
\end{align} where $n\geq 1, f\in\C_b^\infty(\R^n)$ (space of infinite times continuously differentiable functions $f:\R^n\to\R$ such that $f$ and all its partial derivatives are bounded), $h_1,\ldots,h_n\in H$.  We define the Malliavin derivative as two-parameter stochastic process $\{\D_{t,x}F,\ (t,x)\in[0,T]\times[0,1]\}$ for any given random variable $F$ of the from \eqref{6.01} as
\begin{align*}
\D F=	\D_{t,x}F =\sum_{i=1}^{n}\frac{\partial f}{\partial x_i}(\W(h_1),\ldots,\W(h_n))h_i(t,x), \; (t,x)\in[0,T]\times[0,1].
\end{align*}In this way, the derivative $\D F$ is an element of $\L^2(\Omega\times[0,T]\times[0,1])\cong\L^2(\Omega; H)$. In general, one can define the iterative derivative operator on a smooth and cylindrical random variable by setting 
\begin{align*}
	\D_{z_1,\ldots,z_k}^n F =\D_{z_1}\cdots\D_{z_k} F\ \text{ for } \ z_i\in [0,T]\times[0,1], \; i=1,\ldots,k.
\end{align*}This iterative operator $\D^k$ is a closable unbounded operator from $\L^p(\Omega)$ into $\L^p(\Omega;\L^2(([0,T]\times[0,1])^k)),$ for each $k\geq 1$  and $p\geq 1$. Let us denote the closure of $\S$ with respect to the norm  
\begin{align*}
	\| F\|_{k,p}^p =\E\big[| F|^p\big]+\E\left[\sum_{i=1}^{k}\|\D  F\|_{\L^2(([0,T]\times[0,1])^i)}^p\right]
\end{align*} by $\Ds^{k,p}$.  Also, let us set $\Ds^\infty=\bigcap_{k,p}\Ds^{k,p}$.
The operator $\D$ is local in the space $\Ds^{1,1}$ in the sense that for any given random field $ F\in\Ds^{1,1}$ \begin{align*}
	\D F\chi_{\{ F=0\}}=0, \ \P\text{-a.s.}
\end{align*} We denote by $\Ds_{\loc}^{k,p},$ the space of random variables $ F$ such that there exists a sequence $(\Omega_n, F_n)\subset \mathscr{F}\times \Ds^{k,p}$ with $\Omega_n\uparrow\Omega$ and for each $n,  F= F_n$, $\P$-a.s. on $\Omega_n$. In the above case, the derivatives $\D^i F$ are defined by 
\begin{align*}
	\D F^i=\D F_n^i \ \text{ on } \ \Omega_n.
\end{align*} Next, we state a basic criteria for the existence of densities for one-dimensional random variables, the interested readers may refer to \cite{NBFH,ND}, etc. 
\begin{theorem}[Proposition 7.1.2, \cite{NDNE}]\label{theorem6.1}
	Let $ F$ be a random variable in $\Ds_{\loc}^{1,1}$. Suppose that $\|\D F\|_H>0$, $\P$-a.s. Then the law of $ F$ is absolutely continuous with respect to the Lebesgue measure on $\mathbb{R}$.
\end{theorem}

\subsection{Weak differentiability of the solution to SGBH equation} Let us now prove the weak differentiability of the solution to the SGBH equation \eqref{1.01} in the Malliavin calculus sense. In view of Theorem \ref{theorem6.1}, if we manage to show that the solution $u(\cdot)$ of SGBH equation \eqref{1.01} belongs to $\Ds_{\loc}^{1,p},$ for  $p\geq 2\delta+1$ and $\|\D F\|_H>0$, $\P$-a.s, then we are done. Our main goal of this section is to prove that $u(\cdot)$ belongs to $\Ds_{\loc}^{1,p},$ for $p\geq 2\delta+1$.
In order to do this, we need to establish the following technical Proposition and Lemmas.

%
  For $p\geq 2\delta+1$, we set $\mathbb{L}^{1,p}=\L^p([0,T]\times[0,1];\Ds^{1,p})$.
\begin{proposition}\label{prop6.5}
	Let $u(\cdot)$ be the solution of \eqref{4.05} for a fixed value of $n$. Then, for $p\geq 2\delta+1$, for each $(t,x)\in[0,T]\times[0,1]$, the random variable $u(t,x)\in\Ds^{1,p}$, and the process $u\in \mathbb{L}^{1,p}$. Moreover, the derivative of $u$ satisfies the following equation:  
	\begin{align}\label{6.04}\nonumber
	&	\D_{r,z}u(t,x)\nonumber\\ &\nonumber=G(t-r,x,z)g(r,z,\pi_nu(r,z))\\&\nonumber\quad+\beta\int_{r}^{t}\int_{0}^{1}G(t-s,x,y)\big((1+\gamma)(\delta+1)(\pi_nu)^\delta-\gamma-(2\delta+1)(\pi_nu)^{2\delta}\big)\D_{r,z}\pi_nu(s,y)\d y\d s\\&\nonumber\quad+\alpha\int_{r}^{t}\int_{0}^{1}\frac{\partial G}{\partial y}(t-s,x,y)(\pi_nu)^\delta\D_{r,z}\pi_nu(s,y)\d y\d s\\&\quad+\int_{r}^{t}\int_{0}^{1}G(t-s,x,y)M_n(s,y)\D_{r,z}\pi_nu(s,y)\W(\d s, \d y), 
	\end{align}where $M_n$ is an adapted processes bounded by the Lipschitz constant $L$.	
\end{proposition}Before going to the proof of Proposition \ref{prop6.5}, we need the following technical Lemmas (see \cite{LZN} also). 
\begin{lemma}\label{lemma6.6}
	Let $u=\{u(t,x);(t,x)\in[0,T]\times[0,1]\}$ be an adapted process in the space $\mathbb{L}^{1,p}$, for $p\geq 2\delta+1$. Consider the operator $\mathscr{A}$ defined in \eqref{4.06}. Then $\mathscr{A}u$ belongs to $\mathbb{L}^{1,p},$ and we have 
	\begin{align}\label{6.05}\nonumber
		&\E\big[\|\D \mathscr{A}u(t)\|_{\L^p(0,1;H)}\big] \\&\nonumber\leq C+C\int_{0}^{t}\left\{(t-s)^{-\frac{\delta}{2p}}n^\delta(1+\gamma)(\delta+1)+\gamma +(t-s)^{-\frac{\delta}{p}}n^{2\delta}(2\delta+1)\right.\\&\qquad\left.+(t-s)^{-\frac{1}{2}-\frac{\delta}{2p}}n^\delta(\delta+1)+(t-s)^{-\frac{1}{2}-\vartheta}\right\}\E\big[\|\D u(s)\|_{\L^p(0,1;H)}\big]\d s,
	\end{align}
 for $\vartheta\in(0,1)$ and  every $0\leq t\leq T$. Moreover, we have 
	\begin{align}\label{6.06}\nonumber
	\D_{r,z}\mathscr{A}u(t,x)  &=G(t-r,x,z)g(r,z,\pi_nu(r,z))\\&\nonumber\quad+\beta\int_{r}^{t}\int_{0}^{1}G(t-s,x,y)\big((1+\gamma)(\pi_nu)^\delta-\gamma-(\pi_nu)^{2\delta}\big)\D_{r,z}\pi_nu(s,y)\d y\d s\\&\nonumber\quad+\frac{\alpha}{\delta+1}\int_{r}^{t}\int_{0}^{1}\frac{\partial G}{\partial y}(t-s,x,y)(\pi_nu)^\delta\D_{r,z}\pi_nu(s,y)\d y\d s\\&\quad+\int_{r}^{t}\int_{0}^{1}G(t-s,x,y)M_n(s,y)\D_{r,z}\pi_nu(s,y)\W(\d s, \d y), 
\end{align}where $M_n$ is an adapted process bounded by the Lipschitz constant $L$.	
\end{lemma}
\begin{proof}
	Chain rule allows us to write \begin{align*}
		&(\D_{r,z} \mathscr{A}_1u)(t,x)\nonumber\\&= \int_{r}^{t}\int_{0}^{1}G(t-s,x,y)\big((1+\gamma)(\delta+1)(\pi_nu)^\delta-\gamma-(2\delta+1)(\pi_nu)^{2\delta}\big)\D_{r,z} \pi_nu(s,y)\d y\d s.
	\end{align*}Using the estimates \eqref{A1}, \eqref{A7}, Young's and H\"older's inequalities, and Lemma 4.3, \cite{LZN}, 
we obtain 
\begin{align*}&\|\D \mathscr{A}_1u(t)\|_{\L^p(0,1;H)} \\&\leq C \int_{0}^{t}\left\{(t-s)^{-\frac{\delta}{2p}}(1+\gamma)(\delta+1)\|(\pi_nu)^\delta\D \pi_nu\|_{\L^{\frac{p}{\delta+1}}(0,1;H)}+ \gamma\|\D \pi_nu\|_{\L^p(0,1;H)} \right.\\& \qquad\left.+(t-s)^{-\frac{\delta}{p}}(2\delta+1)\|(\pi_nu)^{2\delta}\D \pi_nu\|_{\L^{\frac{p}{2\delta+1}}(0,1;H)}\right\}\d s	 
	\\& \leq C\int_{0}^{t}\left\{(t-s)^{-\frac{\delta}{2p}}(1+\gamma)(\delta+1)\|\pi_nu\|_{\L^p(0,1;H)}^\delta\|\D\pi_nu\|_{\L^p(0,1;H)}\d s+ \gamma\|\D \pi_nu\|_{\L^p(0,1;H)} \right.\\&\qquad\left.+(t-s)^{-\frac{\delta}{p}}(2\delta+1)\|\pi_nu\|_{\L^{p}}^{2\delta}\|\D \pi_nu\|_{\L^p(0,1;H)}\right\}\d s \\& \leq  
	C\int_{0}^{t}\left\{(t-s)^{-\frac{\delta}{2p}}n^\delta(1+\gamma)(\delta+1)+\gamma +(t-s)^{-\frac{\delta}{p}}n^{2\delta}(2\delta+1)\right\}\|\D\pi_nu\|_{\L^p(0,1;H)}\d s
	\\& \leq  
	C\int_{0}^{t}\left\{(t-s)^{-\frac{\delta}{2p}}n^\delta(1+\gamma)(\delta+1)+\gamma +(t-s)^{-\frac{\delta}{p}}n^{2\delta}(2\delta+1)\right\}\|\D u\|_{\L^p(0,1;H)}\d s,
\end{align*}for $p\geq 2\delta+1$. Similarly, we can write  the operator $\mathscr{A}_2$ as
\begin{align*}
	\D_{r,z} \mathscr{A}_2u(t,x) = \int_{r}^{t}\int_{0}^{1}\frac{\partial G}{\partial y}(t-s,x,y)(\delta+1)(\pi_nu)^\delta\D \pi_nu(t,x) \d y\d s.
\end{align*}In a similar way as above,  we estimate the term $	\|\D \mathscr{A}_2u(t)\|_{\L^p(0,1;H)}$ as
\begin{align*}
	\|\D \mathscr{A}_2u(t)\|_{\L^p(0,1;H)} &\leq C\int_{0}^{t}(t-s)^{-\frac{1}{2}-\frac{\delta}{2p}}n^\delta(\delta+1)\|\D u\|_{\L^p(0,1;H)}\d s,
\end{align*}for $p\geq \delta+1$. Again, using the chain rule for the Lipschitz function $g(\cdot,\cdot,\cdot)$,  we get 
\begin{align*}
	\D_{r,z} \mathscr{A}_3u(t,x)&= G(t-r,x,z)g(r,z,\pi_nu(r,z))\\&\quad+\int_{r}^{t}\int_{0}^{1}G(t-s,x,y)M_n(s,y)\D \pi_nu(s,y)\W(\d s, \d y),
\end{align*}where $M_n(s,y)$ is an adapted process bounded by $L$. For better understanding, if we assume that the noise coefficient $g(t,x,r)$ is continuously differentiable in the third variable, then $M_n(t,x)=\frac{\partial g}{\partial r}(t,x,r)$. Using the estimate \eqref{lemma3.5.1} with $\vartheta\in(0,1)$, and Lemma 4.3, \cite{LZN}
, we obtain 
\begin{align*}
	\|\D \mathscr{A}_3u(t)\|_{\L^p(0,1;H)} \leq C+C\int_{0}^{t}(t-s)^{-\frac{1}{2}-\vartheta} \|\D u\|_{\L^p(0,1;H)}\d s.
\end{align*}Using the above estimates, one can obtain \eqref{6.05}. On integrating from $0$ to $T$ and using $p\geq 2\delta+1,$ we infer that  $\mathscr{A}u\in\mathbb{L}^{1,p}$, and the proof is completed. 
\end{proof}

\begin{lemma}\label{lemma6.7}
	Under the same assumption as in Lemma \ref{lemma6.6}, if $p> \max\{6,2\delta+1\}$ and $0<\e<\frac{p}{2}-3$, we have 
	\begin{align}\label{6.07}
		\E\left[\sup_{t\in[0,T]}\|\D \mathscr{A}u(t)\|_{\frac{1+\e}{p},p,H}^p\right] \leq C+C\int_{0}^{T}\E\big[\|\D u\|_{\L^p(0,1;H)}^p\big]\d s.
	\end{align}
	\end{lemma}
\begin{proof}We consider
	 \begin{align*}
		&(\D_{r,z} \mathscr{A}_1u)(t,x)\nonumber\\&= \int_{r}^{t}\int_{0}^{1}G(t-s,x,y)\big((1+\gamma)(\delta+1)(\pi_nu)^\delta-\gamma-(2\delta+1)(\pi_nu)^{2\delta}\big)\D_{r,z} \pi_nu(s,y)\d y\d s.
	\end{align*}Using  Minkowski's  inequality and the estimate \eqref{A5}, we obtain 
\begin{align*}
&	\|\D \mathscr{A}_1u(t)\|_{\frac{1+\e}{p},p,H}^p \nonumber\\&\leq K^p \int_{0}^{1}\int_{0}^{1} |x-y|^{-2-\epsilon}\bigg(\int_{0}^{t}\int_{0}^{1}|x-y|(t-s)^{-1}\max\bigg\{e^{-\frac{|x-y|^2}{\ell_5(t-s)}},e^{-\frac{|y-z|^2}{\ell_5(t-s)}}\bigg\}\\&\qquad\times\big|\big((1+\gamma)(\delta+1)(\pi_n u)^\delta-\gamma-(2\delta+1)(\pi_n u)^{2\delta}\big)\D \pi_n u(s,y)\big| \d z\d s\bigg)^p\d x \d y \\& 	\leq C\int_{0}^{1}\bigg(\int_{0}^{t}\int_{0}^{1}(t-s)^{-1}e^{-\frac{|\cdot|^2}{\ell_5(t-s)}} \big|\big((1+\gamma)(\delta+1)(\pi_n u)^\delta-\gamma-(2\delta+1)(\pi_n u)^{2\delta}\big)\\&\qquad\times\D \pi_n u(s,y) \big|\d z\d s\bigg)^p\d x,
\end{align*} 
for $p>1+\e$. Using  Minkowski's, Young's and H\"older's inequalities, the estimates \eqref{A7} and Lemma 4.4, \cite{LZN}, 
 we find
\begin{align*}
	&\|\D \mathscr{A}_1u(t)\|_{\frac{1+\e}{p},p,H} \\&\leq C\int_{0}^{t}(t-s)^{-1}\left\{(1+\gamma)(\delta+1)\big\|e^{-\frac{|\cdot|^2}{\ell_5(t-s)}}*|(\pi_n u)^\delta\D \pi_n u(s,\cdot)|\big\|_{\L^p(0,1;H)}\right.\\&\qquad\left.+\gamma\big\|e^{-\frac{|\cdot|^2}{\ell_5(t-s)}}*|\D \pi_n u(s,\cdot)|\big\|_{\L^p(0,1;H)}+(2\delta+1) \big\|e^{-\frac{|\cdot|^2}{\ell_5(t-s)}}*|(\pi_n u)^{2\delta}\D \pi_n u(s,\cdot)|\big\|_{\L^p(0,1;H)}\right\}\d s \\&\leq C\int_{0}^{t}(t-s)^{-1}\left\{(1+\gamma)(\delta+1)\big\|e^{-\frac{|\cdot|^2}{\ell_5(t-s)}}\big\|_{\L^{\frac{p}{p-\delta}}(0,1;H)}\|(\pi_nu)^\delta\D \pi_nu\|_{\L^{\frac{p}{\delta+1}}(0,1;H)}\right.\\&\qquad\left. +\gamma\big\|e^{-\frac{|\cdot|^2}{\ell_5(t-s)}}\big\|_{\L^1(0,1;H)}\|\D \pi_nu\|_{\L^p(0,1;H)}+(2\delta+1)\big\|e^{-\frac{|\cdot|^2}{\ell_5(t-s)}}\big\|_{\L^{\frac{p}{p-2\delta}}(0,1;H)}\right.\\&\qquad\left. \times\|(\pi_nu)^{2\delta}\D \pi_nu\|_{\L^{\frac{p}{2\delta+1}}(0,1;H)}\right\}\d s 
	\\& \leq C\int_{0}^{t} \left\{(t-s)^{-\frac{1}{2}-\frac{\delta}{2p}}(1+\gamma)(\delta+1)n^\delta+\gamma(t-s)^{-\frac{1}{2}}+(t-s)^{-\frac{1}{2}-\frac{\delta}{p}}n^{2\delta}\right\}\|\D \pi_n u\|_{\L^p(0,1;H)}\d s
		\\& \leq C\int_{0}^{t} \left\{(t-s)^{-\frac{1}{2}-\frac{\delta}{2p}}(1+\gamma)(\delta+1)n^\delta+\gamma(t-s)^{-\frac{1}{2}}+(t-s)^{-\frac{1}{2}-\frac{\delta}{p}}n^{2\delta}\right\}\|\D  u\|_{\L^p(0,1;H)}\d s,
\end{align*}for $p\geq 2\delta+1$. In a similar fashion, one can estimate the term $\|\D \mathscr{A}_2u(t)\|_{\e,2p,H}$ as \begin{align*}
\|\D \mathscr{A}_2u(t)\|_{\frac{1+\e}{p},2p,H} \leq C\int_{0}^{t}(t-s)^{-\frac{3}{4}-\frac{2\delta+1}{4p}}n^\delta\|\D u\|_{\L^p(0,1;H)}\d s,
\end{align*}for $p>2\delta+1$. Using the above estimates and \eqref{lemma3.5.3}, one can deduce \eqref{6.07}.
	\end{proof}
		\begin{proof}[Proof of Proposition \ref{prop6.5}]
			Considering the operator $\mathscr{A}$ defined in \eqref{4.06} and setting $u^n=\mathscr{A}u^{n-1}$ for any $n\geq 1$ and $u^0=u_0$. We would like  to apply  Lemma 4.4, \cite{LZN} to the sequence $\{u^n\}_{n\geq 1}$. We have  already proved in Proposition \ref{theorem4.2} that the sequence $\{u^n\}_{n\geq 1}$ converges to the fixed point of the operator $\mathscr{A}$ in the Banach space $\mathcal{H}$ (defined in \eqref{4.07}) as limit $n\to\infty$, and this fixed point $u$ is the unique solution of \eqref{4.05}. This implies that  the sequence $\{u^n\}_{n\geq 1}$ converges to $u$ in $\L^p(\Omega\times[0,T]\times[0,1]),\; \text{for } p\geq 2\delta+1$.  Also, the process $u$ has continuous paths in $[0,T]\times[0,1]$.	Using Lemma \ref{lemma6.6} recursively to the process $u^n$, we obtain that the process $u^n(t,x)\in \Ds^{1,p},$ for all $n\geq 1$ and $(t,x)\in[0,T]\times[0,1]$.	Setting $$h_n(t)=\E\big[\|\D u^n(t)\|_{\L^p}^p\big]=\E\big[\|\D \mathscr{A}u^{n-1}(t)\|_{\L^p}^p\big],$$the estimate \eqref{6.05} gives for $p\geq 2\delta+1$: 
			\begin{align*}
				h_n(t)\leq C+C&\int_{0}^{t}\left\{(t-s)^{-\frac{\delta}{2p}}n^\delta(1+\gamma)(\delta+1)+\gamma +(t-s)^{-\frac{\delta}{p}}n^{2\delta}(2\delta+1)\right.\\&\quad\left.+(t-s)^{-\frac{1}{2}-\frac{\delta}{2p}}n^\delta(\delta+1)+(t-s)^{-\frac{1}{2}-\vartheta}\right\}h_{n-1}(s)\d s,
			\end{align*}for every $0\leq t\leq T$ and $n\geq 1$. An application of  Lemma 3.3 from \cite{JBW} (see pp. 316) 
		yields
	\begin{align*}
		\sup_n\sup_{t\in[0,T]}\E\left[\|\D u^n(t)\|_{\L^p(0,1;H)}^p\right]<\infty.
	\end{align*}Using the above inequality in \eqref{6.07}, we obtain for $p>\max\{6,2\delta+1\}$ and $0<\e<\frac{p}{2}-3,$
\begin{align}\label{68}
	\sup_{n}\sup_{t\in[0,T]}\E\left[\|\D u^n(t)\|_{\frac{1+\e}{p},p,H}^p\right]<\infty,
\end{align}and Lemma 4.4, \cite{LZN} 
holds for $u^n(\cdot)$ by an application of \eqref{{2.01}}. The derivative \eqref{6.04} follows form Lemma \ref{lemma6.6}, and the proof is completed. 
\end{proof}
\begin{theorem}\label{thrm6.9}
	Let $u(\cdot)$ be the solution of \eqref{4.02}. Then for any $p\geq 2\delta+1$ and $(t,x)\in[0,T]\times[0,1]$, the random variable $u(t,x)\in \Ds_{\loc}^{1,p}$. 
\end{theorem}
\begin{proof}
	Consider the solution $u^n(\cdot)$ of the equation \eqref{4.05} and the stopping time $\tau_n=\inf\{t\geq 0:\|u^n(t)\|_{\L^p}\geq n\}$ and let $\Omega_n=\{\tau_n=T\}$. Then, the sequence $(\Omega_n,u^n(t,x))$ localizes $u(t,x)\in\Ds^{1,p}$. Also, the processes $M_n(t,x)$  and $M_m(t,x)$ coincide on the set $\{s<\tau_n\}$ for each $m\geq n$.  
\end{proof}
\section{Absolute Continuity of the Law of the Solution}\label{sec7}\setcounter{equation}{0}
In this section, we prove  the absolute continuity of the law of the solution of \eqref{1.01}-\eqref{1.1c} with respect to the Lebesgue measure on $\mathbb{R}$ and hence the existence of density also.
\begin{theorem}\label{theorem7.1}
	Let us denote the solution of \eqref{4.02} by $u(\cdot)$. Assume that the noise coefficient  $g(t,x,r)$ is continuous and satisfy $g(0,y,u_0(y))\neq 0,$ for some $y\in(0,1)$. Then the law of the solution $u(t,x)$ is absolutely continuous with respect to the Lebesgue measure on $\mathbb{R}$ for each $(t,x)\in(0,T]\times(0,1)$. 
\end{theorem} 
\begin{proof}
The proof is quite similar to Theorem 5.1, \cite{LZN}. We are providing only the main outlines of the proof.  Here, we are defining the process $v(s,x)  = \int_{a'}^{b'}\D_{r,z}u(s,x)\d z,$ for $s\geq r$ (where $r\in[0,T]$), and $[a',b']\subset(a,b)$, which is different from Theorem 5.1, \cite{LZN}.
Then the process $\{v(s,x):s\in[r,T]\}$ is the unique solution of the following stochastic integral equation:
\begin{align*}
	v(s,x)&=\int_{a'}^{b'}G(s-r,x,z)g(r,z,\pi_nu(r,z))\d z \\&\quad+\beta \int_{r}^{s}\int_{0}^{1}G(s-\theta,x,y)\big((1+\gamma)(\delta+1)(\pi_nu)^\delta-\gamma-(2\delta+1)(\pi_nu)^{2\delta}\big)v(\theta,y)\d y\d \theta\\&\quad+\alpha \int_{r}^{s}\int_{0}^{1}\frac{\partial G}{\partial y}(s-\theta,x,y)(\pi_nu)^\delta v(\theta,y)\d y \d \theta\\&\quad + \int_{r}^{s}\int_{0}^{1}G(s-\theta,x,y)M_n(\theta,y)v(\theta,y)\W(\d \theta, \d y),
\end{align*}where $M_n$ is an adapted process bounded by the Lipschitz constant $L$.  
Using the following fact we obtain $v(t,x)>0,\; \P$-a.s., on the set $\Lambda=\{r\leq \ttt\}$.
\vskip 0.2 cm 
\noindent 
{\bf Fact:} For any $\rho>0$, there exists a natural number $m_0\in\N$ such that for all $m\geq m_0$ and for $1\leq k\leq m-1$, we have
\begin{align}\label{7.03}
	\P\left\{E_{k+1}^c\cap\Lambda|E_1\cap\cdots\cap E_k\right\} \leq \frac{\rho}{m}.
\end{align}From the above fact, we obtain 
\begin{align*}
	\P\big(\{v(t,x)>0\}\cap\Lambda   \big) &\geq \P\big(\{v(t,y)\geq \ee \kappa^m\chi_{[a',b'+\frac{kd}{m}]}(y), \text{ for all } y\in[0,1]\}\cap \Lambda\big)\\& =\P(E_m)\\& \geq \P\big(E_1\cap\cdots\cap E_m\big) \\&=\P\big(E_m|E_{m-1}\cap\cdots\cap E_1    \big)\P\big(E_{m-1}|E_{m-2}\cap\cdots\cap E_1\big)\cdots \P(E_1)
	\\& \geq \bigg(1 -\frac{\rho}{m}\bigg)^m \\& \geq 1-\rho,
\end{align*}where  the choice of $\rho$ is arbitrary. 
\end{proof}Let us now prove the fact \eqref{7.03}. 
\begin{proof}[Proof of \eqref{7.03}]
	For all $s\in [r_k^m,r_{k+1}^m]$, we have
	\begin{align*}
		v(s,x) &= G(s-r_k^m,x;v(r_k^m,\cdot))\\&\quad+\beta \int^{s}_{r_k^m}\int_{0}^{1}G(s-\theta,x,y)\big((1+\gamma)(\delta+1)(\pi_nu)^\delta-\gamma-(2\delta+1)(\pi_nu)^{2\delta}\big)v(\theta,y)\d y\d \theta\\&\quad+\alpha \int^{s}_{r_k^m}\int_{0}^{1}\frac{\partial G}{\partial y}(s-\theta,x,y)(\pi_nu)^\delta v(\theta,y)\d y \d \theta\\&\quad + \int_{r_k^m}^s\int_{0}^{1}G(s-\theta,x,y)M_n(\theta,y)v(\theta,y)\W(\d \theta ,\d y).
	\end{align*}Using comparison theorem (see  Theorem \ref{theorem5.2}),  on the set $E_1\cap\cdots\cap E_k,$ for $1\leq k\leq m-1$, we obtain 
\begin{align}\label{7.04}
	v(s,y)\geq w(s,y)\geq 0, \; \P\text{-a.s.,}
\end{align}where the process $w=\{w(s,y):(s,y)\in[r_k^m,r_{k+1}^m]\}$ is the solution of the stochastic integral equation:
\begin{align*}
	w(s,x) &= G(s-r_k^m,x;\ee\kappa^k \chi_{\big[a',b'+\frac{kd}{m}\big]})\\&\quad+\beta \int_{r_k^m}^{s}\int_{0}^{1}G(s-\theta,x,y)\big((1+\gamma)(\delta+1)(\pi_nu)^\delta-\gamma-(2\delta+1)(\pi_nu)^{2\delta}\big)w(\theta,y)\d y\d \theta\\&\quad+\alpha \int_{r_k^m}^{s}\int_{0}^{1}\frac{\partial G}{\partial y}(s-\theta,x,y)(\pi_nu)^\delta w(\theta,y)\d y \d \theta\\&\quad + \int_{r_k^m}^{s}\int_{0}^{1}G(s-\theta,x,y)M_n(\theta,y)w(\theta,y)\W(\d \theta ,\d y).
\end{align*}The rest of the proof follows in the similar lines as in the proof of Theorem 5.2, \cite{LZN}.
\end{proof}
\begin{remark}
	By an application of the Radon-Nikodym theorem, one can easily obtain the existence of density for $u(t,x)$  from Theorem \ref{theorem7.1}.
\end{remark}
As established in the proof of Theorem 1.1, \cite{EPTZ}, one can deduce the following result as an immediate consequence of Theorem \ref{theorem7.1}.
\begin{theorem}\label{thm7.3}
	Let  $(t,x)\in(0,T]\times(0,1)$. Then the law of the  random variable $u(t,x)$ is absolutely continuous with respect to the Lebesgue measure on $\mathbb{R}$  if and only if there exists $s\in[0,t)$ such that $g(s,\cdot,u(s,\cdot))\not\equiv 0$.  
\end{theorem}

	\medskip\noindent
{\bf Acknowledgments:} The first author would like to thank Ministry of Education, Government of India - MHRD for financial assistance. M. T. Mohan would  like to thank the Department of Science and Technology (DST), India for Innovation in Science Pursuit for Inspired Research (INSPIRE) Faculty Award (IFA17-MA110).  The authors sincerely would like to thank the reviewers for their valuable comments and suggestions, which helped us to improve the manuscript significantly, especially Proposition \ref{lem3.4}.

\medskip\noindent
{\bf Data availability:} 
Data sharing not applicable to this article as no datasets were generated or analysed during the current study.

\medskip\noindent	{\bf Deceleration:} 	The author has no competing interests to declare that are relevant to the content of this article.

\end{document}